\documentclass[preprint,10pt]{elsarticle}
\usepackage{bibentry}
\usepackage{amsfonts}
\usepackage{amssymb}
\usepackage{amsthm}
\usepackage[namelimits]{amsmath}
\usepackage{graphicx,epsfig}
\usepackage{amscd}
\usepackage{epsfig,color,psfrag}
\allowdisplaybreaks
\date{}

\newcommand{\de}{\delta}
\def\De{\Delta}
\newcommand{\z}{\zeta}
\newcommand{\ep}{\epsilon}

\newcommand{\ga}{\gamma}
\newcommand{\Om}{\Omega}
\newcommand{\om}{\omega}

\newcommand{\tht}{\theta}
\newcommand{\al}{\alpha}
\newcommand{\be}{\beta}

\newcommand{\si}{\sigma}
\newcommand{\ld}{\lambda}
\newcommand{\Ld}{\Lambda}
\newcommand{\lan}{\langle}
\newcommand{\ran}{\rangle}
\newcommand{\NN}{\mathbb{N}}
\newcommand{\ZZ}{\mathbb{Z}}
\newcommand{\RR}{\mathbb{R}}

\newcommand{\ba}{\begin{array}}
\newcommand{\ea}{\end{array}}
\newcommand{\beq}{\begin{equation}}
\newcommand{\eeq}{\end{equation}}
\newcommand{\beqq}{\begin{equation*}}
\newcommand{\eeqq}{\end{equation*}}
\newcommand{\bdm}{\begin{displaymath}}
\newcommand{\edm}{\end{displaymath}}

\theoremstyle{definition}
\newtheorem{theorem}{Theorem}[section]
\newtheorem{definition}{Definition}[section]
\newtheorem{lemma}{Lemma}[section]
\newtheorem{corollary}{Corollary}[section]
\newtheorem{proposition}{Proposition}[section]
\newtheorem{remark}{Remark}[section]

\numberwithin{equation}{section}
\begin{document}

\pagestyle{plain}

\begin{center}
{\large \bf Multibump solutions of a class of second-order discrete Hamiltonian systems}
\\ [0.3in]

XU ZHANG \footnote{ Email addresses:\ \ xuzhang08@gmail.com.}\

 \vspace{0.15in}
{\it  Department of Mathematics, Shandong University \\
Jinan, Shandong 250100, P.~R. China}

\end{center}


\baselineskip=18pt

{\noindent\bf Abstract.}\ For a class of second-order discrete Hamiltonian systems $\De^2x(t-1)-L(t)x(t)+V'_x(t,x(t))=0$, we investigate the existence of homoclinic orbits by applying variational method, where $L$ and $V(\cdot,x)$ are periodic functions. Further, we show that there exist either uncountable many homoclinic orbits or multibump solutions under certain conditions.



\section{Introduction}

After Poincar\'{e}, Birkhoff, and Smale's work on celecical mechanics and general dynamical systems theory, we know that if a system has a transversal
homoclinic orbit, then it has complicated dynamical behavior and it is chaotic in the sense of Li-Yorke or Devaney.
So, the research of the existence of homoclinic orbits plays a very important role in the understanding of the complicated dynamical behavior,
For Hamiltonian systems, thanks to the variational structure, the variational methods have contributed greatly to the investigation of the existence of homoclinic orbits.  Please refer to \cite{Rabinowitz, Str} for more interesting materials about variational approaches. Even though, it is still difficult to verify the transversal condition of the stable and unstable manifolds, that is, it is hard to show that the homolcinic orbit is transversal. A kind of relatively weaker solution ``multibump solution" is introduced to study the complicated dynamics of Hamiltonian systems.

In the research of the existence of homoclinic orbits of Hamiltonian systems, Rabinowitz obtained the existence of a homoclinic orbit of a type of second-order differential equations, where the homoclinic orbit is the limit of a series of periodic orbits \cite{Rabinowitz1990}. Coti-Zelati and Rabinowitz studied the existence of homoclinic orbits of second-order Hamiltonian systems under the superquadratic potential assumptions at both the origin and at the infinity \cite{ZelatiRabinowitz}. In \cite{Sere}, S\'{e}r\'{e} depicted the complicated dynamical behavior combined with the Bernoulli shift by introducing the multibump solutions.

For the research of the existence of heteroclinic orbits of Hamiltonian systems, Felmer studied the existence of heteroclinic orbits of a class of one-order Hamiltonian system under the periodic assumption of one spatial variable and superlinear in another one \cite{Felm}. Bertotti and Montecchiari studied a class of second-order almost periodic system, showing that there are infinitely many heteroclinic solutions connecting degenerate equilibria \cite{BerMon}. Caldiroli and Jeanjean proved the existence of a heteroclinic orbit connecting the origin and a minimal non-contractible periodic orbit under certain conditions \cite{CalJea}. Coti-Zelati and Rabinowitz investigated the existence of heteroclinic orbits connecting two stationary points at different energy levels and studied the multibump solutions \cite{ZelatiRabinowitzh}.

For the research of dynamical systems, the method of modeling and simulation of a real system plays an important role. Since the continuous systems can not be directly applied in the real computation, we need to transform the continuous equations into their corresponding difference equations such that we could observe the dynamical behavior of the systems through simulations. Hence, the research of difference equations gradually becomes important and useful. Discrete Hamiltonian system can be applied in many different areas, such as physics, chemistry, and so on. Please refer to \cite{AhlPet} for more information on discrete Hamiltonian systems. There are many good results on existence of periodic, homoclinic, and heteroclinic orbits of discrete Hamiltonian systems \cite{GuoYu2003a, GuoYu2003c, LinTang, MaGuo1, MaGuo2, YuBinGuo, ZhangLi}. Deng et al. investigated the existence of homoclinic orbits of a kind of discrete Hamiltonian systems with potential changing sign \cite{DengChen}. Lin and Tang obtained that there exist infinitely many homoclinic orbits of a class of difference equations the equation with general assumptions on the potential function \cite{LinTang}. Ma and Guo studied the existence of homoclinic orbits of a type of difference equations provided that the potential function grows superlinearly both at origin and at infinity or the potential function is odd with respect to the spatial variables \cite{MaGuo1}. Ma and Guo proved the existence of homoclinic orbits of scalar difference equations with under the periodicity assumptions on the potential functions and other assumptions \cite{MaGuo2}.

For the multibump solutions of certain systems, S\'{e}r\'{e}'s firstly investigated multibump solutions of a kind of first-order Hamiltonian system under certain conditions \cite{Sere}. Later, similar construction of the orbits is obtained in different situations. For example, systems in the degenerate case \cite{Rabinowitz1997}, the damped systems \cite{Bess1}, systems with the potential changing sign \cite{CalMon}. However, the construction of multibump solutions in difference equations is just beginning, there is few work about the existence of multibump solutions of homoclinic orbits of difference equations.

The procedure to show the existence of multibump solutions is to first use a variational argument,
minimax method, to find a special family of non-trivial homoclinic orbits
which can be regarded as the ``one-bump" solutions.
Secondly, variational arguments are applied to get multibump solutions, i.e. solutions near sums of sufficiently
separated translates of the ``one-bump" solutions. A key assumption for the existence of multibump solutions is that the critical points of the corresponding
functional are isolated. This hypothesis is used to replace
the classical transversality conditions, which are applied to the study of the complicated dynamical behavior of
dynamical systems.

In 1994, Caldiroli and  Montecchiari  studied the existence of multibump solutions of the  following type of second-order differential equation
\begin{equation}\label{Ch2e1}
\ddot{x}-L(t)x+V'_x(t,x)=0,\ \ t\in\RR,
\end{equation}
where $x\in\RR^n$, $L(t)$ is a periodic,  $n\times n$ real  positive definite matrix, $V\in C^1(\RR\times\RR^n,\RR)$, and $V(t,x)$ has changing sign in $x$ with $V'_x(t,0)\equiv0$. They obtained that equation has either uncountable many homoclinic orbits or multibump solutions  under certain conditions \cite{CalMon}. Inspired by  Caldiroli and  Montecchiari's work, we investigate the existence of multibump solutions of the corresponding difference equation
\begin{equation}\label{Ch2e2}
\De^2x(t-1)-L(t)x(t)+V'_x(t,x(t))=0,\ \ t\in\ZZ,
\end{equation}
where $\De x(t-1)=x(t)-x(t-1)$, $\De^2x(t-1)=\De(\De x(t-1))$, $L(t)$ is a periodic, $n\times n$ positive definite matrix for each $t\in\ZZ$, and $V(t,\cdot)\in C^1(\RR^n,\RR)$ for each $t\in\ZZ$ with $V'_x(t,0)\equiv0$. Our results can be regarded as a discrete analog of Caldiroli and Montecchiari 's results obtained in \cite{CalMon}.

Assume that the second-order difference equations \eqref{Ch2e2} satisfying the following assumptions:
\begin{itemize}
\item [(A1)]  there is a positive integer $T$ such that $L(t+T)=L(t)$ and $V(t+T,\cdot)=V(t,\cdot)$ for any $t\in\ZZ$;
\item [(A2)] $L(t)$ is symmetric, positive definite, for any $t\in\ZZ$;
\item [(A3)] $V(t,0)=0$ and $|V'_{x}(t,x)|=\circ(|x|)$ as $x\to0$;
\item [(A4)] there are two constants $\be>2$ and $0<\al<\be/2-1$ such that
$$\be V(t,x)-\lan V'_x(t,x),x\ran\leq\al\lan x,L(t)x\ran,\ \ \forall (t,x)\in\ZZ\times\RR;$$
\item [(A5)] there exists $(t_0,x_0)\in\ZZ\times\RR^n$ with $x_0\neq0$ such that $V(t_0,x_0)-\lan x_0,L(t_0)x_0\ran/2>0.$
\end{itemize}

The rest is organized as follows. In Section 2, we introduce some basic definitions and lemmas.
In Section 3, the existence of homoclinic orbits of the equations \eqref{Ch2e2} is investigated.
In Section 4, the properties of the Palais-Smale sequences are studied.
In Section 5, we investigate some properties of the functionals by applying the construction of certain vector fields, which are useful in the investigation of the existence of homoclinic orbits.
In Section 6, the main theorem is shown. We show that there are either uncountable many homoclinic orbits of \eqref{Ch2e2} or multibump solutions under certain conditions.
In Section 7, the proof of an important lemma is given.

\section{Preliminaries}

In this section, we introduce some basic notations and results.

For any $u,v\in\RR^n$, denote $\lan u,v\ran$ as the usual inner product. Let
$$l^p:=\bigg\{u=\{u(t)\}^{+\infty}_{t=-\infty}\subset\RR^n:\   \sum_{t\in\ZZ}|u(t)|^p<\infty\bigg\},\ 1\leq p<\infty,$$ and the norm is given by
$$\|u\|_p:=\bigg\{\sum_{t\in\ZZ}|u(t)|^p\bigg\}^{1/p};$$
$$l^{\infty}:=\bigg\{u=\{u(t)\}^{+\infty}_{t=-\infty}\subset\RR^n:\   \sup_{t\in\ZZ}|u(t)|<\infty\bigg\},$$ and the norm is defined by
$$\|u\|_{\infty}:=\sup_{t\in\ZZ}|u(t)|.$$
For convenience, denote
$$\lan u,v\ran_2:=\sum_{t\in\ZZ}\lan u(t),v(t)\ran,\ \ u,v\in l^2.$$

Next, we introduce several concepts about minimax systems and a related result \cite{Sch2008}.

\begin{definition}
 Let $E$ be a Banach space. A map $\varphi: E\to E$ is said to be of class $\Lambda$ if it is a homeomorphism onto $E$, and both $\varphi$ and $\varphi^{-1}$ are bounded on bounded sets.
\end{definition}

For any $A\subset E$, set
$$\Lambda(A):=\{\varphi\in\Lambda:\ \varphi|_A=id\}.$$

\begin{definition}
For a nonempty set $A\subset E$, a nonempty collection $\mathcal{K}$, consisting of sets of $E$, is said to be a minimax system for $A$ if it has the following property:
$$\varphi(K)\in\mathcal{K},\ \forall\varphi\in\Lambda(A),\ \forall K\in\mathcal{K}.$$
\end{definition}

\begin{lemma}\label{Ch2pssequ}
Let $\mathcal{K}$ be a minimax system for a nonempty set $A$ of $E$, and let $J$ be a $C^1$ functional on $E$. Define
$$a:=\inf_{K\in\mathcal{K}}\sup_{u\in K}J(u),$$
and assume that $a$ is finite and satisfies
$$a>a_0:=\sup_{u\in A}J(u).$$
Let $\phi(t)$ be a positive, non-increasing, and locally continuous function on $[0,\infty)$ such that
$$\int^{\infty}_{0}\phi(r)dr=\infty.$$
Then, there is a sequence $\{u_k\}^{\infty}_{k=1}\subset E$ such that
$$\lim_{k\to\infty}J(u_k)=a,\ \ \lim_{k\to\infty}DJ(u_k)/\phi(\|u_k\|)=0.$$
\end{lemma}

Now, we introduce the concentration-compactness principle \cite[Page 39]{Str}.

\begin{lemma} \label{Ch2cc}
Suppose a sequence $\{\rho_k\}^{\infty}_{k=1}\subset l^{1}(\ZZ,\RR)$ with $\rho_k(t)\geq0$ and $|\rho_k|_1=1$ for each $k\in\NN$. Then there exists a subsequence of $\{\rho_k\}^{\infty}_{k=1}$ such that one of the following properties holds (without loss of generality, denoted by $\{\rho_k\}^{\infty}_{k=1}$ as the subsequence):
\begin{itemize}
\item [(a)] Vanishing: for any positive integer $N$,
$$\lim_{k\to+\infty}\sup_{p\in\ZZ}\sum^{p+N}_{t=p-N}\rho_k(t)=0;$$
\item [(b)] Concentration: for any $k\in\NN$, there is $p_k\in\ZZ$ such that for any $\ep>0$, there is a positive integer $N$ such that
$$\sum^{p_k+N}_{t=p_k-N}\rho_k(t)\geq1-\ep;$$
\item [(c)] Dichotomy: there exist $\{p_k\}^{\infty}_{k=1}\subset\ZZ$, $\{N^1_k\}^{\infty}_{k=1},\{N^2_k\}^{\infty}_{k=1}\subset\NN$, and $\eta\in(0,1)$, such that
 \begin{itemize}
 \item [(i)] $N^1_k\to+\infty$, $N^2_k\to+\infty$, $N^1_k/N^2_k\to0$ as $k\to+\infty$;
 \item [(ii)] $\lim\limits_{k\to+\infty}\sum\limits^{p_k+N^1_k}_{t=p_k-N^1_k}\rho_k(t)=\eta$, $\lim\limits_{k\to+\infty}\sum\limits^{p_k+N^2_k}_{t=p_k-N^2_k}\rho_k(t)=\eta$;
 \item [(iii)] for any $\ep>0$, there is a positive integer $N$ such that for any $k\in\NN$, $\sum\limits^{p_k+N}_{t=p_k-N}\rho_k(t)\geq\eta-\ep$.
\end{itemize}
\end{itemize}
\end{lemma}

\begin{lemma}\label{Ch2comp} (\cite{ZhangShi}, Lemma 2.4)
Let $p, q\in[1,+\infty]$ and $K$ be a bounded set of $l^{p}$. If for any $\ep>0$, there exists $N>0$ such that for any $u\in K$, one has $\sum_{|t|>N}|u(t)|^q<\ep$ in the case that $q<+\infty$, and $\sup_{|t|>N}|u(t)|<\ep$ in the case that $q=+\infty$, then $K$ is relatively compact in $l^{q}$.
\end{lemma}

Now, we introduce some basic notions and results about variational principles.

\begin{definition} \cite{Rabinowitz} Let $E$ be a real Banach space.  Given $J\in C^1(E,\mathbf{R})$, a sequence $\{x_m\}\subset E$ is called a Palais-Smale sequence (briefly, PS sequence) for $J$ if it satisfies the condition that $\{J(x_m)\}$ is bounded and $DJ(x_m)\to0$ as $m\to\infty$, where $DJ(x)$ means the Fr\'{e}chet derivative at $x\in E$. Further, we say $J$ satisfies the Palais-Smale condition (briefly, PS condition) if any PS sequence possesses a convergent subsequence.
\end{definition}

Now, we introduce the concept of local mountain pass type \cite{CalMon, Hofer, PucSer}.
\begin{definition}

Given a $C^1$ functional $f$ on a Banach space $E$. Assume $c$ is a real number, $\Om$ is a non-empty open subset of $E$, and
$x,y\in\Om$. The points $x$ and $y$ are called $c$-connectible in $\Om$ if there exists a continuous path $p:[0,1]\to\Om$ joining $x$ and
$y$ such that the range of $p$ is contained in $\Om$ and $\max_{t\in[0,1]}f(p(t))<c$. We call a critical point $v\in E$ for $f$ is said to
be local mountain pass type for $f$ on $\Om$ if $v\in\Om$ and for any neighborhood $U$ of $v$, the set $\{w:\ f(w)<f(v), \ w\in E\}\cap U$ contains
two points which are not $f(v)$-connectible in $\Om$.
\end{definition}

\begin{remark} \label{Ch2e4}
Since any complete convex metric space is path connected \cite[Thm. 14.1, p. 41]{Blum}, the definition above is reasonable in the space $l^2$.
\end{remark}

Now, we introduce a result about the fixed point theory \cite{Mir}.
\begin{lemma} \label{Ch2fix}
Let $\Om=[0,1]^k\subset\RR^k$ and $F:\Om\to\RR^k$ be a continuous function satisfying
for any $i$, $1\leq i\leq k$,
$$F_i(x_1,x_2,...,x_{i-1},0,x_{i+1},...,x_k)\geq0,\ F_i(x_1,x_2,...,x_{i-1},1,x_{i+1},...,x_k)\leq0,$$
then $F(x)=0$ has a solution in $\Om$, where $F=(F_1,...,F_k)$ and $x=(x_1,x_2,...,x_k)\in\RR^k$.
\end{lemma}

\section{Existence of homoclinic orbits}
In this section, we show the existence of homoclinic orbits of difference equations \eqref{Ch2e2} under the assumptions (A1)--(A5).

If (A4) holds, then we have the following inequality, which will be used later,
\begin{equation}\label{Ch2b1}
V(t,sx)\geq\bigg(V(t,x)-\frac{\al}{\be-2}\lan x,L(t)x\ran\bigg)s^{\be}+\frac{\al}{\be-2}s^2\lan x,L(t)x\ran,\ \forall t\in\ZZ,\ s\geq1,\ x\in\RR^n.
\end{equation}
Now, we give the proof of \eqref{Ch2b1}. Consider the function
$$\phi(s)=\frac{V(t,sx)}{s^{\be}}-V(t,x)+\frac{\al}{\be-2}\lan x,L(t)x\ran-s^{2-\be}\frac{\al}{\be-2}\lan x,L(t)x\ran.$$
And, it follows from (A4) that
\begin{equation*}
\begin{split}
\frac{d\phi(s)}{ds}=&\frac{\lan V'(t,sx),x\ran}{s^{\be}}-\be\frac{V(t,sx)}{s^{\be+1}}+\al s^{1-\be}\lan x,L(t)x\ran\\
=&\frac{\lan V'(t,sx),sx\ran-\be V(t,sx)+\al \lan sx,L(t)sx\ran}{s^{1+\be}}\geq0.
\end{split}
\end{equation*}
This, together with the fact $\phi(1)=0$, yields that \eqref{Ch2b1} holds.

For the Hilbert space $l^2$, by (A1) and (A2), there is an equivalent norm on $l^2$ which is defined by the following inner product
$$\lan u, v\ran_*:=\sum_{t\in\ZZ}\lan\De u(t-1),\De v(t-1)\ran+\sum_{t\in\ZZ}\lan u(t),L(t)v(t)\ran,$$
and the corresponding norm is given by
\beq\label{Ch2e8}
\|u\|=\lan u,u\ran_*^{1/2}.
\eeq
So, there is a positive constant $L_1$ such that
\beq\label{Ch2b2}
L^{-1}_1\|u\|_*\leq\|u\|_2\leq L_1\|u\|_*,\ \forall u\in l^2.
\eeq

Given any set $F\subset\ZZ$, denote $F^c:=\{t:\ t\in \ZZ\setminus F\}$. Consider the characteristic function of $F$
\beq\label{Ch2f11}
\chi_F(t)=\left\{
    \begin{array}{ll}
      1, & \mbox{if}\ t\in F; \\
      0, & \mbox{if}\ t\in F^c.
    \end{array}
  \right.
\eeq
Set
$$F_L:=\{t\in F:\ t-1\in F^c\},\  F_R:=\{t\in F:\ t+1\in F^c\},$$
$$F^c_L:=\{t\in F^c:\ t-1\in F\},\ F^c_R:=\{t\in F^c:\ t+1\in F\}.$$
It is evident that $F_R=F^c_L$. So, by direct calculation, one has
{\allowdisplaybreaks\beq\label{Ch2f5}
\begin{split}
&\|\chi_Fu\|^2_*=\sum_{t\in F}\lan u(t),L(t)u(t)\ran+\sum_{t\in F}\lan\De(\chi_F(t-1)u(t-1)),\De(\chi_F(t-1)u(t-1))\ran\\
&+\sum_{t\in F^c}\lan \De(\chi_F(t-1)u(t-1)),\De(\chi_F(t-1)u(t-1))\ran\\
=&\sum_{t\in F}\lan u(t),L(t)u(t)\ran+\bigg(\sum_{t\in F_L}\lan u(t),u(t)\ran+\sum_{t\in F_R\setminus F_L}\lan \De u(t-1),\De u(t-1)\ran\\
+&\sum_{t\in F\setminus(F_L\cup F_R)}\lan \De u(t-1),\De u(t-1)\ran\bigg)+\bigg(\sum_{t\in F_R}\lan u(t),u(t)\ran\bigg).
\end{split}
\eeq}
Please note that $\sum_{t\in F_L\cup F_R}\lan u(t),u(t)\ran\neq\sum_{t\in F_L}\lan u(t),u(t)\ran+\sum_{t\in F_R}\lan u(t),u(t)\ran$ in general.
\medskip

Set
\beq\label{Ch2f20}
L_2:=\inf\{\lan x,L(t)x\ran:\ t\in\ZZ,|x|=1\}.
\eeq

Now, we could define a functional $f:l^2\to\RR$, where
\begin{equation}\label{Ch2c2}
f(u)=\frac{1}{2}\|u\|_*^2-g(u)=\frac{1}{2}\|u\|_*^2-\sum_{t\in\ZZ}V(t,u(t)),\ \ u\in l^2.
\end{equation}

We first show that $g(u)$ is well defined.

By (A3), one has that $V'_x(t,0)\equiv0$ for $t\in\ZZ$. So, there exists $\de>0$ such that for any $|x|\leq\de$,
\begin{equation}\label{Ch2c3}
V(t,x)\leq|x|^2.
\end{equation}
For any given $u\in l^2$, there exists an integer $N>0$ such that $|u(t)|\leq\de$ for any $|t|> N$. This, together with \eqref{Ch2c3}, implies that
\begin{equation}\label{Ch2c4}
g(u)\leq\sum_{|t|\leq N}V(t,u(t))+\sum_{|t|>N}|u(t)|^2<\infty.
\end{equation}
So, $g(u)$ is well defined on $l^2$ if (A3) holds.

For $f\in C^1(l^2,\RR)$, by $D_*f(x)$ denote the Fr\'{e}chet derivative of $f$ at $x$ in $(l^2,\lan\cdot,\cdot\ran_*)$, by
$D_2f(x)$ denote the Fr\'{e}chet derivative of $f$ at $x$ in $(l^2,\lan\cdot,\cdot\ran_2)$. By the methods used in
\cite{ZelatiRabinowitz, ZhangShi}, we could show that $f(u)$ is differentiable.

\begin{lemma}\label{Ch2diff1}
If (A1)-(A3) hold, the functional $f$ is differentiable, and for any $u,v\in l^2$,
\beq \label{Ch2deriva}
 D_*f(u)v=\lan u,v\ran_{*}-\sum_{t\in\ZZ}\lan V'_x(t,u(t)),v(t)\ran;
\eeq
\beq\label{Ch2k43}
 D_2f(u)v=\lan u,v\ran_{*}-\sum_{t\in\ZZ}\lan V'_x(t,u(t)),v(t)\ran.
\eeq
\end{lemma}

\begin{proof}
By the method used in the proof of Lemma 3.2 in \cite{ZhangShi} and \eqref{Ch2b2}, we could show that $J(u)\in C^1(l^2,\RR)$. It only needs to show that the first part of $f(u)$ is differentiable. By direct calculation, one has
\beqq
\begin{split}
&\frac{1}{2}\|u+v\|_{*}^2-\frac{1}{2}\|u\|_{*}^2=\lan u,v\ran_*+\lan v,v\ran_*=\lan u,v\ran_*+o(\|v\|_*)=\lan u,v\ran_*+o(\|v\|_2).
\end{split}
\eeqq
where \eqref{Ch2b2} is used. Next, it is to show that the derivative operator is continuous, or it is bounded. It follows from Cauthy inequality and \eqref{Ch2b2} that
\beqq
\lan u,v\ran_*\leq\|u\|_*\|v\|_*\leq L^{-1}_1\|u\|_*\|v\|_2.
\eeqq
This completes the proof.
\end{proof}

\begin{lemma}\label{Ch2smal}
For any $u\in l^2$, $g(u)=o(\|u\|_*^2)$ as $u\to0$ in $l^2$.
\end{lemma}

\begin{proof}
By (A1)-(A2), $L(t)$ is periodic and positive definite, and \eqref{Ch2b2},  $\|u\|_{\infty}\leq L_1\|u\|_*$ for any $u\in l^2$. It follows from (A3) that
$$V(t,x)=o(|x|^2)\ \ \mbox{as}\ x\to0.$$
So,
$$g(u)=\sum_{t\in\ZZ}V(t,u(t))=o\bigg(\sum_{t\in\ZZ}|u(t)|^2\bigg)=o(\|u\|_*^2)\ \mbox{as}\ \|u\|_*\to0.$$
\end{proof}
Since $f(0)=0$ and Lemma \ref{Ch2smal}, one has the following result.

\begin{lemma} \label{Ch2posi}
There is a sufficiently small $r>0$ such that for $f(u)\geq r^2/4$ if $\|u\|_*=r$. And, $f(u)\geq0$ for any $u\in B_{r}(0)$.
\end{lemma}

\begin{lemma}\label{Ch2nega}
There exists $w\in l^2$ such that $f(w)<0$.
\end{lemma}

\begin{proof}
It follows from (A4) and (A5) that $\al/(\be-2)<1/2$, set
$$\de_0:=V(t_0,x_0)-\frac{\al}{\be-2}\lan x_0,L(t_0)x_0\ran>0.$$
Set
$$\hat{u}(t):=\left\{
   \begin{array}{ll}
     x_0, & t=t_0; \\
     0, & t\neq t_0.
   \end{array}
 \right.$$
So,
$f(\ld \hat{u})=\ld^2|x_0|^2+\frac{1}{2}\ld^2\lan x_0,L(t_0)x_0\ran-V(t_0,\ld x_0),$ where $\ld>1$.
By \eqref{Ch2b1}, one has
$$V(t_0,\ld u_0)\geq\bigg(V(t_0,x_0)-\frac{\al}{\be-2}\lan x_0,L(t_0)x_0\ran\bigg)\ld^{\be}+\frac{\al}{\be-2}\ld^2\lan x_0,L(t_0)x_0\ran.$$
So,
$$f(\ld \hat{u})\leq-\de_0\ld^{\be}+\ld^2|x_0|^2+\frac{1}{2}\ld^2\lan u_0,L(t_0)u_0\ran-\frac{\al}{\be-2}\ld^2\lan x_0,L(t_0)x_0\ran.$$
Hence,
$$f(\ld\hat{u})\to-\infty\ \mbox{as}\ {\ld\to\infty}.$$
Therefore, there is $w\in l^2$ such that $f(w)<0$.
\end{proof}

\begin{lemma}\label{Ch2psbou}
If $\{u_k\}^{\infty}_{k=1}\subset l^2$ is a PS sequence, then $\{u_k\}^{\infty}_{k=1}$ is bounded in $l^2$ and $\liminf f(u_k)\geq0$.
\end{lemma}

\begin{proof}
For any given $u\in l^2$, it follows from (A4), \eqref{Ch2c2}, and Lemma \ref{Ch2diff1} that
\begin{equation} \label{Ch2c5}
(\frac{1}{2}-\frac{1}{\be})\|u\|_*^2=f(u)+\sum_{t\in\ZZ}V(t,u(t))-\frac{1}{\be} D_*f(u)u-\frac{1}{\be}\sum_{t\in\ZZ}\lan V'_x(t,u(t)),u(t)\ran.
\end{equation}
By (A4), one has
\begin{equation}\label{Ch2c6}
\begin{split}
&f(u)+\sum_{t\in\ZZ}V(t,u(t))-\frac{1}{\be} D_*f(u)u-\frac{1}{\be}\sum_{t\in\ZZ}\lan V'_x(t,u(t)),u(t)\ran\\
\leq& f(u)+\frac{1}{\be}\|D_*f(u)\|_*\|u\|_*+\frac{\al}{\be}\sum_{t\in\ZZ}\lan u(t),L(t)u(t)\ran.
\end{split}
\end{equation}
By \eqref{Ch2c5} and \eqref{Ch2c6}, one has
\begin{equation} \label{Ch2c7}
(\frac{1}{2}-\frac{1}{\be}-\frac{\al}{\be})\|u\|_*^2-\frac{1}{\be}\|D_*f(u)\|_*\|u\|_*\leq f(u).
\end{equation}
Since $1/2-1/\be-\al/\be>0$ by (A4), it follows from \eqref{Ch2c7} that any PS sequence $\{u_k\}^{\infty}_{k=1}$ is bounded. This, together with \eqref{Ch2c7}, implies that $\liminf f(u_k)\geq0$ for any PS sequence $\{u_k\}^{\infty}_{k=1}$.
\end{proof}

\begin{lemma} \label{Ch2crit}
There is a sequence $\{u_k\}^{\infty}_{k=1}\subset l^{2}$ satisfying that
the sequence $\{f(u_k)\}^{\infty}_{k=1}$ is convergent and $\lim_{k\to\infty}D_*f(u_k)=0$ in $l^{2}$. That is, $\{u_k\}^{\infty}_{k=1}\subset l^{2}$ is a PS sequence.
\end{lemma}

\begin{proof}
Lemma \ref{Ch2pssequ} is used in the following discussions. Set $A:=\{0,w\}$, where $w$ is given in Lemma \ref{Ch2nega}. Denote
\beq\label{Ch2path1}
\mathcal{K}:=\{\ga\in C([0,1],l^{2}):\ \ga(0)=0\ \mbox{and}\ \ga(1)=w\}.
\eeq
Since $l^2$ is a Banach space, it can be easily verified that $\mathcal{K}$ is a minimax system for $A$. Set
\begin{equation}\label{Ch2c31}
\kappa:=\inf_{\ga\in\mathcal{K}}\sup_{t\in[0,1]}f(\ga(t)).
\end{equation}
It follows from Lemma \ref{Ch2posi} that $\kappa\geq r^2/4>0$, where $r$ is specified in Lemma \ref{Ch2posi}. By Lemma \ref{Ch2nega},
$$\kappa>0=\max\{f(0),\ f(w)\}.$$
Denote
$$\phi(x):=1,\ x\in[0,\infty).$$
By Lemma \ref{Ch2pssequ}, one has that there is a sequence $\{u_k\}^{\infty}_{k=1}\subset l^{2}$ such that
\begin{equation}\label{Ch2e1}
\lim_{k\to\infty}f(u_k)=\kappa\ \mbox{and}\ \lim_{k\to\infty}D_*f(u_k)=0.
\end{equation}
Hence, $\{u_k\}^{\infty}_{k=1}\subset l^{2}$ is a PS sequence. This completes the proof.
\end{proof}

\begin{lemma} \label{Ch2critical}
There is a nonzero critical point for the functional $f$.
\end{lemma}

\begin{proof}
By Lemma \ref{Ch2crit}, there exists a sequence $\{u_k\}^{\infty}_{k=1}\subset l^2$ such that \eqref{Ch2e1} holds.

For any $k\geq1$, there is $t'_k\in\ZZ$ such that $\max_{t\in\ZZ}|u_k(t)|=|u_k(t'_k)|$. For this $t'_k$, there is $t_k$ such that $|t_kT-t'_k|<T$, where $T$ is specified in (A1). Set
$$v_k:=u_k(\cdot-t_kT).$$
By the periodicity of $V$ and $L$, \eqref{Ch2c2} and Lemma \ref{Ch2diff1}, one has  $\{v_k\}^{\infty}_{k=1}\subset l^2$ is a PS sequence such that
$$\lim_{k\to\infty}D_*f(v_k)=0\  \mbox{in}\ l^2,\ \lim_{k\to\infty}f(v_k)=\kappa>0.$$

It follows from Lemma \ref{Ch2psbou} that $\{v_k\}^{\infty}_{k=1}$ is bounded in $l^2$ and there exists a subsequence of $\{v_k\}^{\infty}_{k=1}$
which converges to some $v\in l^2$ weakly in $l^2$ and uniformly on compact subsets of $\ZZ$. Without loss of generality, suppose that this subsequence is $\{v_k\}^{\infty}_{k=1}$.

Now, it is to show that $D_*f(v)=0$. Take a dense subset of $l^2$,
$$W:=\{w\in l^2:\ \mbox{there are only finite number of}\ t\ \mbox{such that}\ |w(t)|\neq0\}.$$
For any $w\in W$, one has
$$ D_*f(v)w- D_*f(v_k)w=\lan v-v_k,w\ran_*-\sum_{t\in\ZZ}\lan(V'(t,v(t))-V'(t,v_k(t))),w(t)\ran\to0\ \mbox{as}\ k\to\infty,$$
which implies that
$$ D_*f(v)w=0.$$
By the density of $W$ in $l^2$, one has $D_*f(v)=0$.

Finally, it is to show that $v\neq0$.
By contradiction, suppose $v=0$. It follows from the definition and properties of the sequence $\{v_k\}^{\infty}_{k=1}$ that  $\|u_k\|_{\infty}=\|v_k\|_{\infty}\to0$ as $k\to\infty$. It follows from (A3) that for any $\ep>0$, there exists a positive integer $N$ such that for any $k\geq N$, one has
\beq\label{Ch2c8}
|V(t,u_k(t))|\leq\ep|u_k(t)|^2\ \mbox{and}\ |\lan V'(t,u_k(t)),u_k(t)\ran|\leq\ep|u_k(t)|^2,\ \forall t\in\ZZ.
\eeq
By \eqref{Ch2c2} and \eqref{Ch2deriva}, one has
\beq\label{Ch2c19}
f(u_k)=\frac{1}{2} D_*f(u_k)u_k+\frac{1}{2}\sum_{t\in\ZZ}\lan V'_x(t,u_k(t)),u_k(t)\ran-\sum_{t\in\ZZ}V(t,u_k(t)).
\eeq
It follows from \eqref{Ch2c8} that the left hand side of this equation will go to zero as $k$ goes to infinity.
This contradicts the fact that $\lim_{k\to\infty}f(u_k)=\kappa>0$.

This completes the whole proof.
\end{proof}

It follows from Lemma \ref{Ch2diff1} that a nonzero critical point of $f$ is a nontrivial homoclinic orbit of the difference equation \eqref{Ch2e2}.

\begin{corollary}
There exists a nontrivial homoclinic orbit of the difference equation \eqref{Ch2e2}.
\end{corollary}

\begin{lemma} \label{Ch2finite}
Under the assumptions above, one has $\inf_{v\in\mathcal{C}}\|v\|_*>0$ and $\inf_{v\in\mathcal{C}}f(v)>0$.
\end{lemma}

\begin{proof}
By \eqref{Ch2deriva}, for any $v\in l^2$,\ $\lan v,v\ran_{*}=\sum_{t\in\ZZ}\lan V'_x(t,v(t)),v(t)\ran$. By contradiction, suppose $\inf_{v\in\mathcal{C}}\|v\|_{*}=0$. There is a sequence $\{v_k\}^{\infty}_{k=1}\subset \mathcal{C}$ such that $\lim_{k\to\infty}\|v_k\|_{*}=0$, which yields that $\lim_{k\to\infty}\|v_k\|_{\infty}=0$, where \eqref{Ch2b2} and $\|v_k\|_{\infty}\leq\|v_k\|_2$ are used. By (A3), for any $\eta>0$, there is a sufficiently large $k$ such that $\sum_{t\in\ZZ}|\lan V'_x(t,v(t)),v(t)\ran|\leq \eta\sum_{t\in\ZZ}|v_k(t)|^2$. Suppose $\eta<1/L_1^2$. This, together with \eqref{Ch2b2},
\eqref{Ch2deriva}, and
\beqq
\lan v_k,v_k\ran_{*}=\sum_{t\in\ZZ}\lan V'_x(t,v_k(t)),v_k(t)\ran,
\eeqq
implies that
\beqq
\|v_k\|_{*}^2\leq\eta\|v_k\|^2_2\leq\eta L^2_1\|v_k\|_{*}^2<\|v_k\|_{*}^2.
\eeqq
Now, we arrive at a contradiction.

It follows from \eqref{Ch2c7} that for any $v\in\mathcal{C}$, one has
$$f(v)\geq(\frac{1}{2}-\frac{1}{\be}-\frac{\al}{\be})\|v\|_{*}^2>0.$$
\end{proof}
Set
\beq\label{Ch2diam}
D_0:=\inf_{v\in\mathcal{C}}\|v\|_*,\  C_0:=\inf_{v\in\mathcal{C}}f(v).
\eeq

\section{The properties of Palais-Smale sequences}

In this section, we investigate some properties of the Palais-Smale sequences, which will be used in the following sections.

Given any PS sequence $\{u_k\}^{\infty}_{k=1}\subset l^{2}$ satisfying that
\begin{equation} \label{Ch2c1}
\lim_{k\to+\infty}f(u_k)=C\ \mbox{and}\ \lim_{k\to+\infty} D_*f(u_k)=0.
\end{equation}
By Lemma \ref{Ch2psbou}, one has that $C\geq0$.
When $C>0$, it follows from the proof of Lemma \ref{Ch2critical} and Lemma \ref{Ch2psbou} that $\{u_k\}^{\infty}_{k=1}$ is bounded and we could assume that there is a positive constant $\tau$ such that
\beq\label{Ch2d1}
\lim_{k\to\infty}\|u_k\|_{2}=\tau,
\eeq
and
\beq\label{Ch2d2}
\inf_{k\geq0}\|u_k\|_{2}>0.
\eeq

\begin{lemma} \label{Ch2ccrit}
For any sequence $\{u_k\}^{\infty}_{k=1}\subset l^{2}$ satisfying \eqref{Ch2c1} with $C>0$, there exist $1\leq m<+\infty$ critical points $u^{(j)}$, $1\leq j\leq m$, of $f$ and a subsequence $\{u_{k_p}\}^{\infty}_{p=1}$ such that
\beq\label{Ch2f160}
\lim_{p\to+\infty}\bigg\|u_{k_p}-\sum^{m}_{j=1}u^{(j)}(\cdot+t^{(j)}_p)\bigg\|_*=0,
\eeq
where $t^{(j)}_p\in\ZZ$, for any $j\neq j'$, $|t^{(j)}_p-t^{(j')}_p|\to+\infty$ as $p\to+\infty$, and,
$$\lim_{k\to+\infty}f(u_k)=\sum^{m}_{j=1}f(u^{(j)}).$$
\end{lemma}

\begin{proof}
By \eqref{Ch2b2}, the norms $\|\cdot\|_{*}$ and $\|\cdot\|_2$ are equivalent. So, it suffices to show that \eqref{Ch2f160} holds with the norm $\|\cdot\|_2$.

Lemma \ref{Ch2cc} will be used in the proof of this theorem.
By \eqref{Ch2d2}, it is reasonable to define
\begin{equation*}
\rho_k(t):=\frac{|u_k(t)|^{2}}{\|u_k\|^{2}_{2}}.
\end{equation*}

{\bf Step 1.} It is to show that the case of vanishing can not happen.

By contradiction, suppose that there is a sequence $\{\ep_k\}^{\infty}_{k=1}$ with $\ep_k\to0$ as $k\to+\infty$, such that for any $l\in\ZZ$,
$$\sum^{l+1}_{t=l-1}|u_k(t)|^{2}\leq\ep_k\|u_k\|^{2}_{2}.$$
Since $\{u_k\}^{\infty}_{k=1}$ is a PS sequence, $\{u_k\}^{\infty}_{k=1}$ is bounded by Lemma \ref{Ch2psbou}. So, one has
$$\lim_{k\to\infty}\|u_k\|_{\infty}=0.$$
By using \eqref{Ch2c19}, (A3), and the similar discussions in the proof of Lemma \ref{Ch2critical}, we could show that $\lim_{k\to\infty}\|u_k\|_2=0$. This contradicts \eqref{Ch2d1}.

{\bf Step 2.} It is to show that if the case of concentration happens, then the statement holds.

We could assume that $p_k$ is an integral multiple of period $T$. Since $p_k$ can be written as $p_k=t_kT+s_k$, $0\leq s_k\leq T-1$, when $s_k\neq0$, let $p'_k=t_kT$, $N'=N+2T$. For these $p'_k$ and $N'$, it is evident that the concentration still holds, where $N$ is the constant specified in the statement of the concentration in Lemma \ref{Ch2cc}.

If the case of concentration happens, set
$$v_k(t):=\frac{u_k(t+y_k)}{\|u_k\|_{2}}.$$
So,
\begin{equation} \label{Ch2c9}
\sum_{t\in\ZZ}|v_k(t)|^{2}=1,\ \ \forall k\in\NN,
\end{equation}
and for every $\ep>0$, there is some positive integer $N$ such that
\begin{equation} \label{Ch2c10}
\sum_{|t|> N}|v_k(t)|^{2}\leq\ep.
\end{equation}
It follows from Lemma \ref{Ch2comp} with $p=q=2$ that there is a subsequence of $\{v_k\}^{\infty}_{k=1}$ which is convergent. Without loss of generality, assume that $\{v_k\}^{\infty}_{k=1}$ is convergent and the limit is $v_0$. This, together with \eqref{Ch2d1}, yields that
\beqq
\lim_{k\to\infty}\|u_k(\cdot+p_k)-\tau v_0\|_2=0.
\eeqq
Since $\{u_k\}^{\infty}_{k=1}$ is a PS sequence, $f$ is $C^1$, and $p_k$ is an integral multiple of period $T$, one has that $\tau v_0$ is a critical point of $f$, which implies that the statements hold for the sequence $\{u_k\}^{\infty}_{k=1}$ if the case of concentration happens.

{\bf Step 3.} It is to show that if the case of dichotomy happens, then the statement holds.

Without loss of generality, assume that $p_k$ is an integral multiple of period $T$.

Denote
\begin{equation*}
B^1_k:=[-N^1_k,N^1_k]\cap\ZZ,\
B^2_k:=[-N^2_k,N^2_k]\cap\ZZ,\ (B^2_k)^c:=(-\infty,-N^2_k-1]\cup[N^2_k+1,+\infty)\cap\ZZ,\ \ k\in\NN.
\end{equation*}

Set
$$w_k(t)=u_k(t+p_k),\ \
w^{(1)}_k(t)=w_k(t)\chi_{B^1_k}(t),$$
$$w^{(2)}_k(t)=w_k(t)\chi_{(B^2_k)^c}(t),\ \
w^{(3)}_k(t)=w_k(t)-w^{(1)}_k(t)-w^{(2)}_k(t),\ k\in\NN,$$
where $\chi_J$ is the characteristic function of $J$,
\beqq
\chi_{J}(t)=\left\{
  \begin{array}{ll}
    1, & \hbox{if}\ t\in J; \\
    0, & \hbox{if}\ t\not\in J.
  \end{array}
\right.
\eeqq

Now, it is to show that
\begin{equation} \label{Ch2c11}
f(w_k)=f(w^{(1)}_k)+f(w^{(2)}_k)+o(1)\ \mbox{as}\ k\to+\infty.
\end{equation}
Since
\beqq
\begin{split}
f(w^{(1)}_k)=&\frac{1}{2}\sum_{-N^1_k<t\leq N^1_k}\lan \De w_k(t-1),\De w_k(t-1)\ran\\
+&\frac{1}{2}\sum_{t\in B^1_k}\lan w_k(t),L(t)w_k(t)\ran-\sum_{t\in B^1_k}V(t,w_k(t))+\frac{1}{2}|w_k(-N^1_k)|^2+\frac{1}{2}|w_k(N^1_k)|^2,
\end{split}
\eeqq
\beqq
\begin{split}
f(w^{(2)}_k)=&\frac{1}{2}\sum_{t<-N^2_k,t> N^2_k+1}\lan \De w_k(t-1),\De w_k(t-1)\ran\\
+&\frac{1}{2}\sum_{t\in (B^2_k)^c}\lan w_k(t),L(t)w_k(t)\ran-\sum_{t\in (B^2_k)^c}V(t,w_k(t))+\frac{1}{2}|w_k(N^2_k+1)|^2+\frac{1}{2}|w_k(-N^2_k-1)|^2,
\end{split}
\eeqq
and
\beqq
f(w_k)=\sum_{t\in B^1_k}+\sum_{t\in (B^2_k)^c}+\sum_{t\in B^2_k\setminus B^1_k},
\eeqq
one has
\beqq
\begin{split}
&f(w_k)-f(w^{(1)}_k)-f(w^{(2)}_k)\\
=&\frac{1}{2}|\De w_k(-N^1_k-1)|^2-\frac{1}{2}|w_k(-N^1_k)|^2-\frac{1}{2}|w_k(N^1_k)|^2\\
&+\frac{1}{2}|\De w_k(N^2_k)|^2-\frac{1}{2}|w_k(N^2_k+1)|^2-\frac{1}{2}|w_k(-N^2_k-1)|^2\\
&+\sum_{t\in B^2_k\setminus B^1_k}\bigg(\frac{1}{2}(\lan \De w_k(t-1),\De w_k(t-1)\ran+\lan w_k(t),L(t)w_k(t)\ran)-V(t,w_k(t))\bigg).
\end{split}
\eeqq

By direct calculation, one has
\beqq
\begin{split}
&\frac{1}{2}|\De w_k(-N^1_k-1)|^2-\frac{1}{2}|w_k(-N^1_k)|^2-\frac{1}{2}|w_k(N^1_k)|^2\\
=&\frac{1}{2}|w_k(-N^1_k-1)|^2-\frac{1}{2}|w_k(N^1_k)|^2-\lan w_k(-N^1_k),w_k(-N^1_k-1)\ran,
\end{split}
\eeqq
\beqq
\begin{split}
&\frac{1}{2}|\De w_k(N^2_k)|^2-\frac{1}{2}|w_k(N^2_k+1)|^2-\frac{1}{2}|w_k(-N^2_k-1)|^2\\
=&\frac{1}{2}|w_k(N^2_k)|^2-\frac{1}{2}|w_k(N^2_k+1)|^2-\lan w_k(N^2_k+1),w_k(N^2_k)\ran,
\end{split}
\eeqq
and,
\beqq
\begin{split}
&\frac{1}{2}|\De w_k(-N^2_k-1)|^2+\frac{1}{2}|\De w_k(N^2_k-1)|^2+\frac{1}{2}|\De w_k(-N^1_k-2)|^2+\frac{1}{2}|\De w_k(N^1_k)|^2\\
=&\frac{1}{2}| w_k(-N^2_k)|^2-\lan w_k(-N^2_k), w_k(-N^2_k-1)\ran+\frac{1}{2}| w_k(-N^2_k-1)|^2\\
&+\frac{1}{2}| w_k(N^2_k)|^2-\lan w_k(N^2_k), w_k(N^2_k-1)\ran+  \frac{1}{2}| w_k(N^2_k-1)|^2\\
&+\frac{1}{2}| w_k(-N^1_k-1)|^2-\lan w_k(-N^1_k-1), w_k(-N^1_k-2)\ran+  \frac{1}{2}| w_k(-N^1_k-2)|^2\\
&+\frac{1}{2}| w_k(N^1_k+1)|^2-\lan w_k(N^1_k), w_k(N^1_k+1)\ran+  \frac{1}{2}| w_k(N^1_k)|^2.
\end{split}
\eeqq

By the calculations above, (i) and (ii) of Dichotomy in Lemma \ref{Ch2cc} and $\{u_n\}^{\infty}_{n=1}$ is bounded, one has that \eqref{Ch2c11} holds.

Next, we show that $\lim_{k\to\infty} D_2f(w^{(1)}_k)=0$. In the following discussions, we assume that $N^1_k$ is equal to the twice of the value provided by Lemma \ref{Ch2cc}. And, it is easy to obtain that this assumption does not affect the conclusion.

For any $u\in l^{2}$, $u$ can be written as
\begin{equation}\label{Ch2c16}
u=\chi_{B^1_k}u+\chi_{(B^2_k)^c}u+\chi_{B^2_k\setminus B^1_k}u.
\end{equation}

By direct calculation and \eqref{Ch2deriva}, one has for any $u,v\in l^2$ and $u$ is written in the form of \eqref{Ch2c16},
\begin{equation} \label{Ch2c13}
\begin{split}
&D_2f(v)(\chi_{B^1_k}u)\\
=&\sum_{-N^1_k<t\leq N^1_k}\lan\De v(t-1),\De u(t-1)\ran+\sum_{t\in B^1_k}\bigg(\lan v(t),L(t)u(t)\ran-\lan V'_x(t,v(t)),u(t)\ran\bigg)\\
+&\lan\De v(-N^1_k-1),u(-N^1_k)\ran+\lan\De v(N^1_k),-u(N^1_k)\ran,
\end{split}
\end{equation}
\begin{equation} \label{Ch2c14}
\begin{split}
&D_2f(v)(\chi_{(B^2_k)^c}u)\\
=&\sum_{\mbox{\tiny$\begin{array}{c}
t\leq-N^2_k-1\\
t> N^2_k+1\end{array}$}}\lan\De v(t-1),\De u(t-1)\ran+\sum_{t\in (B^2_k)^c}\bigg(\lan v(t),L(t)u(t)\ran-\lan V'_x(t,v(t)),u(t)\ran\bigg)\\
+&\lan\De v(-N^2_k-1),-u(-N^2_k-1)\ran+\lan\De v(N^2_k),u(N^2_k+1)\ran,
\end{split}
\end{equation}
\begin{equation} \label{Ch2c15}
\begin{split}
&D_2f(v)(\chi_{B^2_k\setminus B^1_k}u)\\
=&\sum_{\mbox{\tiny$\begin{array}{c}
-N^2_k<t\leq-N^1_k-1\\
N^1_k+1<t\leq N^2_k\end{array}$}}\lan\De v(t-1),\De u(t-1)\ran+\sum_{t\in B^2_k\setminus B^1_k}\bigg(\lan v(t),L(t)u(t)\ran-\lan V'_x(t,v(t)),u(t)\ran\bigg)\\
+&\lan\De v(-N^2_k-1),u(-N^2_k)\ran+\lan\De v(-N^1_k-1),-u(-N^1_k-1)\ran\\
+&\lan\De v(N^1_k),u(N^1_k+1)\ran+\lan\De v(N^2_k),-u(N^2_k)\ran.
\end{split}
\end{equation}
So, for any $u,v\in l^2$,
\begin{equation} \label{Ch2c12}
\begin{split}
D_2f(v)u= D_2f(v)(\chi_{B^1_k}u)+ D_2f(v)(\chi_{(B^2_k)^c}u)+ D_2f(v)(\chi_{B^2_k\setminus B^1_k}u).
\end{split}
\end{equation}
In particular, we have
\beq\label{Ch2c17}
\begin{split}
& D_2f(w^{(1)}_k)u\\
=&D_2f(w^{(1)}_k)(\chi_{B^1_k}u)+ D_2f(w^{(1)}_k)(\chi_{(B^2_k)^c}u)
 + D_2f(w^{(1)}_k)(\chi_{B^2_k\setminus B^1_k}u).
\end{split}
\eeq

By \eqref{Ch2c13}, one has
\beq
\begin{split}
& D_2f(w^{(1)}_k)(\chi_{B^1_k}u)\\
=& D_2f(w_k)(\chi_{B^1_k}u)+\lan w_k(-N^1_k-1),u(-N^1_k)\ran+\lan w_k(N^1_k+1),u(N^1_k)\ran.
\end{split}
\eeq
By \eqref{Ch2c14}, we have
\beq
 D_2f(w^{(1)}_k)(\chi_{(B^2_k)^c}u)=0.
\eeq
By \eqref{Ch2c15}
\beq\label{Ch2c18}
 D_2f(w^{(1)}_k)(\chi_{B^2_k\setminus B^1_k}u)=\lan w_k(-N^1_k),-u(-N^1_k-1)\ran+\lan-w_k(N^1_k),u(N^1_k+1)\ran.
\eeq
It follows from \eqref{Ch2c17}--\eqref{Ch2c18} that one has
\beqq
\begin{split}
& D_2f(w^{(1)}_k)u\\
=& D_2f(w_k)(\chi_{B^1_k}u)+\lan w_k(-N^1_k-1),u(-N^1_k)\ran\\
+&\lan w_k(N^1_k+1),u(N^1_k)\ran+\lan w_k(-N^1_k),-u(-N^1_k-1)\ran+\lan-w_k(N^1_k),u(N^1_k+1)\ran
\end{split}
\eeqq
Hence, $\lim_{k\to\infty} D_2f(w^{(1)}_k)=0$.\

By the similar discussions in Step 2, one has that there exist a convergent subsequence of $\{w^{(1)}_k\}^{\infty}_{k=1}$ (without loss of generality, we could assume that this subsequence is $\{w^{(1)}_k\}^{\infty}_{k=1}$) and $z_1\in l^{2}$ such that
\begin{equation} \label{Ch2e1}
\lim_{k\to+\infty} w^{(1)}_k=z_1\ \mbox{in}\ l^{2}\ \mbox{and}\ D_2f(z_1)=0.
\end{equation}

Next, consider the following sequence
$$\hat{w}_k=w_k-z_1.$$
It follows from \eqref{Ch2c11} and Lemma \ref{Ch2psbou} that
$$\lim_{k\to+\infty}f(\hat{w}_k)=C-f(z_1)\geq0.$$
Next, it is to show that
\beq\label{Ch2eee2}
\lim_{k\to+\infty}D_2f(\hat{w}_k)=0.
\eeq
It follows from \eqref{Ch2deriva} and direct calculation that
\beqq
\begin{split}
& D_2f(\hat{w}_k)u+\sum_{t\in\ZZ}\lan V'_x(t,\hat{w}_k(t)),u(t)\ran\\
=& D_2f(w_k)u- D_2f(z_1)u+\sum_{t\in\ZZ}\lan V'_x(t,w_k(t)),u(t)\ran-\sum_{t\in\ZZ}\lan V'_x(t,z_1(t)),u(t)\ran.
\end{split}
\eeqq
By (A3) and \eqref{Ch2e1}, it is easy to show that for any $u\in l^2$ with $\|u\|_2\leq1$,
\beqq
\begin{split}
&\sum_{t\in\ZZ}\lan V'_x(t,w_k(t)),u(t)\ran-\sum_{t\in\ZZ}\lan V'_x(t,z_1(t)),u(t)\ran-\sum_{t\in\ZZ}\lan V'_x(t,\hat{w}_k(t)),u(t)\ran\\
=&\sum_{t\in\ZZ}\lan V'_x(t,w^{(1)}_k(t)),u(t)\ran+\sum_{t\in\ZZ}\lan V'_x(t,w^{(2)}_k(t)),u(t)\ran+\sum_{t\in\ZZ}\lan V'_x(t,w^{(3)}_k(t)),u(t)\ran\\
&-\sum_{t\in\ZZ}\lan V'_x(t,z_1(t)),u(t)\ran-\sum_{t\in\ZZ}\lan V'_x(t,w^{(1)}_k(t)+w^{(2)}_k(t)+w^{(3)}_k(t)-z_1(t)),u(t)\ran\\
=&o(1)\ \mbox{as}\ k\to\infty,
\end{split}
\eeqq
where the small constant $o(1)$ does not depend on the choice of $u$. So, \eqref{Ch2eee2} holds.

Hence, one has
$$\lim_{k\to+\infty}f(\hat{w}_k)=C-f(z_1),\ \ \lim_{k\to+\infty}D_2f(\hat{w}_k)=0.$$

If $C-f(z_1)=0$, we finish the proof. If $C-f(z_1)>0$, without loss of generality, assume that the sequence $\hat{w}_k$ satisfying the following
assumptions:
\beqq
0<\inf_{k\geq1}\|\hat{w}_k\|_{2}\ \mbox{and}\ \lim_{k\to\infty}\|\hat{w}_k\|_{2}\ \mbox{exists and is equal to a positive constant}.
\eeqq
For the sequence $\hat{w}_k$, repeat the procedure above. By Lemmas \ref{Ch2finite} and \eqref{Ch2diam}, we could repeat the procedure for at most $[C/C_0]+1$ times. Suppose the number that we could repeat is $m$. Thus, there are $u^{(1)},...,u^{(m)}\in l^2$, such that
$$D_2f(u^{(i)})=0,\ 1\leq i\leq m;\
\lim_{k\to+\infty}f(w_k)=\sum^m_{i=1}f(u^{(i)}),$$
and \eqref{Ch2f160} holds. The properties of $t^{(j)}_p$ can be derived from the properties of the chosen sequences $p_k$, $N^1_k$, and $N^2_k$.
This completes the whole proof.

\end{proof}

By applying similar discussions used in the proof of Lemmas 3.4-3.10 in \cite{CalMon}, one has the following results, Lemmas \ref{Ch2sumc}--\ref{Ch2diffsum}. For the convenience of the readers and the completeness of the paper, we give the proof of the corresponding statements.

\begin{lemma} \label{Ch2sumc}
For any sequence $\{u_k\}^{\infty}_{k=1}\subset l^2$ in the form of $u_k=\sum^p_{i=1}v_i(\cdot-t^i_k)$, where $p$ is a fixed positive integer, $v_1,...,v_p\in l^2$, and
$\{t^1_k\},...,\{t^p_k\}$ are sequences in $\ZZ$ such that $\lim_{k\to\infty}|t^i_k-t^j_k|=\infty$ for $i\neq j$, then $\lim_{k\to\infty}\|u_k\|_*^2=
\sum^{p}_{i=1}\|v_i\|_*^2$.
\end{lemma}

 \begin{proof}
It is to show the statement by induction. It is evident the statement is correct if $p=1$. Suppose the statement is correct when $p=l$,  we need to show the statement holds when $p=l+1$. Since
\beqq
\|u_k\|_{*}^2=\|\sum^{l}_{i=1}v_i(\cdot-t^i_k)\|_{*}^2+\|v_{l+1}\|_{*}^2+\sum^l_{i=1}\lan v_i,v_{l+1}(\cdot-t^{l+1}_k+t^{i}_k)\ran_{*},
\eeqq
and for any $i\neq j$, $\lim_{k\to\infty}|t^i_k-t^j_k|=\infty$, one has that the statement holds when $p=l+1$. This completes the proof.
\end{proof}

\begin{lemma} \label{Ch2diamconv}
If the sequence $\{u_k\}^{\infty}_{k=1}\subset l^2$ is a PS sequence with diameter less than $D_0/2$, then $\{u_k\}^{\infty}_{k=1}$ has a convergent subsequence, where $D_0$ is specified in \eqref{Ch2diam}.
\end{lemma}

 \begin{proof}
By Lemma \ref{Ch2ccrit}, there is a subsequence of the PS sequence $\{u_k\}^{\infty}_{k=1}\subset l^2$ (without loss of generality, assume this sequence is $\{u_k\}^{\infty}_{k=1}$) such that
\beqq
u_k=v_0+\sum^p_{i=1}v_i(\cdot-t^i_k)+w_k,
\eeqq
where $\|w_k\|_{*}\to0$ as $k\to\infty$, and for any $i\neq j$, $\lim_{k\to\infty}|t^i_k-t^j_k|=\infty$. By Lemmas \ref{Ch2sumc} and \ref{Ch2finite}, one has
\beqq
\frac{D^2_0}{4}\geq\|u_k-u_0\|_{*}^2=\|v_0-u_0\|_{*}^2+\sum^p_{i=1}\|v_i\|_{*}^2+o(1)\geq p D_0+o(1),
\eeqq
which implies that $p=0$. Hence,\ $u_k=v_0+w_k$.\
\end{proof}

\begin{lemma}\label{Ch2b3}
For any PS sequence $\{u_k\}^{\infty}_{k=1}\subset l^2$, if there is a positive integer $N$ such that
\beqq
\sum_{t\in(\ZZ\setminus[-N,N])}\big(\lan\De u_k(t-1),\De u_k(t-1)\ran+\lan u_k(t),L(t)u_k(t)\ran\big)<D^2_0/4,\ \forall k\in\NN,
\eeqq
then there is a convergent subsequence of $\{u_k\}^{\infty}_{k=1}$, where $D_0$ is specified in \eqref{Ch2diam}.
\end{lemma}

 \begin{proof}
There is a subsequence of $\{u_k\}^{\infty}_{k=1}\subset l^2$ which is uniformly convergent over the interval $[-N,N]$. This, together with the method used in the proof of Lemma \ref{Ch2diamconv}, implies that there is a convergent subsequence.
\end{proof}

As in \cite{CalMon, Sere}, we introduce the following notations:
$$\mathcal{S}^{b}_{PS}:=\{\{u_k\}\subset l^2:\ \lim_{k\to\infty}D_*f(u_k)=0,\ \limsup_{k\to\infty}f(u_k)\leq b\},$$
$$\Phi^b:=\{l\in\RR:\ \mbox{there exists}\ \{u_k\}\subset\mathcal{S}^b_{PS}\ \mbox{such that}\ \lim_{k\to\infty}f(u_k)=l\},$$
$$D^b:=\{r\in\RR:\ \mbox{there exist}\ \{u_k\},\{u'_k\}\in\mathcal{S}^b_{PS}\ \mbox{such that}\ \lim_{k\to\infty}\|u_k-u'_k\|_*=r\}.$$
For any real numbers $a$ and $b$ with $a\leq b$, set
$$f_a:=f^{-1}([a,+\infty)),\ f^b:=f^{-1}((-\infty,b]),\ f^b_a:=f^{-1}([a,b]),$$
$$\mathcal{C}_a:=\mathcal{C}\cap f_a,\ \mathcal{C}^b:=\mathcal{C}\cap f^b,\ \mathcal{C}(b):=\mathcal{C}\cap f^b_b.$$
Given any non-empty subset $S\subset E$, for any $r>0, r_2>r_1\geq0$, set
$$\mathcal{B}_r(S):=\{u\in l^2:\ \mbox{dist}(u,S)<r\},\ \mathcal{A}_{r_1,r_2}(S):=\{u\in l^2:\ r_1<\mbox{dist}(u,S)<r_2\}.$$

\begin{lemma}  \label{Ch2cptd}
For any $b\geq0$, the set $\Phi^b$ and $D^b$ are compact subsets of $\RR$.
\end{lemma}

 \begin{proof}
By \eqref{Ch2c7}, any sequence $\{u_k\}^{\infty}_{k=1}\subset\mathcal{S}^{b}_{PS}$ is bounded, which implies that $D^b$ is bounded, and the bound is dependent on the constant $b$. Hence, we need to show that $D^b$ is closed. Next, it is to show that for any $r\not\in D^b$, there is $\varrho>0$ such that $(r-\varrho,r+\varrho)\cap D^b=\emptyset$. By contradiction, suppose that there are $r\not\in D^b$ and a sequence $r_k\in(r-1/k,r+1/k)\cap D^b$.\ Then, there exist sequences $\{u_k\}^{\infty}_{k=1},\{u'_k\}^{\infty}_{k=1}\subset l^2$\ such that\ $\|D_*f(u_k)\|_{*}<1/k$,\ $\|D_*f(u'_k)\|_{*}<1/k$,\
$f(u_k)\leq b+1/k$, $f(u'_k)\leq b+1/k$, and $|\|u_k-u'_k\|_{*}-r_k|<1/k$, which implies that $\{u_k\}^{\infty}_{k=1},\{u'_k\}^{\infty}_{k=1}\subset \mathcal{S}^{b}_{PS}$, and $\|u_k-u'_k\|_{*}\to r$ as $k\to\infty$. Hence, $r\in D^b$. This is a contradiction. Similarly, we could show that $\Phi^b$ is compact.
\end{proof}

\begin{lemma} \label{Ch2diffposi}
Given $b>0$, for any $r\in[0,\infty)\setminus D^b$, there exists $d_r\in(0,r/3)$ such that $[r-3d_r,r+3d_r]\subset[0,\infty)\setminus D^b$ and
there is $\mu_r>0$ such that $\|D_*f(u)\|_*\geq\mu_r$ for any $u\in\mathcal{A}_{r-3d_r,r+3d_r}(\mathcal{C}^b)\cap f^b$.
\end{lemma}

 \begin{proof}
It follows from Lemma \ref{Ch2cptd} that $(0,\infty)\setminus D^b$ is open, the existence of $d_r$ can be derived from this fact. Next, it is to show the existence of $\mu_r$ by contradiction.
Suppose that there is $\{u_k\}^{\infty}_{k=1}\subset \mathcal{A}_{r-3d_r,r+3d_r}(\mathcal{C}^b)\cap f^b$\ such that\ $D_*f(u_k)\to0$.\ Further, for any $k$, there is
$v_k\in\mathcal{C}^b$ such that\ $r-3d_r\leq\|u_k-v_k\|_{*}\leq r+3d_r$.\ Hence,\ $\{u_k\}^{\infty}_{k=1}, \{v_k\}^{\infty}_{k=1}\subset \mathcal{S}^{b}_{PS}$.
Hence, there exist subsequences (without loss of generality, assume that the subsequences are $\{u_k\}^{\infty}_{k=1}, \{v_k\}^{\infty}_{k=1}$) such that $\|u_k-v_k\|_{*}\to\bar{r}\in[r-3d_r,r+3d_r]$.\ Hence,\ $\bar{r}\in D^b$,\ $[r-3d_r,r+3d_r]\cap D^b\neq\emptyset$.\ This is a contradiction.
\end{proof}

\begin{remark} \label{Ch2k40}
Using the method used in the proof of Lemma \ref{Ch2diffposi}, we could show that: for any given $b>0$ and any $r\in(0,\infty)\setminus D^b$, there are $d_r\in(0,r/3)$\ such that\ $[r-3d_r,r+3d_r]\subset[0,\infty)\setminus D^b$ and $\mu'_r>0$\ such that for any $u\in\mathcal{A'}_{r-3d_r,r+3d_r}(\mathcal{C}^b)\cap f^b$,\ $\|D_*f(u)\|_{*}\geq\mu'_r$,\ where
\beqq
\mathcal{A'}_{r-3d_r,r+3d_r}(\mathcal{C}^b)=\{u\in l^2:\ \exists w\in\mathcal{C}^b\ \mbox{such that}\ r-3d_r<\|u-w\|_{*}<r+3d_r\}.
\eeqq
\end{remark}

\begin{lemma} \label{Ch2diffposi1}
Given $b>0$, for any $r\in[0,\infty)\setminus \Phi^b$, there exists $\rho_r>0$ such that $[r-\rho_r,r+\rho_r]\subset[0,\infty)\setminus \Phi^b$ and there exists $\nu>0$ such that $\|D_*f(u)\|_*\geq\nu$ for any $u\in f^{r+\rho_r}_{r-\rho_r}\cap f^{b}$.
\end{lemma}

 \begin{proof}
The method used here is the same as that used in the proof of Lemma \ref{Ch2diffposi}, the proof is omitted here.
\end{proof}

\begin{lemma} \label{Ch2diffsum}
Any number $r\in D^b$ can be represented as follows
$$r=\bigg(\sum^k_{j=1}\|v_j(\cdot-t_j)-\bar{v}_j\|_*^2\bigg)^{1/2},$$
where $t_j$ are integers, $v_j, \bar{v}_j\in\mathcal{C}\cup\{0\}$, $1\leq j\leq k$, and $\sum^{k}_{j=1}f(v_j)\leq b$, $\sum^{k}_{j=1}f(\bar{v}_j)\leq b$.
\end{lemma}
\begin{remark}
If $T=1$, it is easy to show that $t_j=0$, $1\leq j\leq k$.
\end{remark}

 \begin{proof}
For any $r\in D^b$, by Lemma \ref{Ch2ccrit}, there are two sequences $\{u_k\}^{\infty}_{k=1},\{u'_k\}^{\infty}_{k=1}\in\mathcal{S}^{b}_{PS}$ such that $\|u_k-u'_k\|_{*}\to r$. Without loss of generality, assume that
\beqq
u_k=v_0+\sum^p_{j=1}v_j(\cdot-t^j_k),\ u'_k=\sum^{p+h}_{j=p+1}v_j(\cdot-t^j_k)+v_{h+p+1},
\eeqq
where\ $h,p$ are non-negative integers, if $1\leq i<j\leq p$ or $p+1\leq i<j\leq p+h$, $|t^i_k-t^j_k|\to\infty$ as $k\to\infty$, which implies that for any $j\in\{1,...,p\}$, there is at most one $l=l(j)\in\{p+1,...,p+h\}$ such that $\sup_k|t^j_k-t^l_k|<\infty$. For this situation, without loss of generality, assume that
$t^j_k-t^l_k=t^j$ and let $v^1_j:=v_j$, $v^2_j:=v_{l(j)}$. For other situation, let $l(j):=-1$, $v^1_j:=v_j$, $v^2_j:=0$,\ $t^j:=0$.\

Hence, for any $j\in\{p+1,...,p+h\}\setminus\{l(1),...,l(p)\}$, if $i\in\{1,...,p+h\}\setminus\{j\}$, then $|t^i_k-t^j_k|\to\infty$ as $k\to\infty$.
For these $j$, let $v^1_j:=0$,\ $v^2_j:=v_j$. For $j\in\{p+1,...,p+h\}\setminus\{l(1),...,l(p)\}$, let $v^1_j:=v^2_j:=0$. Finally, let $v^1_0:=v_0$,\
$v^2_0:=v_{h+p+1}$.

Hence,
\beqq
\sum^{p+h}_{j=0}f(v^1_j)=\sum^p_{j=0}f(v_j)=\lim_{k\to\infty}f(u_k),\
\sum^{p+h}_{j=0}f(v^2_j)=\sum^{p+h}_{j=p+1}f(v_j)=\lim_{k\to\infty}f(u'_k).
\eeqq
\beqq
u_k-u'_k=v^1_0-v^2_0+\sum^{p+h}_{j=1}[v^1_j(\cdot-t^j_k-t^j)-v^2_j(\cdot-t^j_k)].
\eeqq
By Lemma \ref{Ch2sumc}, one has
$$r^2=\lim_{k\to\infty}\|u_k-u'_k\|_{*}^2=\sum^{p+h}_{j=0}\|v^1_j(\cdot-t^j)-v^2_j\|_{*}^2.$$

\end{proof}

By Lemma \ref{Ch2ccrit}, we have
\begin{lemma} \label{Ch2zhisum}
$\Phi^b=\{\sum f(v_i):\ v_i\in\mathcal{C}\}\cap[0,b]$.
\end{lemma}

Now, the critical assumption on multibump solutions is introduced.

\begin{itemize}
\item [(A6)] There is $\kappa^*>\kappa$ such that the set $\mathcal{C}\cap\{u\in l^2:\ f(u)\leq\kappa^*\}$ is countable.
\end{itemize}

\begin{remark} \label{Ch2k42}
If (A6) holds,  it follows from Lemmas \ref{Ch2diffsum} and \ref{Ch2zhisum} that $D^{\kappa^*}$ and $\Phi^{\kappa^*}$ are countable.
\end{remark}

\bigskip

\section{Existence of a local mountain pass type critical point}

In this section, we investigate the properties of the functional $f$ under the action of certain flows and show the existence of a local mountain pass type critical point. In the following discussions, we always assume that (A6) holds. The statements in this section can be derived from the methods used in the proof of Lemmas 4.3--4.7 in \cite{CalMon}. For the convenience of the readers and the completeness of our work, we give the details.

\begin{lemma}\label{Ch2vecfield}
Fix any $b\leq\kappa^*$.\ Suppose (A6) holds. If there exists a local Lipschitz continuous vector field $\mathcal{Y}:l^2\to l^2$ satisfying that
\begin{itemize}
\item $D_*f(u)(\mathcal{Y}(u))\leq0$ for any $u\in l^2$;
\item $\|\mathcal{Y}(u)\|_{*}\leq\frac{2}{\|D_*f(u)\|_{*}}$ for any $u\in l^2\setminus (\mathcal{C}\cup\{0\})$;
\item $\mathcal{Y}(u)=0$ for any $u\in\mathcal{C}^b\cup\{0\}$.\
\end{itemize}
Consider the following initial value problem:
\beqq
\left\{
  \begin{array}{ll}
    \frac{d\eta}{ds}(s,u)=\mathcal{Y}(\eta(s,u)); \\
    \eta(0,u)=u,
  \end{array}
\right.
\eeqq
Then, for any $u\in l^2$, there is a unique local solution $\eta(\cdot,u)$, such that for any $u\in f^{b}$, the solution is continuously dependent on the initial value $u$, and the solution is defined on the whole $\RR^+$. Further, the function $s\to f(\eta(s,u))$ is non-increasing.
\end{lemma}

 \begin{proof}
The existence, uniqueness, and continuous dependence of the solutions can be derived from the assumption that the vector field is local Lipschitz \cite{Deim}. The function $s\to f(\eta(s,u))$ is non-increasing, since
$$\frac{d}{ds}f(\eta(s,u))= D_*f(\eta(s,u))(\mathcal{Y}(\eta(s,u)))\leq0.$$
Next, it is to show that for any $u\in f^{b}$, the corresponding solution is defined on the whole positive real axis. By contradiction, suppose that there is some $u\in f^{b}$ with the maximal solution interval is $[0,t_0)$, where $t_0<\infty$.\ Hence, there exists an increasing sequence
$\{t_n\}^{\infty}_{n=1}\subset[0,t_0)$\ such that\ $t_n\to t_0$ and $\|\mathcal{Y}(\eta(t_n,u))\|_{*}\to\infty$.\ If there does not exist such sequence, then we could extend the solution interval at the point $t_0$, this contradicts the assumption that $[0,t_0)$ is the maximal interval. Set $u_n:=\eta(t_n,u)$, it follows from $\|\mathcal{Y}(u)\|_{*}\leq\frac{2}{\|D_*f(u)\|_{*}}$ and $\|\mathcal{Y}(\eta(t_n,u))\|_{*}\to\infty$ that $D_*f(u_n)\to0$. This, together with the fact $\{t_n\}^{\infty}_{n=1}$ is increasing, yields that $f(u_{n+1})\leq f(u_n)\leq f(u)$. Thus, $\{u_n\}^{\infty}_{n=1}\in \mathcal{S}^{b}_{PS}$.\

Since for any $u\in\mathcal{C}^b\cup\{0\}$,\ $\mathcal{Y}(u)=0$, there does not exist any convergent subsequence of $\{u_n\}^{\infty}_{n=1}$. If this statement is incorrect, then there is a convergent subsequence with the limit $w$, which implies that $\mathcal{Y}(w)=0$. This contradicts  $\|\mathcal{Y}(\eta(t_n,u))\|_{*}\to\infty$. Hence, there are a positive constant $\tau$,\ two subsequences of $\{u_n\}^{\infty}_{n=1}$, $\{u_{n_p}\}^{\infty}_{p=1}$ and $\{u_{n_q}\}^{\infty}_{1=1}$ satisfying that $\|u_{n_p}-u_{n_q}\|_{*}\geq\tau$ and $n_p<n_q<n_{p+1}$.\ This, together with (A6), implies that $[r_1,r_2]\subset\RR^+\setminus D^b$ and $0<r_1<r_2<\tau$.\ Hence, there are two sequences $\{s_p\}^{\infty}_{p=1},\{s_q\}^{\infty}_{q=1}\subset[0,t_0)$ satisfying that
$t_{n_p}\leq s_p<s_q\leq t_{n_q}$\ such that\ $\|\eta(s_p,u)-u_{n_p}\|_{*}=r_1$,\ $\|\eta(s_q,u)-u_{n_p}\|_{*}=r_2$, and for any $s\in(s_p,s_q)$,\ $\eta(s,u)\in\mathcal{A}_{r_1,r_2}(u_{n_p})$,\ yielding that
$$r_2-r_1\leq\int^{s_q}_{s_p}\|\mathcal{Y}(\eta(s,u))\|_{*}ds=(s_q-s_p)\|\mathcal{Y}(\eta(s_{p,q},u))\|_{*},\ s_{p,q}\in[s_p,s_q].$$
This implies that
$$ \|\mathcal{Y}(u_{p,q})\|_{*}\geq\frac{r_2-r_1}{s_q-s_p}\to\infty\ \mbox{as}\ p,q\to\infty,$$
where\ $u_{p,q}=\eta(s_{p,q},u)$.\
Hence,\ $\{u_{p,q}\}\in\mathcal{S}^b_{PS}$ and $r_1\leq\|u_{p,q}-u_{n_p}\|_{*}\leq r_2$, which means that $[r_1,r_2]\cap D^b\neq\emptyset$.\ This is a contradiction. This completes the whole proof.

\end{proof}

Fix any\ $b\in[\kappa,\kappa^*)$,\ $r\in\RR^+\setminus D^{\kappa^*}$, set $h_r:=\frac{1}{4}d_r\mu_r$, where $d_r$ and $\mu_r$ are specified in Lemma \ref{Ch2diffposi}. Set
\beq\label{Ch2b6}
\hat{h}:=\min\{\kappa^*-b,h_r\}.
\eeq

\begin{lemma} \label{Ch2vect2}
Fix any\ $h\in(0,\hat{h})$,\ there exists a continuous function $\eta:f^{b+h}\to f^{b+h}$\ such that
\begin{itemize}
\item [(a1)] $f(\eta(u))\leq f(u)$ for any $u\in f^{b+h}$;
\item [(a2)] if $\eta(u)\not\in\mathcal{B}_r(\mathcal{C}^{b+h}_{b-h})$, then $f(\eta(u))<b-h$;\
\item [(a3)] if $\eta(u)\in\mathcal{A}_{r-d_r,r+d_r}(\mathcal{C}^{b+h}_{b-h})$, then $f(\eta(u))<b-h_r$.
\end{itemize}
\end{lemma}

 \begin{proof}
Fix any $r$, it follows from Lemma \ref{Ch2diffposi} that for any $u\in\mathcal{A}_{r-3d_r,r+3d_r}(\mathcal{C}^{b+h})\cap f^{b+h}$, one has that $\|D_*f(u)\|_{*}\geq\mu_r$.\ Hence, we could construct a vector field satisfying the assumptions of Lemma \ref{Ch2vecfield} and the following assumptions:
\beq\label{Ch2k1}
D_*f(u)(\mathcal{Y}(u))\leq-1,\ \forall u\in [f^{b+h}_{b-h}\setminus\mathcal{B}_{r-2d_r}(\mathcal{C}^{b+h}_{b-h})]\cup[f^{b+h}\cap\mathcal{A}_{r-2d_r,r+2d_r}(\mathcal{C}^{b+h}_{b-h})].
\eeq
By Lemma \ref{Ch2vecfield}, there is a continuous function $\eta:\RR^+\times f^{b+h}\to f^{b+h}$, which is the solution of the initial value problem with respect to the vector field $\mathcal{Y}$. Set
$\psi(u):=\eta(3h_r,u)$,\ where\ $u\in f^{b+h}$.\ By Lemma \ref{Ch2vecfield}, $f(\psi(u))\leq f(u)$, for any $u\in f^{b+h}$.

Now, it is to  show that the function $\psi$ satisfies (a2). By contradiction, suppose that there is $u\in f^{b+h}$\ such that\ $\psi(u)\in f_{b-h}\setminus\mathcal{B}_r(\mathcal{C}^{b+h}_{b-h})$.\ There are two different situations:
\begin{itemize}
\item [(i)] $\eta(s,u)\not\in\mathcal{B}_{r-2d_r}(\mathcal{C}^{b+h}_{b-h})$ for all $s\in[0,3h_r]$;
\item [(ii)] there is $t_0\in[0,3h_r]$\ such that\ $\eta(t_0,u)\in\mathcal{B}_{r-2d_r}(\mathcal{C}^{b+h}_{b-h})$.
\end{itemize}
For Case (i), it follows from \eqref{Ch2k1} that
\beqq
f(\psi(u))-f(u)=\int^{3h_r}_0 D_*f(\eta(s,u))(\mathcal{Y}(\eta(s,u)))ds\leq-3h_r,
\eeqq
which implies that
\beqq
f(u)\geq f(\psi(u))+3h_r\geq b-h+3h>b+h,
\eeqq
which contradicts the fact $u\in f^{b+h}$.

For Case (ii), since we assume that $\psi(u)\in f_{b-h}\setminus\mathcal{B}_r(\mathcal{C}^{b+h}_{b-h})$, there are $0\leq t_1<t_2\leq 3h_r$ such that
$\eta(t_1,u)\in\mathcal{B}_{r-2d_r}(\mathcal{C}^{b+h}_{b-h})$,\ $\eta(t_2,u)\in\mathcal{B}_{r}(\mathcal{C}^{b+h}_{b-h})$, and for any $t\in(t_1,t_2)$,\
$\eta(t,u)\in\mathcal{A}_{r-2d_r,r}(\mathcal{C}^{b+h}_{b-h})$.\ Hence,
\beqq
2d_r\leq\|\eta(t_2,u)-\eta(t_1,u)\|_{*}\leq\int^{t_2}_{t_1}\|\mathcal{Y}(\eta(t,u))\|_{*}dt\leq\int^{t_2}_{t_1}\frac{2}{\|D_*f(\eta(t,u))\|_{*}}dt.
\eeqq
This, together with Lemma \ref{Ch2diffposi}, implies that,
\beqq
2d_r\leq\frac{2}{\mu_r}(t_2-t_1),
\eeqq
yielding that
\beqq
t_2-t_1\geq 4h_r.
\eeqq
Now, we arrive at a contradiction.

Next, it is to show (a3). We show this by contradiction. Assume that there is $u\in f^{b+h}$ such that $\psi(u)\in\mathcal{A}_{r-d_r,r+d_r}(\mathcal{C}^{b+h}_{b-h})\cap f_{b-h_r}$.\ There exist two different situations:
\begin{itemize}
\item[(i)]  $\eta(t,u)\in\mathcal{A}_{r-2d_r,r+2d_r}(\mathcal{C}^{b+h}_{b-h})$ for all $t\in[0,3h_r]$;
\item[(ii)] there is $t_0\in[0,3h_r]$\ such that\ $\eta(t_0,u)\not\in\mathcal{A}_{r-2d_r,r+2d_r}(\mathcal{C}^{b+h}_{b-h})$.
\end{itemize}
For Case (i), by $\eta(s,u)\in f^{b+h}_{b-h}$,\ $s\in[0,3h_r]$,\ and\ \eqref{Ch2k1},\ one has
\beqq
f(\psi(u))-f(u)=\int^{3h_r}_0 D_*f(\eta(s,u))(\mathcal{Y}(\eta(s,u)))ds\leq-3h_r,
\eeqq
which implies that
\beqq
f(u)\geq f(\psi(u))+3h_r\geq b-h+3h>b+h,
\eeqq
contradicting the fact $u\in f^{b+h}$.

For Case (ii), there exists $0\leq t_1<t_2\leq 3h_r$ satisfying $\|\eta(t_2,u)-\eta(t_1,u)\|_{*}=d_r$,\ $\eta(t,u)\in\mathcal{A}_{r-2d_r,r+2d_r}(\mathcal{C}^{b+h}_{b-h})$,\ where
$t\in(t_1,t_2)$.\ Hence,\
\beqq
d_r=\|\eta(t_2,u)-\eta(t_1,u)\|_{*}\leq\int^{t_2}_{t_1}\|\mathcal{Y}(\eta(t,u))\|_{*}dt\leq\int^{t_2}_{t_1}\frac{2}{\|D_*f(\eta(t,u))\|_{*}}dt,
\eeqq
this, together with Lemma \ref{Ch2diffposi}, yields that
\beqq
d_r\leq\frac{2}{\mu_r}(t_2-t_1),
\eeqq
which implies that
\beqq
t_2-t_1\geq 2h_r.
\eeqq
By \eqref{Ch2k1},
\beqq
\begin{split}
&f(\psi(u))\leq f(\eta(t_2,u))\\
=&f(\eta(t_1,u))+\int^{t_2}_{t_1} D_*f(\eta(t,u))(\mathcal{Y}(\eta(t,u))) dt\leq f(u)-(t_2-t_1),
\end{split}
\eeqq
yielding that
\beqq
f(u)\geq b-h_r+2h_r>b+h.
\eeqq
This contradicts the assumption $u\in f^{b+h}$. This completes the whole proof.
\end{proof}

\begin{lemma} \label{Ch2path2}
For any $h\in(0,\hat{h})$, there exist a continuous path $\ga\in\mathcal{K}$ and finite number critical points $v_1$,...,$v_q\in\mathcal{C}^{\kappa+h}_{\kappa-h}$, these critical points are dependent on $h$\ and\ $\ga$, such that,
\begin{itemize}
\item [(b1)]\ $\max_{t\in[0,1]}f(\ga(t))<\kappa+h$;
\item [(b2)] if $f(\ga(s))\geq \kappa-h$, then $\ga(s)\in\cup^q_{j=1}\mathcal{B}_r(v_j)$;
\item [(b3)] if $\ga(s)\in\cup^q_{j=1}\mathcal{A}_{r-d_r,r+d_r}(v_j)$, then $f(\ga(s))<\kappa-h_r$,
\end{itemize}
where\ $\hat{h}=\min\{\kappa^*-\kappa,h_r\}$,\ $\mathcal{K}$ is specified in \eqref{Ch2path1}.
\end{lemma}

\begin{proof}
Take $\ga\in\mathcal{K}$\ such that\ $\max_{t\in[0,1]}f(\ga(t))<\kappa+h$. Set $\hat{\ga}:=\psi\circ\ga$,\ where $\psi$ is the function introduced from Lemma \ref{Ch2vect2}, where $b=\kappa$. It is evident that $\hat{\ga}\in\mathcal{K}$.\ Hence, it follows from (a1) of Lemma \ref{Ch2vect2} that (b1) holds. By (a2) of Lemma \ref{Ch2vect2}, one has that if $f(\hat{\ga}(s))\geq\kappa-h$, then $\hat{\ga}(s)\in\mathcal{B}_r(\mathcal{C}^{\kappa+h}_{\kappa-h})$.\ Hence, the family of sets $\{B_r(v):\ v\in\mathcal{C}^{\kappa+h}_{\kappa-h}\}$ is an open cover of the compact set $\mbox{range}\hat{\ga}\cap f_{\kappa-h}$. So, there are finite number critical points $v_1,...,v_q\in\mathcal{C}^{\kappa+h}_{\kappa-h}$\ such that\ $\mbox{range}\hat{\ga}\cap f_{\kappa-h}\subset\cup^q_{j=1}B_r(v_j)$. Hence, (b2) holds. (b3) can be derived from (a3) of Lemma \ref{Ch2vect2}. The proof is thus completed.
\end{proof}

Fix
\beq\label{Ch2k7}
\bar{r}\in(0,D_0/4)\setminus D^{\kappa^*},\ \bar{h}\in\big(0,\frac{1}{2}\min\{h_{\bar{r}},\kappa^*-\kappa\}\big),
\eeq
where $D_0$ is specified in \eqref{Ch2diam}, $h_{\bar{r}}$ is the same as the definition of $h_r$ defined in \eqref{Ch2b6}. By Lemma \ref{Ch2path2}, there exist $\bar{\ga}\in\mathcal{K}$ and finite number critical points $v_1,...,v_q\in\mathcal{C}^{\kappa+\bar{h}}_{\kappa-\bar{h}}$ satisfying (b1)--(b3) of  Lemma \ref{Ch2path2}. By the definition of $\kappa$, there exist $\bar{v}\in\{v_1,...,v_q\}$ and $[s_0,s_1]\subset[0,1]$\ such that\ $\bar{\ga}([s_0,s_1])\subset \mathcal{B}_{\bar{r}}(\bar{v})$,\ $\bar{\ga}(s_0),\bar{\ga}(s_1)\in\partial\mathcal{B}_{\bar{r}}(\bar{v})$, and $\bar{\ga}(s_0)$\ and\ $\bar{\ga}(s_1)$ are not $\kappa$-connectible in $l^2$.\

By (b3) of Lemma \ref{Ch2path2}, $\bar{\ga}(s_0),\bar{\ga}(s_1)\in f^{\kappa-h_{\bar{r}}}$. Consider the following set:
\beq\label{Ch2f48}
\bar{\mathcal{K}}:=\{\ga\in C([0,1],l^2):\ \ga(0)=\bar{\ga}(s_0),\ \ga(1)=\bar{\ga}(s_1),\ \mbox{and}\ \mbox{range}\ga\subset\mathcal{B}_{\bar{r}}(\bar{v})\cup f^{\kappa-\frac{1}{2}h_{\bar{r}}}\}.
\eeq
Hence,\ $\bar{\mathcal{K}}$ is non-empty, the corresponding minimax value is defined as follows:
\beq\label{Ch2f47}
\bar{\kappa}:=\inf_{\ga\in\bar{\mathcal{K}}}\sup_{t\in[0,1]}f(\ga(t)).
\eeq
Obviously,\ $\kappa\leq\bar{\kappa}<\kappa+\bar{h}<\kappa^*$. It is possible that $\bar{\kappa}$ is not a critical value, we will show that it is a critical value under certain conditions.

\begin{lemma} \label{Ch2path3}
For any $r\in(0,\frac{1}{2}d_{\bar{r}})\setminus D^{\kappa^*}$ and any $h\in(0,\min\{\kappa+\bar{h}-\bar{\kappa},h_r\})$, there exist $v_{r,h}\in\mathcal{C}^{\bar{\kappa}+h}_{\bar{\kappa}-h}\cap\mathcal{B}_{\bar{r}}(\bar{v})$,\ $u^0_{r,h},u^1_{r,h}\in\mathcal{B}_{\bar{r}}(\bar{v})$\ and a continuous path $\ga_{r,h}\in C([0,1],l^2)$ with $\ga_{r,h}(0)=u^0_{r,h},\ga_{r,h}(1)=u^1_{r,h}$\ such that
\begin{itemize}
\item [(c1)] $u^0_{r,h},u^1_{r,h}\in\partial\mathcal{B}_{r+d_r}(v_{r,h})\cap f^{\bar{\kappa}-h_r}$;
\item [(c2)] $u^0_{r,h}$\ and\ $u^1_{r,h}$ are not $\bar{\kappa}$-connectible in $\mathcal{B}_{\bar{r}}(\bar{v})$;\
\item [(c3)] $\mbox{range}\ga_{r,h}\subset \overline{\mathcal{B}}_{r+d_r}(v_{r,h})\cap f^{\bar{\kappa}+h}$;
\item [(c4)] $\mbox{range}\ga_{r,h}\cap\overline{\mathcal{A}}_{r-d_r,r+d_r}(v_{r,h})\subset f^{\bar{\kappa}-h_r}$.
\end{itemize}
\end{lemma}

 \begin{proof}
In \eqref{Ch2f48},\ $\bar{\ga}(s_0),\bar{\ga}(s_1)\in f^{\kappa-h_{\bar{r}}}$,\ and $f$ is $C^1$. Hence, there is
$\de\in(0,d_{\bar{r}})$\ such that\ $\mathcal{B}_{\de}(\bar{\ga}(s_0))\cup\mathcal{B}_{\de}(\bar{\ga}(s_1))\subset f^{\kappa-\frac{1}{2}h_{\bar{r}}}$.\
Introduce a smooth function $\chi:l^2\to\RR$\ such that
\begin{itemize}
\item [(i)]\ $0\leq\chi\leq1$;
\item [(ii)] $\chi(u)=0$ for all $u\in\mathcal{B}_{\frac{\de}{2}}(\bar{\ga}(s_0))\cup\mathcal{B}_{\frac{\de}{2}}(\bar{\ga}(s_1))$;
\item [(iii)] $\chi(u)=1$ for all $u\not\in\mathcal{B}_{\de}(\bar{\ga}(s_0))\cup\mathcal{B}_{\de}(\bar{\ga}(s_1))$.
\end{itemize}
For any fixed $r\in(0,\frac{1}{2}d_{\bar{r}})\setminus D^{\kappa^*}$\ and\ $h\in(0,\kappa+\bar{h}-\bar{\kappa})$,\ we could construct a vector field $\mathcal{Y}_{r,h}$\ satisfying the following assumptions:
\begin{itemize}
\item $ D_*f(u)(\mathcal{Y}_{r,h}(u))\leq0, u\in l^2$;
\item $\|\mathcal{Y}_{r,h}(u)\|_{*}\leq\frac{2}{\|D_*f(u)\|_{*}}, u\in l^2\setminus(\mathcal{C}\cup\{0\})$;
\item $\mathcal{Y}_{r,h}(u)=0$, $u\in\mathcal{C}^{\bar{\kappa}+h}$;\ and\
\beq\label{Ch2k2}
\begin{split}
 D_*f(u)(\mathcal{Y}_{r,h}(u))\leq&-1,\ \mbox{where}\ u\in [f^{\bar{\kappa}+h}_{\bar{\kappa}-h}\setminus\mathcal{B}_{r-2d_r}(\mathcal{C}^{\bar{\kappa}+h}_{\bar{\kappa}-h})]\\
&\cup[f^{\bar{\kappa}+h}\cap
\mathcal{A}_{r-2d_r,r+2d_r}(\mathcal{C}^{\bar{\kappa}+h}_{\bar{\kappa}-h})]\cup
[f^{\bar{\kappa}+h}\cap\mathcal{A}_{\bar{r}-2d_{\bar{r}},\bar{r}+2d_{\bar{r}}}(\bar{v})].
\end{split}
\eeq
\end{itemize}
Consider the vector field $\bar{\mathcal{Y}}_{r,h}=\chi\mathcal{Y}_{r,h}$.\

Since $h<\bar{h}<\frac{1}{2}h_{\bar{r}}$ and $\mathcal{B}_{\de}(\bar{\ga}(s_0))\cup\mathcal{B}_{\de}(\bar{\ga}(s_1))\subset f^{\kappa-\frac{1}{2}h_{\bar{r}}}$, one has that $\mathcal{B}_{\de}(\bar{\ga}(s_0))\cup\mathcal{B}_{\de}(\bar{\ga}(s_1))\cap f_{\bar{\kappa}-h}=\emptyset$. This implies that
\beq\label{Ch2k3}
\bar{\mathcal{Y}}_{r,h}=\mathcal{Y}_{r,h},\ \forall u\in f^{\bar{\kappa}+h}_{\bar{\kappa}-h}.
\eeq
On the other hand, because $\bar{\ga}(s_0),\bar{\ga}(s_1)\in\partial\mathcal{B}_{\bar{r}}(\bar{v})$ and $\de\in(0,d_{\bar{r}})$, one has
\beq\label{Ch2k4}
[\mathcal{B}_{\de}(\bar{\ga}(s_0))\cup\mathcal{B}_{\de}(\bar{\ga}(s_1))]\subset\mathcal{A}_{\bar{r}-d_{\bar{r}},\bar{r}+d_{\bar{r}}}(\bar{v}).
\eeq
Further, it follows from $\mathcal{C}^{\kappa^*}\cap\mathcal{A}_{\bar{r}-3d_{\bar{r}},\bar{r}+3d_{\bar{r}}}(\bar{v})=\emptyset$\ and\ $r+2d_r<d_{\bar{r}}$ that
\beq\label{Ch2k5}
\mathcal{B}_{r+2d_r}(\mathcal{C}^{\kappa^*})\cap\mathcal{A}_{\bar{r}-2d_{\bar{r}},\bar{r}+2d_{\bar{r}}}(\bar{v})=\emptyset.
\eeq
By \eqref{Ch2k4}\ and\ \eqref{Ch2k5},
\beqq
[\mathcal{B}_{\de}(\bar{\ga}(s_0))\cup\mathcal{B}_{\de}(\bar{\ga}(s_1))]\cap\mathcal{A}_{r-2d_r,r+2d_r}(\mathcal{C}^{\kappa^*})=\emptyset.
\eeqq
So,
\beq\label{Ch2k6}
\bar{\mathcal{Y}}_{r,h}=\mathcal{Y}_{r,h},\ u\in f^{\bar{\kappa}+h}\cap\mathcal{A}_{r-2d_r,r+2d_r}(\mathcal{C}^{\bar{\kappa}+h}_{\bar{\kappa}-h}).
\eeq

By \eqref{Ch2k2}, \eqref{Ch2k3}, and \eqref{Ch2k6}, the vector field $\bar{\mathcal{Y}}_{r,h}$ satisfies the similar properties of the vector field in Lemma \ref{Ch2vect2}, they all satisfy \eqref{Ch2k1},\ where $b=\bar{\kappa}$.\

By the definition of the vector field $\bar{\mathcal{Y}}_{r,h}$, it satisfies the assumptions of Lemma \ref{Ch2vecfield}. Hence, we could investigate the corresponding initial value problem, and assume the corresponding solution is $\eta_{r,h}$. Set
$$\bar{t}:=\max\{3h_r,3h_{\bar{r}}\},\ \psi_{r,h}(u):=\eta_{r,h}(\bar{t},u).$$
Pick $\ga\in\bar{\mathcal{K}}$ satisfying $\max_{t\in[0,1]}f(\ga(t))\leq\bar{\kappa}+h$,\ where\ $\bar{\mathcal{K}}$ is specified in \eqref{Ch2f48}.
Set
\beqq
\ga_{r,h}:=\psi_{r,h}\circ\ga.
\eeqq

Next, it is to show $\ga_{r,h}\in\bar{\mathcal{K}}$.

It is evident that $\ga_{r,h}\in C([0,1],l^2)$, $\mbox{range}\ga_{r,h}\subset f^{\bar{\kappa}+h}$.\ By the (ii) of the definition of the function $\chi$, one has that
$\bar{\mathcal{Y}}_{r,h}(\bar{\ga}(s_0))=\bar{\mathcal{Y}}_{r,h}(\bar{\ga}(s_1))=0$.\ Hence,\ $\ga_{r,h}(0)=\bar{\ga}(s_0)$,\ $\ga_{r,h}(1)=\bar{\ga}(s_1)$.\

Next, it is to show that $\mbox{range}\ga_{r,h}\subset\mathcal{B}_{\bar{r}}(\bar{v})\cup f^{\kappa-\frac{1}{2}h_{\bar{r}}}$.\ For any $t\in[0,1]$, if $\ga(t)\in f^{\kappa-\frac{1}{2}h_{\bar{r}}}$,\ by (a1) of Lemma \ref{Ch2vect2}, one has that $\ga_{r,h}(t)\in f^{\kappa-\frac{1}{2}h_{\bar{r}}}$.\ Hence, it suffices to study the following situation:
$\ga(t)\in \mathcal{B}_{\bar{r}}(\bar{v})\setminus f^{\kappa-\frac{1}{2}h_{\bar{r}}}$ and $\ga_{r,h}(t)\not\in \mathcal{B}_{\bar{r}}(\bar{v})$. For this situation, we need to show that $\ga_{r,h}(t)\in f^{\kappa-\frac{1}{2}h_{\bar{r}}}$. If there is some $t\in[0,\bar{t}]$ such that $\eta_{r,h}(t,u)\in \mathcal{B}_{\de}(\bar{\ga}(s_0))\cup\mathcal{B}_{\de}(\bar{\ga}(s_1))$, since $\mathcal{B}_{\de}(\bar{\ga}(s_0))\cup\mathcal{B}_{\de}(\bar{\ga}(s_1))\subset f^{\kappa-\frac{1}{2}h_{\bar{r}}}$, one has $\ga_{r,h}(t)\in f^{\kappa-\frac{1}{2}h_{\bar{r}}}$.\ Hence, we need to consider the situation that for any $t\in[0,\bar{t}]$, $\eta_{r,h}\not\in \mathcal{B}_{\de}(\bar{\ga}(s_0))\cup\mathcal{B}_{\de}(\bar{\ga}(s_1))$. There are two different situations:
\begin{itemize}
\item [(i)]  $\eta_{r,h}(t,u)\in \mathcal{A}_{\bar{r}-2d_{\bar{r}},\bar{r}+2d_{\bar{r}}}(\bar{v})$ for all $t\in[0,\bar{t}]$;
\item [(ii)] there is $t_0\in[0,\bar{t}]$\ such that\ $\eta_{r,h}(t_0,u)\not\in \mathcal{A}_{\bar{r}-2d_{\bar{r}},\bar{r}+2d_{\bar{r}}}(\bar{v})$.
\end{itemize}
For Case (i), by \eqref{Ch2k2},\
\beqq
\begin{split}
&f(\psi_{r,h}(u))=f(u)+\int^{\bar{t}}_0 D_*f(\eta(s,u))(\mathcal{Y}(\eta(s,u)))ds\leq f(u)-3h_{\bar{r}}\\
\leq&\bar{\kappa}+h-3h_{\bar{r}}
<\bar{\kappa}+\kappa+\bar{h}-\bar{\kappa}-3h_{\bar{r}}<\kappa+\frac{1}{2}h_{\bar{r}}-3h_{\bar{r}}<\kappa-\frac{1}{2}h_{\bar{r}}.
\end{split}
\eeqq
For Case (ii), there are $0\leq t_1<t_2\leq \bar{t}$ satisfying $\|\eta(t_2,u)-\eta(t_1,u)\|_{*}=d_{\bar{r}}$,\ $\eta(t,u)\in\mathcal{A}_{\bar{r}-2d_{\bar{r}},\bar{r}+2d_{\bar{r}}}(\bar{v})$,\ where\
$t\in(t_1,t_2)$.\ Hence,
\beqq
d_{\bar{r}}=\|\eta(t_2,u)-\eta(t_1,u)\|_{*}\leq\int^{t_2}_{t_1}\|\mathcal{Y}(\eta(t,u))\|_{*}dt\leq\int^{t_2}_{t_1}\frac{2}{\|D_*f(\eta(t,u))\|_{*}}dt,
\eeqq
this, together with Lemma \ref{Ch2diffposi}, implies that
\beqq
d_{\bar{r}}\leq\frac{2}{\mu_{\bar{r}}}(t_2-t_1),
\eeqq
yielding that
\beqq
t_2-t_1\geq 2h_{\bar{r}}.
\eeqq
This, together with \eqref{Ch2k1}, implies that
\beqq
\begin{split}
f(\psi(u))&\leq f(\eta(t_2,u))=f(\eta(t_1,u))+\int^{t_2}_{t_1} D_*f(\eta(t,u))(\mathcal{Y}(\eta(t,u))) dt\\
\leq& f(u)-(t_2-t_1)\leq \bar{\kappa}+h-2h_{\bar{r}}<\bar{\kappa}+(\kappa+\bar{h}-\bar{\kappa})-2h_{\bar{r}}\\
\leq&\kappa+\frac{1}{2}h_{\bar{r}}-2h_{\bar{r}}<\kappa-\frac{1}{2}h_{\bar{\kappa}}.
\end{split}
\eeqq
Hence, $\ga_{r,h}\in\bar{\mathcal{K}}$.\

Using the method of the proof Lemma \ref{Ch2vect2} and the choice of $h$, we could show that $\psi_{r,h}:f^{\bar{\kappa}+h}\to f^{\bar{\kappa}+h}$ and satisfies that
\begin{itemize}
\item $f(\psi_{r,h}(u))\leq f(u)$ for any $u\in f^{\bar{\kappa}+h}$;
\item if $\psi_{r,h}(u)\not\in\mathcal{B}_r(\mathcal{C}^{\bar{\kappa}+h}_{\bar{\kappa}-h})$, then $f(\psi_{r,h}(u))<\bar{\kappa}-h$;
\item if $\psi_{r,h}(u)\in\mathcal{A}_{r-d_r,r+d_r}(\mathcal{C}^{\bar{\kappa}+h}_{\bar{\kappa}-h})$, then $f(\psi_{r,h}(u))<\bar{\kappa}-h_r$.
\end{itemize}
Furthermore, by the method used in the proof Lemma \ref{Ch2path2}, one has that for any given $h$\ and\ $\ga_{r,h}$, there are finite number critical points $v_1,...,v_q\in\mathcal{C}^{\bar{\kappa}+h}_{\bar{\kappa}-h}$ such that,
\begin{itemize}
\item\ $\max_{t\in[0,1]}f(\ga_{r,h}(t))<\bar{\kappa}+h$;
\item if $f(\ga_{r,h}(s))\geq \bar{\kappa}-h$, then $\ga_{r,h}(s)\in\cup^q_{j=1}\mathcal{B}_r(v_j)$;
\item if $\ga_{r,h}(s)\in\cup^q_{j=1}\mathcal{A}_{r-d_r,r+d_r}(v_j)$, then $f(\ga_{r,h}(s))<\bar{\kappa}-h_r$.
\end{itemize}
By \eqref{Ch2k5}, there is a part of $\ga_{r,h}$ satisfying the above (c1)--(c4). Reparametrize that part of $\ga_{r,h}$, we could show that (c1)--(c4) hold.
\end{proof}

\begin{lemma} \label{Ch2critpoint}
 There exist sequences
$\{r_n\}^{\infty}_{n=1}\subset\RR$, $\{v_n\}^{\infty}_{n=1}\subset\mathcal{C}(\bar{\kappa})$ with
$\overline{\mathcal{B}}_{r_n}(v_n)\subset\mathcal{B}_{\bar{r}}(\bar{v})$, $\lim_{n\to\infty}r_n=0$, and $\lim_{n\to\infty}v_n=v_{\infty}\in\mathcal{C}(\bar{\kappa})$, so that for any $n$ and $h>0$, there exists a path $\ga\in C([0,1],l^2)$ such that
\begin{itemize}
\item [(d1)] $\ga(0),\ga(1)\in\partial\mathcal{B}_{r_n}(v_n)\cap f^{\bar{\kappa}-\frac{h_{r_n}}{2}}$;
\item [(d2)] $\ga(0)$ and $\ga(1)$ are not $\bar{\kappa}$-connectible in $\mathcal{B}_{\bar{r}}(\bar{v})$;
\item [(d3)] $\mbox{range}\ga\subset\overline{\mathcal{B}}_{r_n}(v_n)\cap f^{\bar{\kappa}+h}$;
\item [(d4)] $\mbox{range}\ga\cap\mathcal{A}_{r_n-\frac{d_{r_n}}{2},r_n}(v_n)\subset f^{\bar{\kappa}-\frac{h_{r_n}}{2}}$;
\item [(d5)] the support of $\ga(\theta)$ is contained in $[-N_n,N_n]$ for any $\theta\in[0,1]$, where $N_n$ is a positive integer independent on the choice of $\theta$.
\end{itemize}
Consequently, there exists a critical point of local mountain pass type for the functional $f$ in $\mathcal{B}_{\bar{r}}(\bar{v})$.
\end{lemma}
\begin{proof}
Fix any $r\in(0,d_{\bar{r}}/2)\setminus D^{\kappa^*}$, there exists a sequence $\{h_n\}^{\infty}_{n=1}\subset(0,\kappa+\bar{h}-\bar{\kappa})$ with $\lim_{n\to\infty}h_n=0$. For this $r$ and any $h_n$, by Lemma \ref{Ch2path3}, there are $v_{r,h_n}\in\mathcal{C}^{\bar{\kappa}+h_n}_{\bar{\kappa}-h_n}\cap\mathcal{B}_{\bar{r}}(\bar{v})$ and
$\ga_{r,h_n}\in C([0,1],l^2)$ such that
\begin{itemize}
\item [(\~{c}1)] $\ga_{r,h_n}(0)=u^0_{r,h_n},\ga_{r,h_n}(1)=u^1_{r,h_n}\in\partial\mathcal{B}_{r+d_r}(v_{r,h_n})\cap f^{\bar{\kappa}-h_r}$;
\item [(\~{c}2)] $u^0_{r,h_n}$ and $u^1_{r,h_n}$ are not $\bar{\kappa}$-connectible in $\mathcal{B}_{\bar{r}}(\bar{v})$;
\item [(\~{c}3)] $\mbox{range}\ga_{r,h_n}\subset \overline{\mathcal{B}}_{r+d_r}(v_{r,h_n})\cap f^{\bar{\kappa}+h_n}$;
\item [(\~{c}4)] $\mbox{range}\ga_{r,h_n}\cap\overline{\mathcal{A}}_{r-d_r,r+d_r}(v_{r,h_n})\subset f^{\bar{\kappa}-h_r}$.
\end{itemize}

Since $v_{r,h_n}\in\mathcal{C}^{\bar{\kappa}+h_n}_{\bar{\kappa}-h_n}\cap\mathcal{B}_{\bar{r}}(\bar{v})$, one has that $\{v_{r,h_n}\}^{\infty}_{n=1}$ is a PS sequence and the diameter of the set of $\{v_{r,h_n}\}^{\infty}_{n=1}$ is less than $2\bar{r}$. This, together with the fact that $\bar{r}\in(0,D_0/4)\setminus D^{\kappa^*}$ and Lemma \ref{Ch2diamconv}, implies that $\{v_{r,h_n}\}^{\infty}_{n=1}$ has a convergent subsequence. Without loss of the generality, we assume that $\lim_{n\to\infty}v_{r,h_n}=v_r\in\mathcal{C}(\bar{\kappa})$. Since $r\in(0,d_{\bar{r}}/2)\setminus D^{\kappa^*}$, one has
\beq\label{Ch2e3}
\mathcal{B}_{r+2d_r}(\mathcal{C}^{\kappa*})\cap\mathcal{A}_{\bar{r}-2d_{\bar{r}},\bar{r}+2d_{\bar{r}}}(\bar{v})=\emptyset,
\eeq
which implies that $v_r\in\mathcal{C}(\bar{\kappa})\cap\mathcal{B}_{\bar{r}}(\bar{v})$.

For any given positive integer $N$, define the following characteristic function
\beqq
\chi_N(t):=\left\{
  \begin{array}{ll}
    1, & \hbox{if}\ |t|\leq N \\
    0, & \hbox{otherwise.}
  \end{array}
\right.
\eeqq
Set $\bar{\ga}_{r,h_n}=\chi_N\ga_{r,h_n}$. It follows from the fact $\mbox{range}\ga_{r,h_n}$ is compact and \eqref{Ch2f5} that there is a positive integer $N_1$ such that for any $N>N_1$, one has
\beq\label{Ch2e6}
\|\bar{\ga}_{r,h_n}(\tht)-\ga_{r,h_n}(\tht)\|_*<\frac{d_r}{4},\ \forall \tht\in[0,1].
\eeq
This, together with (\~{c}3), yields that, there is positive integer $N_2\geq N_1$ so that for any integer $N\geq N_2$,
\beqq
\mbox{range}\chi_N\ga_{r,h_n}\subset \overline{\mathcal{B}}_{r+2d_r}(v_{r,h_n})\cap f^{\bar{\kappa}+2h_n}.
\eeqq

\begin{itemize}
\item [(\^{c}1)] $\chi_N\ga_{r,h_n}(0),\chi_N\ga_{r,h_n}(1)\in\mathcal{A}_{r+\frac{3d_r}{4},r+\frac{5d_r}{4}}(v_{r,h_n})\cap f^{\bar{\kappa}-\frac{h_{r}}{2}}$;
\item [(\^{c}2)] $\chi_N\ga_{r,h_n}(0)$ and $\chi_N\ga_{r,h_n}(1)$ are not $\bar{\kappa}$-connectible in $\mathcal{B}_{\bar{r}}(\bar{v})$;
\item [(\^{c}3)] $\mbox{range}\chi_N\ga_{r,h_n}\subset \overline{\mathcal{B}}_{r+\frac{5d_r}{4}}(v_{r,h_n})\cap f^{\bar{\kappa}+2h_n}$;
\item [(\^{c}4)] $\mbox{range}\chi_N\ga_{r,h_n}\cap\overline{\mathcal{A}}_{r-\frac{3d_r}{4},r+\frac{3d_r}{4}}(v_{r,h_n})\subset f^{\bar{\kappa}-\frac{h_r}{2}}$,
\end{itemize}
where the proof of (\^{c}2) is as follows: suppose to the contrary, by \eqref{Ch2e3}, one has $$\chi_N\ga_{r,h_n}(0),\ \chi_N\ga_{r,h_n}(1)\in\mathcal{B}_{\bar{r}}(\bar{v}).$$
This, together with (\~{c}1), (\^{c}1), and Remark \ref{Ch2e4}, implies that
there is a continuous path contained in $\mathcal{B}_{\bar{r}}(\bar{v})$ joining $u^0_{r,h_n}$ and $u^1_{r,h_n}$ such that the maximal value of $f$ restricted to the continuous path is no larger than $\bar{\kappa}$, which contradicts (\~{c}2).

Given any $h>0$, there is a positive integer $N_3$ such that for any $n>N_3$, one has
\beq\label{Ch2e5}
\mathcal{B}_{r-\frac{d_r}{2}}(v_{r,h_n})\subset\mathcal{B}_r(v_r),\  \mathcal{A}_{r-\frac{d_r}{2},r+\frac{d_r}{2}}(v_r)\subset\mathcal{A}_{r-\frac{3d_r}{4},r+\frac{3d_r}{4}}(v_{r,h_n}),\ \mbox{and}\ h_n<\frac{h}{2}.
\eeq

Next, we show that $\mbox{range}\chi_N\ga_{r,h_n}\cap \mathcal{B}_{r}(v_r)\neq\emptyset$ for $N>N_2$ and $n>N_3$. By contradiction, suppose $\mbox{range}\chi_N\ga_{r,h_n}\cap \mathcal{B}_{r}(v_r)=\emptyset$. By \eqref{Ch2e5} and (\^{c}3),  $\mbox{range}\chi_N\ga_{r,h_n}\subset\overline{\mathcal{A}}_{r-\frac{d_r}{2},r+\frac{5d_r}{4}}(v_{r,h_n})$. This, together with \eqref{Ch2e6}, yields that
$\mbox{range}\ga_{r,h_n}\subset\overline{\mathcal{A}}_{r-\frac{3d_r}{4},r+\frac{3d_r}{2}}(v_{r,h_n})$. So, it follows from (\~{c}3) and (\~{c}4) that
$\mbox{range}\ga_{r,h_n}\subset\overline{\mathcal{A}}_{r-d_r,r+d_r}(v_{r,h_n})\subset f^{\bar{\kappa}-h_r}$. Now, we arrive at a contradiction.

By the discussions above, if $N>N_2$ and $n>N_3$, then a part of the path $\mbox{range}(\chi_N\ga_{r,h_n})$ satisfies (d1)--(d5).\ Reparametrize this part of curve, we could show the statements hold.

For a sequence $\{r_n\}^{\infty}_{n=1}\subset(0,d_{\bar{r}}/2)\setminus D^{\kappa^*}$ with $\lim_{n\to\infty}r_n=0$, by applying the same discussions above and Lemma \ref{Ch2diamconv}, one could show the statements. This completes the proof.
\end{proof}

\section{Existence of multibump solutions}

In this section, we study the existence of multibump solutions and show the results.

Given $k,N\in\NN$, set
\beq\label{Ch2f49}
\begin{split}
\mathcal{P}(k,N)&:=\{(p_1,...,p_k)\in\ZZ^k:\ p_{i+1}-p_{i}\geq 2N^2+4N,\ 1\leq i\leq k-1;\\
& p_i\ \mbox{is an integer multiple of}\ T,\ 1\leq i\leq k\},
\end{split}
\eeq
where $T$ is specified in (A1).

For any $P=(p_1,...,p_k)\in \mathcal{P}(k,N)$, we introduce the following intervals:
\beq\label{Ch2e7}
I_i:=\bigg(\frac{p_{i-1}+p_i}{2},\frac{p_i+p_{i+1}}{2}\bigg]\cap\ZZ,\ 1\leq i\leq k,
\eeq
\beq\label{Ch2f38}
M_i:=(p_{i}+N(N+1),p_{i+1}-N(N+1)]\cap\ZZ,\ 0\leq i\leq k,
\eeq
where $p_0=-\infty$ and $p_{k+1}=+\infty$.

For any subset $F\subset\ZZ$, $u,v\in l^2$, set
\beq\label{Ch2f13}
\lan u,v\ran_{F}:=\sum_{t\in F}(\lan \De u(t-1),\De v(t-1)\ran+\lan u(t),L(t)v(t)\ran),\ \|u\|_F:=\sqrt{\lan u,u\ran_F}.
\eeq

For any $P=(p_1,...,p_k)\in \mathcal{P}(k,N)$, define a functional $f_i:l^2\to\RR$ as follows:
\beq\label{Ch2f23}
f_i(u):=\frac{1}{2}\|u\|^2_{I_i}-\sum_{t\in I_i}V(t,u(t)),\ u\in l^2,\ 1\leq i\leq k.
\eeq
It is evident that
\beq\label{Ch2b4}
\|u\|_*^2=\sum_{i=1}^k\|u\|^2_{I_i},\ f(u)=\sum_{i=1}^kf_i(u).
\eeq
By applying the same method to show that $f$ is $C^1$, one has that
$f_i$ is $C^1$ and
\beq \label{Ch2diff3}
 D_*f_i(u)v=\lan u,v\ran_{I_i}-\sum_{t\in I_i}\lan V'_x(t,u(t)),v(t)\ran,\ 1\leq i\leq k.
\eeq

For any $P=(p_1,...,p_k)\in \mathcal{P}(k,N)$, set
\beq\label{Ch2e10}
\mathcal{B}^{P}_r(v):=\{u\in l^2:\ \|u-v(\cdot-p_i)\|_{I_i}<r,\ 1\leq i\leq k\},
\eeq
\beq\label{Ch2e11}
\mathcal{B}^{P}_{r,b}(v):=\{u\in \mathcal{B}^{P}_r(v):\ f_i(u)\leq b,\ 1\leq i\leq k\}.
\eeq
For any $\ep>0$, define
\beq\label{Ch2e12}
\mathcal{M}_{\ep}:=\{u\in l^2:\ \|u\|^2_{M_i}\leq\ep,\ 0\leq i\leq k\}.
\eeq

It follows from (A3) that there is $r_0\in(0,\min\{D_0/4,1/8\})$ such that for any $F\subset\ZZ$, one has that if $\|u\|_F\leq r_0$, then
\beq\label{Ch2f4}
\sum_{t\in F}V(t,u(t))\leq\frac{1}{8}\|u\|^2_F,\ \ \sum_{t\in F}\lan V'_x(t,u(t)),w(t)\ran\leq\frac{1}{8}\|u\|_F\|w\|_F,
\eeq
where $D_0$ is specified in \eqref{Ch2diam}.

Now, the main results are given as follows.

\begin{theorem}
Assume that (A1)--(A6) hold. Let $v_{\infty}$ be the critical point of $f$ given by Lemma \ref{Ch2critpoint}. Then for any $r>0$, there exists a positive integer $N$ such that for any positive integer $k$ and $P=(p_1,...,p_k)\in \mathcal{P}(k,N)$, one has that $\mathcal{C}\cap\mathcal{B}^P_r(v_{\infty})\neq\emptyset$.
\end{theorem}

We will show the statements by contradiction. There is $\rho\in(0,\bar{r})$ such that for any positive integer $N$, there exist a positive integer $k$
and $(p_1,...,p_k)\in \mathcal{P}(k,N)$ with $\mathcal{C}\cap\mathcal{B}^P_{\rho}(v_{\infty})=\emptyset$, where $\bar{r}$ is specified in \eqref{Ch2k7}. By Lemma \ref{Ch2critpoint}, there are
$\{r_n\}^{\infty}_{n=1}\subset\RR$, $\{v_n\}^{\infty}_{n=1}\subset\mathcal{C}(\bar{\kappa})$ satisfying
$\overline{\mathcal{B}}_{r_n}(v_n)\subset\mathcal{B}_{\bar{r}}(\bar{v})$, $\lim_{n\to\infty}r_n=0$, and $\lim_{n\to\infty}v_n=v_{\infty}\in\mathcal{C}(\bar{\kappa})$. Hence, if $n$ is large enough, then $\|v_n-v_{\infty}\|_*<\rho/2$, $r_n<\rho/2$, and $\mathcal{B}_{2r_n}(v_n)\subset\mathcal{B}_{\bar{r}}(\bar{v})$. So, for any $u\in \mathcal{B}^P_{r_n}(v_n)$ and any $i$, $1\leq i\leq k$, one has
\beqq
\|u-v_{\infty}(\cdot-p_i)\|_{I_i}\leq\|u-v_n(\cdot-p_i)\|_{I_i}+\|v_n(\cdot-p_i)-v_{\infty}(\cdot-p_i)\|_{I_i}\leq r_n+\frac{\rho}{2}<\frac{\rho}{2}+\frac{\rho}{2}=\rho,
\eeqq
which yields that $u\in\mathcal{B}^P_{\rho}(v_{\infty})$. Therefore, $\mathcal{B}_{r_n}^P(v_n)\subset\mathcal{B}_{\rho}^P(v_{\infty})$.

Since $\lim_{n\to\infty}r_n=0$ and $\lim_{n\to\infty}v_n=v_{\infty}\in\mathcal{C}(\bar{\kappa})$, when $n$ is sufficiently large, one has that $r_n\in(0,\min\{\frac{r_0}{8}.\frac{\sqrt{3}\|v_n\|_{*}}{4}\})\setminus D^{\kappa^*}$. Fix these $r_n$ and $v_n$. Let $b=\bar{\kappa}$, it follows from Proposition \ref{globvect} that we have the following results.

\begin{lemma} \label{Ch2vector}
Suppose $n$ is sufficiently large, fix $r_n$ and $v_n$, such that for any positive constants $r_{L},r,r_{R}$ satisfying
\beq\label{Ch2k8}
r_n-\frac{d_{r_n}}{2}\leq r_{L}<r<r_{R}<r_n<\min\bigg\{\frac{r_0}{8}.\frac{\sqrt{3}\|v_n\|_{*}}{4}\bigg\},
\eeq
any positive constants $\rho_{L},\rho_{R},\de$ with $[\rho_{L}-\de,\rho_{L}+2\de]\subset(0,\bar{\kappa})\setminus\Phi^{\kappa^*}$ and $[\rho_{R}-\de,\rho_{R}+2\de]\subset(\bar{\kappa},\kappa^*)\setminus\Phi^{\kappa^*}$, there is $\ep_1>0$, such that for any $\ep\in(0,\ep_1)$,
there are a positive constant $\mu$ and an positive integer $N^*(\ep)$ such that for any positive integer $N>N^*(\ep)$ and $k$ with $P=(p_1,...,p_k)\in \mathcal{P}(k,N)$, there exists a locally Lipschitz continuous vector field $W:l^2\to l^2$ satisfying the following conditions:
\begin{itemize}
\item [(I)] $ D_*f(u)W(u)\geq0$, $\|W(u)\|_{I_i}\leq 2$, $1\leq i\leq k$, for any $u\in l^2$; if $u\in l^2\setminus \mathcal{B}_{r_{R}}^P(v_n)$, then $W(u)=0$;
\item [(II)] if for some $i$, $1\leq i\leq k$, $u\in\mathcal{B}^{P}_{r,\rho_{R}+\de}(v_n)$ and $r_{L}\leq\|u-v_n(\cdot-p_i)\|_{I_i}\leq r$, then
 $$ D_*f_i(u)W(u)\geq\mu;$$
\item [(III)] $ D_*f_i(u)W(u)\geq0$, for any $u\in(f^{\rho_{R}+\de}_i\setminus f^{\rho_{R}}_i)\cup(f^{\rho_{L}+\de}_i\setminus f^{\rho_{L}}_i)$, $1\leq i\leq k$;
\item [(IV)] $\lan u,W(u)\ran_{M_i}\geq0$, if $u\in l^2\setminus\mathcal{M}_{4\ep}$, $0\leq i\leq k$;
\item [(V)]  if $\mathcal{C}\cap\mathcal{B}^P_{r_{L}}(v_n)=\emptyset$, then there is $\mu_k>0$ so that for any $u\in\mathcal{B}^P_{r_{L}}(v_n)$,
               $$ D_*f(u)W(u)\geq\mu_k.$$
\end{itemize}

\end{lemma}

In the proof of Proposition \ref{globvect}, the choice of the constant of $\mu$ is dependent on $r_n$, $d_{r_n}$, $r_L$, and $r_R$, and is not dependent on
$\rho_L$, $\rho_R$ and $\de$. So, we should choose the constants such that
\beq\label{Ch2k9}
\rho_{L}\in\bigg(\bar{\kappa}-\frac{1}{4}\min\{h_{r_n},\mu(r-r_{L})\},\bar{\kappa}\bigg),\ \rho_{R}\in\bigg(\bar{\kappa},\min\bigg\{\kappa^*,\bar{\kappa}+\frac{1}{4}\mu(r-r_{L})\bigg\}\bigg).
\eeq

It follows from Lemma \ref{Ch2critpoint} that there exists $\ga\in C([0,1],l^2)$
such that
\begin{itemize}
\item [(e1)] $\ga(0),\ga(1)\in\partial\mathcal{B}_{r_n}(v_n)\cap f^{\bar{\kappa}-\frac{h_{r_n}}{2}}$;
\item [(e2)] $\ga(0)$ and $\ga(1)$ are not $\bar{\kappa}$-connectible in $\mathcal{B}_{\bar{r}}(\bar{v})$;
\item [(e3)] $\mbox{range}\ga\subset\overline{\mathcal{B}}_{r_n}(v_n)\cap f^{\rho_{R}}$;
\item [(e4)] $\mbox{range}\ga\cap\mathcal{A}_{r_n-\frac{d_{r_n}}{2},r_n}(v_n)\subset f^{\bar{\kappa}-\frac{h_{r_n}}{2}}$;
\item [(e5)] the support of $\ga(s)$ is contained in $[-N_{\ga},N_{\ga}]$ for any $s\in[0,1]$, where $N_{\ga}$
is a positive integer independent on the choice of $s$.
\end{itemize}

Now, we choose $\ep$ such that
\beq\label{Ch2epb}
0<\ep<\min\bigg\{1,\ \ep_1,\ \frac{r_0}{2},\ \frac{\bar{\kappa}-\rho_{L}}{41},\ \frac{d_{r_n}}{3},\ \frac{d^2_{r_n}}{4},\ \frac{4}{5}(r_n^2-r_{R}^2),\ \frac{r_n^2}{4},\ \frac{4}{3}(r_n-r_{R})\bigg\},
\eeq
where $\ep_1$ is given by Lemma \ref{Ch2vector}.

For this given $\ep$, there is a positive integer $\hat{N}$ such that
\beq
\|v_n\|_{|t|\geq \hat{N}}^2\leq\ep.
\eeq
Set
\beq\label{Ch2f50}
N_0:=\max\bigg\{\bigg[\frac{5}{L_2}\bigg]+1,\ 5,\ N^*,\ \hat{N},\ N_{\ga}\bigg\}+1,
\eeq
where $N^*$ is introduced in Lemma \ref{Ch2vector}, $L_2$ is specified in \eqref{Ch2f20} and $[x]$ is the greatest integer function, which gives the largest integer less than or equal to $x$.

We assume that $N$ is bigger than twice of $N_0$, where $N$ is introduced in \eqref{Ch2f49}. Consider the sets introduced in \eqref{Ch2f49}, \eqref{Ch2e7}, and \eqref{Ch2f38}.

If $\mathcal{C}\cap\mathcal{B}^P_{r_n}(v_n)=\emptyset$, then there exists a locally Lipschitz continuous vector field
$W:l^2\to l^2$ which satisfies (I)-(V) of Lemma \ref{Ch2vector}. Consider the following initial value problem
\beqq\label{Ch2f59}
\left\{
  \begin{array}{ll}
    \frac{d\eta}{ds}(s,u)=-W(\eta(s,u)); \\
    \eta(0,u)=u,
  \end{array}
\right.
\eeqq
It follows from (I) of Lemma \ref{Ch2vector} that $\|W(u)\|_*\leq 2k$ for all $u\in l^2$. Therefore, there exists a unique solution
$\eta(\cdot,u)$ such that the solution is defined on the whole real line for all $u\in l^2$.

We introduce a function $G:L=[0,1]^k\to l^2$ as follows:
\beq\label{Ch2k12}
G(\theta):=\sum^k_{i=1}\ga(\theta_i)(\cdot-p_i),
\eeq
where $\theta=(\theta_1,...,\theta_k)\in [0,1]^k$.
The boundary of $L$ is equal to $\cup^k_{i=1}(L^0_i\cup L^1_i)$, where $L^0_i=\{\tht\in L:\ \tht_i=0\}$ and $L^1_i=\{\tht\in L:\ \tht_i=1\}$.
From (e5) of the properties of $\ga$, it follows that $G(\tht)|_{I_i}=\ga(\tht_i)(\cdot-p_i)|_{I_i}$ and the support of $\ga(\tht_i)(\cdot-p_i)$ is contained in
$[-N+p_i,N+p_i]\subset I_i\setminus(M_i\cup M_{i-1})$. As a consequence, one has
\beqq
f_i(G(\tht))=f(\ga(\tht_i)),\ \forall 1\leq i\leq k,\ \tht=(\tht_1,...,\tht_k)\in L.
\eeqq

\begin{lemma}\label{Ch2pathg}
There exists $\tau>0$ such that the function $\bar{G}(\tht)=\eta(\tau,G(\tht))$ is a continuous function from $L$ to $l^2$, which satisfies the following properties:
\begin{itemize}
\item [(g1)] $\bar{G}=G$ on the boundary of $L$;
\item [(g2)] $\bar{G}(\tht)\in \mathcal{M}_{4\ep}$, $\forall \tht\in L$;
\item [(g3)] there exist $j$, $1\leq j\leq k$, and a path $\beta\subset L$ such that $\beta(s)=(\beta_1(s),...,\beta_k(s))\in C([0,1],L)$ with $\beta_j(0)=0$ and $\beta_j(1)=1$ satisfying $\bar{G}(\beta(s))\in f^{\rho_{L}+\ep}_j$ for any $s\in[0,1]$.
\end{itemize}
\end{lemma}

\begin{proof}

First, it is to show (g1).

If $\tht$ is on the boundary of $L$, then there exists $j$, $1\leq j\leq k$, such that either $\tht_j=0$ or $\tht_j=1$. Without loss of generality,
assume that $\tht_j=0$. It follows from (e5), \eqref{Ch2f50}, and $N>2N_0$ that  $G(\tht)|_{I_j}=\ga(0)(\cdot-p_j)|_{I_j}$. This, together with \eqref{Ch2epb} and the support of $\ga(\tht_j)(\cdot-p_j)$ is contained in
$[-N+p_j,N+p_j]\subset I_j\setminus(M_j\cup M_{j-1})$, yields that
\beqq
\|G(\tht)-v_n(\cdot-p_j)\|^2_{I_j}=\|\ga(0)-v_n\|_*^2-\| v_n(\cdot-p_j)\|^2_{\ZZ\setminus I_j}\geq r^2_n-\ep\geq r_{R}^2,
\eeqq
where \eqref{Ch2epb} is used. So, $G(\tht)\in l^2\setminus \mathcal{B}_{r_{R}}^{P}(v)$. This, together with (I) of Lemma \ref{Ch2vector}, implies that
$G=\bar{G}$ on the boundary of $L$.

Next, it is to show (g2). This can be proved by the method used in the proof of (vii) in \cite{CalMon}. For the completeness and the convenience, the proof is given here.

Since $\mbox{supp}(G(\tht))\subset\cup^k_{i=1}I_i\setminus(M_i\cup M_{i-1})$, one has $\|G(\tht)\|_{M_i}=0$, $0\leq i\leq k$.\ Hence,\ $G(\tht)\in\mathcal{M}_{4\ep}$.\
To show (g2), it suffices to show that the set $\mathcal{M}_{4\ep}$ is forward invariant under the flow $\eta(\cdot,u)$, that is,  $\eta(t,\mathcal{M}_{4\ep})\subset\mathcal{M}_{4\ep}$. We show this by contradiction, suppose that there are $u\in\mathcal{M}_{4\ep}$,\ $(t_1,t_2)\subset\RR$,\ $0\leq j\leq k$\ such that\ $\|\eta(t_1,u)\|^2_{M_j}=16\ep$, and for any $t\in(t_1,t_2)$, $\|\eta(t,u)\|^2_{M_j}>16\ep$.\ So,\ by (IV) of Lemma \ref{Ch2vector},
\beqq
\frac{d}{dt}\|\eta(t,u)\|^2_{M_j}=-2\lan \eta(t,u), W(\eta(t,u))\ran_{M_j}\leq0,
\eeqq
which yields that
\beqq
\|\eta(t_2,u)\|^2_{M_j}\leq\|\eta(t_1,u)\|^2_{M_j}=16\ep.
\eeqq
Now, we arrive at a contradiction. Hence,\ (g2) holds

Now, we show (g3). We split it into several steps.

{\bf Step 1.} The functional $f$ sends bounded sets into bounded sets.

Consider a bounded set $A=\{u:\ u\in l^2,\ \|u\|_*\leq B_1\}$. It is evident that there is a positive constant $B_2$ such that $|u(t)|\leq B_2$ for any $u\in A$ and any $t\in\ZZ$. By \eqref{Ch2c3}, there is $\de>0$ so that $|V(t,x)|\leq|x|^2$ for any $|x|\leq\de$. So, for any $u\in A$, by \eqref{Ch2b2} and \eqref{Ch2c2}, one has
\beqq
\begin{split}
|f(u)|=&\bigg|\frac{1}{2}\|u\|_*^2-\sum_{t\in\ZZ}V(t,u(t))\bigg|\leq\frac{B^2_1}{2}+\sum_{t\in\ZZ,|u(t)|\leq\de}V(t,u(t))+\sum_{t\in\ZZ,|u(t)|>\de}V(t,u(t))\\
\leq&\frac{B^2_1}{2}+\sum_{t\in\ZZ,|u(t)|\leq\de}|u(t)|^2+B_3\sum_{t\in\ZZ,|u(t)|>\de}\frac{|u(t)|^2}{\de^2}\\
\leq&\frac{B^2_1}{2}+L_1\|u\|_*^2+\frac{B_3}{\de^2}L_1\|u\|_*^2=\bigg(L_1+\frac{B_3}{\de^2}L_1+\frac{1}{2}\bigg)B^2_1,
\end{split}
\eeqq
where $B_3=\max_{t\in\ZZ,|x|\leq B_2}|V(t,x)|$, by (A1), it is a finite number

The statements of Steps 2--4 can be derived by the method used in the proof of Lemmas 5.4--5.6 in \cite{CalMon}. For the convenience and completeness, the whole proof is provided here.

{\bf Step 2.} For any $i$, $1\leq i\leq k$, and any $t\geq0$, one has that $\eta(t,f^{\rho_{R}}_i)\subset f^{\rho_{R}}_i$ and $\eta(t,f^{\rho_{L}}_i)\subset f^{\rho_{L}}_i$.

By contradiction, suppose that there exist $\bar{t}\in\RR^+$ and $u\in f^{\rho_{R}}_i$ such that $\eta(\bar{t},u)\not\in f^{\rho_{R}}_i$.\ Hence, there is $(t_1,t_2)\subset[0,\bar{t}]$ such that
$f_i(\eta(t_1,u))=\rho_{R}$, $f_i(\eta(t_2,u))>\rho_{R}$, and for any $t\in(t_1,t_2)$,\ $\eta(t,u)\in f^{\rho_{R}+\de}_i\setminus f^{\rho_{R}}_i$.\ This, together with (I) and (IV) of Lemma \ref{Ch2vector}, implies that
\beqq
f_i(\eta(t_2,u))-\rho_{R}=-\int^{t_2}_{t_1} D_*f_i(\eta(t,u))(W(\eta(t,u))) dt\leq0.
\eeqq
This is a contradiction. Similarly, we could show that $\eta(t,f^{\rho_{L}}_i)\subset f^{\rho_{L}}_i$.\

{\bf Step 3.} There exists $\tau>0$ such that for any $u\in\mathcal{B}^{P}_{r_{L},\rho_{R}}(v_n)$, there is $j$,\ $1\leq j\leq k$, such that
$\eta(\tau,u)\in f^{\rho_{L}}_j$,\ $j$ is dependent on the choice of $u$,\ $\tau$ does not depend on the choice of $u$.

Set $\lambda:=2f(\mathcal{B}^{P}_{r_{L},\rho_{R}}(v_n))$. Since $\mathcal{B}^{P}_{r_{L},\rho_{R}}(v_n)$ is a bounded set and the conclusion of Step 1, one has that $\ld<\infty$. Denote $\omega:=\min\{\mu,\mu_k\}$,\ $\tau:=\frac{\ld}{\om}$.\

By the conclusion of Step 2, for any $u\in\mathcal{B}^{P}_{r_{L},\rho_{R}}(v_n)$ and any $t\geq0$, $\eta(t,u)\in\cap^k_{i=1}f^{\rho_{R}}_i$. Next, it is to show that there is $\bar{t}\in(0,\tau)$ such that $\eta(\bar{t},u)\not\in\mathcal{B}^{P}_{r}(v_n)$. By contradiction, by (II)\ and\ (V) of Lemma \ref{Ch2vector},
\beqq
f(\eta(\tau,u))-f(u)=-\int^{\tau}_{0} D_*f(\eta(t,u))(W(\eta(t,u))) dt\leq-\tau\om=-\ld.
\eeqq
Hence,\ $f(u)-f(\eta(\tau,u))\geq\ld$.\ This contradicts the definition of $\ld$.

Hence, for any $u\in\mathcal{B}^{P}_{r_{L},\rho_{R}}(v_n)$, there are $j$,\ $1\leq j\leq k$, and $[t_1,t_2]\subset(0,\tau)$\ such that\ $\|\eta(t_1,u)-v(\cdot-p_j)\|_{I_j}=r_{L}$,\
$\|\eta(t_2,u)-v(\cdot-p_j)\|_{I_j}=r$,\ and for any $t\in(t_1,t_2)$,\ $r_{L}<\|\eta(t,u)-v(\cdot-p_j)\|_{I_j}<r$.\
By Step 2, for any $t\geq0$,\ $\eta(t,u)\in f^{\rho_{R}}_j$. This, together with (II) of Lemma \ref{Ch2vector}, yields that
\beq\label{Ch2k10}
f_j(\eta(t_2,u))\leq f_j(\eta(t_1,u))-\int^{t_2}_{t_1} D_*f_j(\eta(t,u))(W(\eta(t,u))) dt\leq\rho_{R}-\mu(t_2-t_1).
\eeq
It follows from $\|W(\eta(t,u))\|_{I_j}\leq2$ that
\beq\label{Ch2k11}
r-r_{L}\leq\|\eta(t_2,u)-\eta(t_1,u)\|_{I_j}\leq\int^{t_2}_{t_1}\|W(\eta(t,u))\|_{I_j}dt\leq2(t_2-t_1),
\eeq
By \eqref{Ch2k9}, \eqref{Ch2k10}, and \eqref{Ch2k11},
\beqq
f_j(\eta(t_2,u))\leq\rho_{R}-\frac{1}{2}\mu(r-r_{L})<\rho_{L}.
\eeqq
By Step 2, for any $t\geq t_2$,\ $\eta(t,u)\in f^{\rho_{L}}_j$.\ In particular,\ $\eta(\tau,u)\in f^{\rho_{L}}_j$.\

{\bf Step 4.} For any $\tht\in L$, there is $i$, $1\leq i\leq k$, such that $f_i(\bar{G}(\tht))\leq\rho_{L}$.

First, it is to study the situation that $G(\tht)\in\mathcal{B}_{r_{L}}^P(v_n)$. By (e3) of the properties of $\ga$ and the construction of $G$ in \eqref{Ch2k12}, one has that $G(\tht)\in\cap^k_{j=1}f^{\rho_{R}}_j$.\ This, together with the conclusion of Step 3, yields that the statement holds in this situation, where the constant $\tau$ is specified in Step 3 and $\bar{G}(\tht)=\eta(\tau,G(\tht))$.

Second, it is to consider the situation that $G(\tht)\not\in\mathcal{B}_{r_{L}}^P(v_n)$. There is $i$, $1\leq i\leq k$, such that,
\beqq
r_n-\frac{1}{2}d_{r_n}\leq r_{L}\leq\|G(\tht)-v_n(\cdot-p_i)\|_{I_i}=\|\ga(\tht_i)(\cdot-p_i)-v_n(\cdot-p_i)\|_{I_i}\leq\|\ga(\tht_i)-v_n\|_{*},
\eeqq
where \eqref{Ch2k8} is used. This, together with (e4) of $\ga$, yields that,
\beqq
f_i(G(\tht))=f_i(\ga(\tht_i)(\cdot-p_i))=f(\ga(\tht_i))\leq\bar{\kappa}-\frac{1}{2}h_{r_n}\leq \rho_{L}.
\eeqq
By Step 2, for any $t\geq0$,\ $\eta(t,G(\tht))\in f^{\rho_{L}}_i$.\ Hence, the statement of Step 4 holds.

{\bf Step 5.} It is to show (g3).

We show this by contradiction. Set
\beqq
D_i:=(f_i\circ\bar{G})^{-1}([\rho_{L}+\ep,+\infty)),\ 1\leq i\leq k.
\eeqq
So, the set $D_i$ separates the set $L^0_i$ from $L^1_i$ in $L$. Assume $C_i$ is the component of $L\setminus D_i$ which contains $L^1_i$, and set
\beqq
\si_i(\tht):=\left\{
  \begin{array}{ll}
    \mbox{dist}(\tht,D_i), & \hbox{if}\ \tht\in L\setminus C_i, \\
    -\mbox{dist}(\tht,D_i), & \hbox{if}\ \tht\in C_i.
  \end{array}
\right.
\eeqq
So, $\si_i$ is a continuous function on $L$ with $\si_i|_{L^0_i}\geq0$, $\si_i|_{L^1_i}\leq0$, and $\si_i(\tht)=0$ if and only if $\tht\in D_i$.
For the function $\si=(\si_1,...,\si_k)$, by Lemma \ref{Ch2fix}, there exists $\tht\in L$ such that $\si(\tht)=0$, which implies that
$\cap^k_{i=1}D_i\neq\emptyset$. This contradicts the conclusion of Step 4.
\end{proof}

Finally, we give the proof of the main results of this paper.

Now, for the $j$ given in (g3) of Lemma \ref{Ch2pathg}, since $I_j\setminus(M_j\cup M_{j-1})$ is a bounded subset, suppose $a$ and $b$ are the minimal and maximal integers contained in $I_j\setminus(M_j\cup M_{j-1})$, respectively. Set $\tilde{J}:=(I_j\setminus(M_j\cup M_{j-1}))\cup([a-1-N_0,b+1+N_0]\cap\ZZ)$, where $N_0$ is specified in \eqref{Ch2f50}. We introduce a characteristic function on $\tilde{J}$ as follows,
\beqq
\tilde{\chi}_{\tilde{J}}(t)=\left\{
  \begin{array}{ll}
    1, & \hbox{if}\ t\in [a-1,b+1]\cap\ZZ\\
    \frac{N_0-l}{N_0}, & \hbox{if}\ t=b+1+l,\ 1\leq l\leq N_0\\
    \frac{N_0-l}{N_0}, & \hbox{if}\ t=a-1-l,\ 1\leq l\leq N_0\\
    0, & \hbox{otherwise.}
  \end{array}
\right.
\eeqq
So, $|\De \tilde{\chi}_{\tilde{J}}(t)|\leq1/N_0$ for any $t\in\ZZ$. It is evident that $\tilde{J}\subset I_j$, since we assume that $N$ is big enough. Define a path $g=g(s)=\tilde{\chi}_{\tilde{J}}\bar{G}(\be(s))\in C([0,1],l^2)$, $s\in[0,1]$. Since $\mbox{supp}(\ga(s)(\cdot-p_j))\subset I_j\setminus(M_j\cup M_{j-1})$, (g1), and (g2),
one has $g(0)=\tilde{\chi}_{\tilde{J}}\bar{G}(\be(0))=\tilde{\chi}_{\tilde{J}} G(\be(0))=\ga(0)(\cdot-p_j)$, $g(1)=\tilde{\chi}_{\tilde{J}}\bar{G}(\be(1))=\tilde{\chi}_{\tilde{J}} G(\be(1))=\ga(1)(\cdot-p_j)$.

Next, we show that for any subset $F\subset \ZZ$ and any $u\in l^2$, one has
\beq\label{Ch2f51}
\|\tilde{\chi}_{\tilde{J}}u\|^2_{F}\leq 2\|u\|^2_{F\cap I_j}.
\eeq

By direct calculation, $N$ is bigger than twice of $N_0$, one has
\beq\label{Ch2f58}
\begin{split}
&\lan\De(\tilde{\chi}_{\tilde{J}}(t-1)u(t-1)),\De(\tilde{\chi}_{\tilde{J}}(t-1) u(t-1))\ran\\
=&\lan \tilde{\chi}_{\tilde{J}}(t) u(t)-\tilde{\chi}_{\tilde{J}}(t-1) u(t-1),
\tilde{\chi}_{\tilde{J}}(t) u(t)-\tilde{\chi}_{\tilde{J}}(t-1) u(t-1)\ran\\
=&\lan(\De \tilde{\chi}_{\tilde{J}}(t-1)) u(t),(\De \tilde{\chi}_{\tilde{J}}(t-1)) u(t)\ran
+\lan (\De \tilde{\chi}_{\tilde{J}}(t-1)) u(t), \tilde{\chi}_{\tilde{J}}(t-1)(\De u(t-1))\ran\\
+&\lan \tilde{\chi}_{\tilde{J}}(t-1)(\De u(t-1)),(\De \tilde{\chi}_{\tilde{J}}(t-1)) u(t)\ran
+\lan \tilde{\chi}_{\tilde{J}}(t-1)(\De u(t-1)), \tilde{\chi}_{\tilde{J}}(t-1)(\De u(t-1)) \ran\\
\leq&\frac{1}{N_0^2}| u(t)|^2+\frac{2}{N_0}| u(t)||\De u(t-1)|+|\De u(t-1)|^2\\
\leq&\bigg(\frac{1}{N_0^2}+\frac{4}{N_0}\bigg)| u(t)|^2+\bigg(1+\frac{4}{N_0}\bigg)|\De u(t-1)|^2.
\end{split}
\eeq
This, together with \eqref{Ch2f50}, yields that \eqref{Ch2f51} holds.

By similar calculation, for any $u\in l^2$, one has
\beq\label{Ch2f52}
\|(1-\tilde{\chi}_{\tilde{J}})u\|^2_{I_j}\leq 2\|u\|^2_{I_j\cap(M_j\cup M_{j-1})},
\eeq
which implies that
\beq\label{Ch2f53}
\| (1-\tilde{\chi}_{\tilde{J}})v_n(\cdot-p_j)\|^2_{I_j}\leq2\|v_n(\cdot-p_j)\|^2_{I_j\cap(M_j\cup M_{j-1})}\leq2\|v_n\|^2_{|t|\geq \hat{N}}\leq2\ep.
\eeq

Next, we show that the image of $g$ is contained in $B_{\bar{r}}(\bar{v}(\cdot-p_j))$.

It follows from $\mbox{supp}(g(s))\subset I_j$ that
\beq\label{Ch2f54}
\|g(s)-v_n(\cdot-p_j) \|_*^2=\|g(s)-v_n(\cdot-p_j) \|^2_{I_j}+\| v_n(\cdot-p_j)\|^2_{\ZZ\setminus I_j}.
\eeq
And,
\beq\label{Ch2f46}
\| v_n(\cdot-p_j)\|^2_{\ZZ\setminus I_j}=\| v_n\|^2_{\ZZ\setminus (I_j-p_j)}\leq\| v\|^2_{|t|\geq \hat{N}}\leq\ep.
\eeq

Set
\beqq
\Ld_0:=\{s\in[0,1]:\ \|G(\be(s))-v_n(\cdot-p_j)\|_{I_j}>r_{R}\},\ \Ld_1:=[0,1]\setminus\Ld_0.
\eeqq

Now, we split our discussions into two steps:

{\bf Step 1.} We study the case that $s\in\Ld_0$.

Since $\be(s)=(\be_1(s)),...,\be_k(s))$, when $s=0$, $\be_j(0)=0$. So, $G(\be(0))|_{I_j}=\ga(0)(\cdot-p_j)|_{I_j}$,
\beqq
\begin{split}
&\|G(\be(0))-v_n(\cdot-p_j)\|^2_{I_j}=\| \ga(0)(\cdot-p_j)-v_n(\cdot-p_j)\|_*^2-\|\ga(0)(\cdot-p_j)-v_n(\cdot-p_j)\|^2_{\ZZ\setminus I_j}\\
=&\| \ga(0)(\cdot-p_j)-v_n(\cdot-p_j)\|_*^2-\|v_n(\cdot-p_j)\|^2_{\ZZ\setminus I_j}\geq r_n^2-\ep\geq r_n^2-\frac{d^2_{r_n}}{4}
> r_{R}^2,
\end{split}
\eeqq
where (e1) and \eqref{Ch2epb} are used. This yields that $\Ld_0\neq\emptyset$.

For the vector field given by Lemma \ref{Ch2vector}, it follows from (I) that the sets $l^2\setminus \mathcal{B}_{r_{R}}^{P}(v)$ and $\mathcal{B}_{r_{R}}^{P}(v)$ are invariant sets under the flow. So, $\bar{G}(\be(s))=G(\be(s))$ for any $s\in\Ld_0$.
So, $$g(s)=\tilde{\chi}_{\tilde{J}}\bar{G}(\be(s))=\tilde{\chi}_{\tilde{J}}G(\be(s))
=\tilde{\chi}_{\tilde{J}}\sum^k_{i=1}\ga(\be_i(s))(\cdot-p_i)=\ga(\be_j(s))(\cdot-p_j),$$
$$\|g(s)- v_n(\cdot-p_j)\|_{I_j}=\|\ga(\be_j(s))(\cdot-p_j)-v_n(\cdot-p_j)\|_{I_j}=\|G(\be(s))-v_n(\cdot-p_j)\|_{I_j}>r_{R}.$$
Hence,
$$\|\ga(\be_j(s))-v_n\|_*\geq\|\ga(\be_j(s))(\cdot-p_j)-v_n(\cdot-p_j)\|_{I_j}>r_{R}\geq r_n-\frac{d_{r_n}}{2}.$$

On the other hand, by (e3), one has $\|\ga(\be_j(s))-v_n\|_*\leq r_n$. This, together with (e4) and the assumption that $p_j$ is a multiple of $T$, implies that
\beq\label{Ch2f57}
g(s)=\ga(\be_j(s))(\cdot-p_j)\in f^{\bar{\kappa}-\frac{h_{r_n}}{2}}.
\eeq

As a consequence, we know that $\Ld_0\subsetneqq[0,1]$, otherwise, this contradicts (e2).

{\bf Step 2.} We investigate the case that $s\in\Ld_1$.
\beqq
\|g(s)-v_n(\cdot-p_j) \|^2_{I_j}\leq(\|\tilde{\chi}_{\tilde{J}}(\bar{G}(\be(s))-v_n(\cdot-p_j))\|_{I_j}+\|(1-\tilde{\chi}_{\tilde{J}})v_n(\cdot-p_j)\|_{I_j})^2.
\eeqq
By \eqref{Ch2f51} and the invariance of the set $\mathcal{B}_{r_{R}}^{P}(v)$ under the flow, one has
\beq\label{Ch2f55}
\|\tilde{\chi}_{\tilde{J}}(\bar{G}(\be(s))-v_n(\cdot-p_j))\|^2_{I_j}\leq2 \|\bar{G}(\be(s))-v_n(\cdot-p_j)\|^2_{I_j}\leq 2r_{R}^2.
\eeq
It follows from  \eqref{Ch2f54}, \eqref{Ch2f46}, and \eqref{Ch2f55} that
\beqq
\|g(s)-v_n(\cdot-p_j)\|_*^2\leq2(r_{R}+\ep^{1/2})^2+\ep<4r_n^2,
\eeqq
where \eqref{Ch2epb} is used.
Hence, $g(s)\in\mathcal{B}_{2r_n}(v_n(\cdot-p_j))\subset\mathcal{B}_{\bar{r}}(\bar{v}(\cdot-p_j))$.

By the definition of $f_j$ in \eqref{Ch2f23}, one has for any $s\in[0,1]$,
\beq\label{Ch2f42}
\begin{split}
&f(g(s))=f_j(g(s))=\frac{1}{2}\|g(s)\|^2_{I_j}-\sum_{t\in I_j}V(t,g(s)(t))\\
=&\frac{1}{2}\|g(s)\|^2_{I_j\setminus (M_j\cup M_{j-1})}+
\frac{1}{2}\|g(s)\|^2_{I_j\cap(M_j\cup M_{j-1})}-\sum_{t\in\tilde{J}}V(t,g(s)(t))\\
=&\frac{1}{2}\|\tilde{\chi}_{\tilde{J}}\bar{G}(\be(s))\|^2_{I_j\setminus (M_j\cup M_{j-1})}+
\frac{1}{2}\|\tilde{\chi}_{\tilde{J}}\bar{G}(\be(s))\|^2_{I_j\cap(M_j\cup M_{j-1})}-\sum_{t\in\tilde{J}}V(t,\tilde{\chi}_{\tilde{J}}(t)\bar{G}(\be(s))(t))\\
=&\frac{1}{2}\|\bar{G}(\be(s))\|^2_{I_j\setminus (M_j\cup M_{j-1})}+
\frac{1}{2}\|\tilde{\chi}_{\tilde{J}}\bar{G}(\be(s))\|^2_{I_j\cap(M_j\cup M_{j-1})}\\
&-\sum_{t\in I_j\setminus (M_j\cup M_{j-1})}V(t,\bar{G}(\be(s))(t))-\sum_{t\in I_j\cap(M_j\cup M_{j-1})}V(t,\tilde{\chi}_{\tilde{J}}(t)\bar{G}(\be(s))(t)),
\end{split}
\eeq
and
\beq
f_j(\bar{G}(\be(s)))=\frac{1}{2}\| \bar{G}(\be(s))\|^2_{I_j}-\sum_{t\in I_j}V(t,\bar{G}(\be(s))(t)).
\eeq
So,
\beq\label{Ch2f43}
\begin{split}
&f_j(g(s))=f_j(\bar{G}(\be(s)))+\frac{1}{2}\|\tilde{\chi}_{\tilde{J}}\bar{G}(\be(s))\|^2_{I_j\cap(M_j\cup M_{j-1})}
-\frac{1}{2}\| \bar{G}(\be(s))\|^2_{I_j\cap(M_j\cup M_{j-1})}\\+&\sum_{t\in I_j\cap(M_j\cup M_{j-1})}V(t,\bar{G}(\be(s))(t))-\sum_{t\in I_j\cap(M_j\cup M_{j-1})}V(t,\tilde{\chi}_{\tilde{J}}(t)\bar{G}(\be(s))(t)).
\end{split}
\eeq

By \eqref{Ch2f51}, (g2) of Lemma \ref{Ch2pathg}, one has
\beq\label{Ch2f44}
\begin{split}
&\bigg|\frac{1}{2}\|\tilde{\chi}_{\tilde{J}}\bar{G}(\be(s))\|^2_{I_j\cap(M_j\cup M_{j-1})}
-\frac{1}{2}\| \bar{G}(\be(s))\|^2_{I_j\cap(M_j\cup M_{j-1})}\bigg|\\
\leq&2(\| \bar{G}(\be(s))\|^2_{M_j}+\| \bar{G}(\be(s))\|^2_{M_{j-1}})\leq16\ep.
\end{split}
\eeq
It follows from \eqref{Ch2f4}, \eqref{Ch2epb}, and (g2) of Lemma \ref{Ch2pathg} that
\beq\label{Ch2f45}
\bigg|\sum_{t\in I_j\cap(M_j\cup M_{j-1})}V(t,\bar{G}(\be(s))(t))\bigg|\leq\| \bar{G}(\be(s))\|^2_{I_j\cap(M_j\cup M_{j-1})}\leq 8\ep.
\eeq
Similarly, one has
\beq\label{Ch2f60}
\bigg|\sum_{t\in I_j\cap(M_j\cup M_{j-1})}V(t,\tilde{\chi}_{\tilde{J}}(t)\bar{G}(\be(s))(t))\bigg|\leq16\ep.
\eeq

By \eqref{Ch2f43}--\eqref{Ch2f60} and (g3) of Lemma \ref{Ch2pathg}, one has
\beq\label{Ch2f56}
f(g(s))\leq \rho_{L}+41\ep<\bar{\kappa},
\eeq
where \eqref{Ch2epb} is used.

By the translation of $-p_j$ along $g$, we know that $g$ is a path joining $\ga(0)$ and $\ga(1)$. This, together with \eqref{Ch2f57} and \eqref{Ch2f56}, implies that $\ga(0)$ and $\ga(1)$ are $\bar{\kappa}$-connectible. This contradicts (e2). This completes the proof.

\section{The proof of Lemma \ref{Ch2vector} and the construction of a vector field}

In this section, we give the proof of Lemma \ref{Ch2vector} by constructing a proper vector field on $l^2$.

Choose a constant $N_0$ such that
\beq\label{Ch2g2}
N_0\geq\max\bigg\{8,\ \frac{8}{L_2},\ \bigg[\frac{5}{L_2}\bigg]+1\bigg\},
\eeq
where $L_2$ is specified in \eqref{Ch2f20}.

Given any $v\in l^2$ and $r>0$, there is a positive integer $\bar{N}=\bar{N}(v,r)$ such that for any integer $N>\bar{N}$, any $k\in\NN$, and $(p_1,...,p_k)\in \mathcal{P}(k,N)$, one has that for any $u\in\mathcal{B}^{P}_{r}(v)$ and any $i$, $1\leq i\leq k$, there is $j=j(i)$, $1\leq j\leq N$, such that
$$\|u\|^{2}_{jN\leq|t-p_i|\leq(j+1)N}\leq\frac{4r^2}{N},$$
since $\|u\|_{jN\leq|t-p_i|\leq(j+1)N}\leq\|u-v(\cdot-p_i)\|_{jN\leq|t-p_i|\leq(j+1)N}+\|v(\cdot-p_i)\|_{jN\leq|t-p_i|\leq(j+1)N}$, $v$ is fixed, $N$ is sufficiently large, and \eqref{Ch2e10}.

It is possible that $j=j(i)$ is not unique, let $j_{u,i}$ be the smallest index such that the above inequality holds.

For any $\ep\in(0,\min\{r,\|v\|_*^2\})$, there exists $N_{\ep}\geq \max\{N_0,\ \bar{N}\}$ such that
\begin{equation}\label{Ch2f3}
\max\bigg\{\|v\|^2_{|t|\geq N_{\ep}},\frac{4r^2}{N_{\ep}}\bigg\}<\frac{\ep}{4}.
\end{equation}
So, for any $N>N_{\ep}$, any positive integer $k$, $P=(p_1,...,p_k)\in \mathcal{P}(k,N)$, one has that for any $u\in\mathcal{B}^{P}_{r}(v)$ and any $1\leq i\leq k$,
\beq\label{Ch2f15}
\|u\|^{2}_{j_{u,i}N\leq|t-p_i|\leq(j_{u,i}+1)N}<\frac{\ep}{2}.
\eeq

Given any $N>N_{\ep}$ and $u\in\mathcal{B}_r^{P}(v)$, set
$$A_{u,0}:=(p_0,p_1-(j_{u,1}+1)N]\cap\ZZ;$$
$$A_{u,i}:=(p_{i}+(j_{u,i}+1)N,p_{i+1}-(j_{u,i+1}+1)N]\cap\ZZ,\ 1\leq i\leq k-1;$$
$$A_{u,k}:=(p_k+(j_{u,k}+1)N,p_{k+1})\cap\ZZ;\ A_u:=\cup^{k}_{l=0}A_{u,l};$$
$$B_{u,l}:=\{t\in\ZZ:\ d(t,A_{u,l})\leq N\},\ 0\leq l\leq k;\ B_u:=\cup^{k}_{l=0}B_{u,l};$$
$$F_{u,i}:=I_i\cap(B_u\setminus A_u),\ 1\leq i\leq k.$$
From the construction above, it follows that $M_l\subset A_{u,l}$, $0\leq l\leq k$; the number of integers contained in $B_{u,l}\setminus A_{u,l}$ is $2N$, $1\leq l\leq k-1$. It follows from \eqref{Ch2f15} that
\begin{equation} \label{Ch2f10}
\|u\|^2_{F_{u,i}}\leq\frac{\ep}{2},\ 1\leq i\leq k;
\end{equation}
and
\begin{equation}\label{Ch2f6}
\|u\|^2_{B_{u,l}\setminus A_{u,l}}\leq\ep,\ 0\leq l\leq k.
\end{equation}

For any set $B\subset\mathbb{Z}$, set $R_B:=\sup_{t\in B}t$ and $L_B:=\inf_{t\in B}t$.
For the set $A_{u,l}$, $0\leq l\leq k$, we define the following step functions
\beq \label{Ch2h1}
 \chi_{u,0}(t):=\left\{
    \begin{array}{ll}
      1, & \hbox{if}\ t\in A_{u,0}\cup\{R_{A_{u,0}}+1\}, \\
      \frac{N_0-i}{N_0}, & \hbox{if}\ t=R_{A_{u,0}}+1+i,\ 1\leq i\leq N_0,  \\
      0, & \hbox{otherwise;}
    \end{array}
  \right.
\eeq
\beq \label{Ch2h2}
\chi_{u,l}(t):=\left\{
    \begin{array}{ll}
      1, & \hbox{if}\ t\in A_{u,l}\cup\{L_{A_{u,l}}-1,R_{A_{u,l}}+1\}, \\
       \frac{N_0-i}{N_0}, & \hbox{if}\ t=R_{A_{u,l}}+1+i,\ 1\leq i\leq N_0,  \\
       \frac{N_0-i}{N_0}, & \hbox{if}\ t=L_{A_{u,l}}-1-i,\ 1\leq i\leq N_0,  \\
      0, & \hbox{otherwise,}
    \end{array}
  \right.\ \mbox{when}\ 1\leq l\leq k-1;
\eeq
\beq \label{Ch2h3}
\chi_{u,k}(t):=\left\{
    \begin{array}{ll}
      1, & \hbox{if}\ t\in A_{u,k}\cup\{L_{A_{u,k}}-1\}, \\
      \frac{N_0-i}{N_0}, & \hbox{if}\ t=L_{A_{u,k}}-1-i,\ 1\leq i\leq N_0,  \\
      0, & \hbox{otherwise.}
    \end{array}
  \right.
\eeq

For $1\leq i\leq k$, set
\beq\label{Ch2f24}
\bar{\chi}_{u,i}(t):=\left\{
  \begin{array}{ll}
    0, & \hbox{if}\ t\not\in I_i, \\
    1-\chi_{u,i-1}-\chi_{u,i}, & \hbox{if}\ t\in I_i.
  \end{array}
\right.
\eeq
By direct computation, one has
\beq\label{Ch2g1}
\begin{split}
&\lan u,\chi_{u,l}u\ran_*=\sum_{t\in\ZZ}\lan\De u(t-1), \De(\chi_{u,l}(t-1)u(t-1))\ran+\sum_{t\in\ZZ}\lan u(t),L(t)\chi_{u,l}(t)u(t) \ran\\
=&\sum_{t\in\ZZ}\lan\De u(t-1), \chi_{u,l}(t)u(t)-\chi_{u,l}(t-1)u(t-1)\ran+\sum_{t\in\ZZ}\lan u(t),L(t)\chi_{u,l}(t)u(t) \ran\\
=&\sum_{t\in\ZZ}\lan\De u(t-1),\chi_{u,l}(t)u(t)-\chi_{u,l}(t-1)u(t)+\chi_{u,l}(t-1)u(t)-\chi_{u,l}(t-1)u(t-1) \ran\\
+&\sum_{t\in\ZZ}\lan u(t),L(t)\chi_{u,l}(t)u(t) \ran=\sum_{t\in\ZZ}\lan\De u(t-1), \chi_{u,l}(t-1)(\De u(t-1))\ran\\
+&\sum_{t\in\ZZ}\lan\De u(t-1),(\De(\chi_{u,l}(t-1))u(t)\ran+\sum_{t\in\ZZ}\lan u(t),L(t)\chi_{u,l}(t)u(t) \ran\\
=&\lan u, u\ran_{A_{u,l}}+\sum_{t\in B_{u,l}\setminus A_{u,l}}\lan \De u(t-1), \chi_{u,l}(t-1)(\De u(t-1))\ran\\
+&\sum_{t\in B_{u,l}\setminus A_{u,l}}\lan u(t),L(t)\chi_{u,l}(t)u(t) \ran+\sum_{t\in\ZZ}\lan\De u(t-1),(\De(\chi_{u,l}(t-1))u(t) \ran.
\end{split}
\eeq
By \eqref{Ch2g2} and \eqref{Ch2f6}, one has
\beq\label{Ch2g3}
\begin{split}
&\sum_{t\in\ZZ}\lan\De u(t-1),(\De(\chi_{u,l}(t-1))u(t) \ran\leq\frac{1}{N_0}\sum_{t\in B_{u,l}\setminus A_{u,l}}\lan \De u(t-1), u(t)\ran\\
\leq&\frac{1}{N_0}\sum_{t\in B_{u,l}\setminus A_{u,l}}|\De u(t-1)||u(t)|
\leq\frac{1}{2N_0}\sum_{t\in B_{u,l}\setminus A_{u,l}}(|\De u(t-1)|^2+|u(t)|^2)\\
\leq&\frac{1}{2N_0}\sum_{t\in B_{u,l}\setminus A_{u,l}}(|\De u(t-1)|^2+\frac{1}{L_2}\lan u(t),L(t)u(t)\ran)\\
\leq&\frac{\max\{1,1/L_2\}}{2N_0}\sum_{t\in B_{u,l}\setminus A_{u,l}}(|\De u(t-1)|^2+\lan u(t),L(t)u(t)\ran)
\leq\frac{\ep}{4}
\end{split}
\eeq
It follows from \eqref{Ch2g1} and \eqref{Ch2g3} that
\beq\label{Ch2g4}
\lan u,\chi_{u,l}u\ran_*\geq\lan u, u\ran_{A_{u,l}}-\frac{\ep}{4},\ 0\leq l\leq k.
\eeq

For any $0\leq l\leq k$, set
\beq\label{Ch2f35}
h_l(u):=\left\{
    \begin{array}{ll}
      1, & \hbox{if}\ \|u\|^2_{A_{u,l}}\geq\ep, \\
      \frac{1}{k+1}, & \hbox{otherwise},
    \end{array}
  \right.
\eeq
and
\beq\label{Ch2f36}
W_u:=\sum^{k}_{l=0}h_l(u)\chi_{u,l}u.
\eeq

\begin{lemma}\label{Ch2diff2}
Let $r\in(0,r_0/4)$ and $\ep\in(0,r^2)$, for any $u\in\mathcal{B}_r^{P}(v)$, one has
\begin{equation*}
 D_*f(u)W_u\geq\frac{1}{2}\sum^k_{l=0}h_l(u)(\|u\|^2_{A_{u,l}}-\ep),
\end{equation*}
\begin{equation*}
 D_*f_i(u)W_u\geq\frac{1}{2}\sum^k_{l=0}h_l(u)(\|u\|^2_{I_i\cap A_{u,l}}-\ep),\ 1\leq i\leq k,
\end{equation*}
\end{lemma}
where $r_0$ is specified in \eqref{Ch2f4}.

\begin{proof}
Since $\ep\in(0,r^2)$, $u\in\mathcal{B}_r^{P}(v)$, and \eqref{Ch2f3}, one has $\|u\|_{A_{u,l}\cap I_i}\leq 2r$, $1\leq i\leq k$. This implies that $\|u\|_{A_{u,l}}\leq 4r<r_0$. By \eqref{Ch2f6} and $\ep\in(0,r^2)$, one has that $\|u\|_{B_{u,l}\setminus A_{u,l}}\leq\ep^{1/2}<r<r_0$. This, together with \eqref{Ch2deriva}, \eqref{Ch2f4}, and \eqref{Ch2g4}, yields that
\begin{equation*}
\begin{split}
& D_*f(u)W_u=\lan u,W_u\ran_*-\sum_{t\in\ZZ}\lan V'_x(t,u(t)),W_u(t)\ran\\
=&\bigg\lan u,\sum^k_{l=0}h_l(u)\chi_{u,l}u\bigg\ran_*-\sum_{t\in\ZZ}\bigg\lan V'_x(t,u(t)),\sum^k_{l=0}h_l(u)\chi_{u,l}(t)u(t)\bigg\ran\\
=&\sum^k_{l=0}h_l(u)\bigg(\lan u,\chi_{u,l}u\ran_*-\sum_{t\in\ZZ}\lan V'_x(t,u(t)),\chi_{u,l}(t)u(t)\ran\bigg)\\
\geq&\sum^k_{l=0}h_l(u)\bigg(\|u\|^2_{A_{u,l}}-\frac{\ep}{4}-\sum_{t\in A_{u,l}}\lan V'_x(t,u(t)),u(t)\ran-\sum_{t\in B_{u,l}\setminus A_{u,l}}\lan V'_x(t,u(t)),\chi_{u,l}(t)u(t)\ran\bigg)\\
\geq&\sum^k_{l=0}h_l(u)\bigg(\|u\|^2_{A_{u,l}}-\frac{\ep}{4}-\frac{1}{8}\|u\|^2_{A_{u,l}}-\frac{1}{8}\|u\|^2_{B_{u,l}\setminus A_{u,l}} \bigg)\\
\geq&\sum^k_{l=0}h_l(u)\bigg(\frac{7}{8}\|u\|^2_{A_{u,l}}-\frac{\ep}{4}-\frac{\ep}{8}\bigg)\geq\frac{1}{2}\sum^k_{l=0}h_l(u)(\|u\|^2_{A_{u,l}}-\ep).
\end{split}
\end{equation*}

By similar discussions and \eqref{Ch2diff3}, one has
\beqq
 D_*f_i(u)W_u\geq\sum^k_{l=0}h_l(u)\bigg(\frac{7}{8}\|u\|^2_{I_i\cap A_{u,l}}-\frac{3\ep}{8}\bigg)\geq\frac{1}{2}\sum^k_{l=0}h_l(u)(\|u\|^2_{I_i\cap A_{u,l}}-\ep).
\eeqq

\end{proof}

It follows from Lemma \ref{Ch2diff2} that
\beq\label{Ch2f28}
 D_*f(u)W_u\geq\frac{1}{2}\sum^k_{l=0}h_l(u)(\|u\|^2_{A_{u,l}}-\ep)\geq\frac{1}{2}\sum_{\{l:\ \|u\|^2_{A_{u,l}}<\ep\}}h_l(u)(\|u\|^2_{A_{u,l}}-\ep)\geq-\frac{\ep}{2}.
\eeq
Similarly, one has
\beq\label{Ch2f29}
 D_*f_i(u)W_u\geq-\frac{\ep}{2}.
\eeq

\begin{proposition} \label{globvect}
Fix any $b<\kappa^*$ with $\mathcal{C}(b)\neq\emptyset$. Fix any $v\in\mathcal{C}(b)$. For any $r\in(0,\min\{\frac{r_0}{8}.\frac{\sqrt{3}\|v\|_{*}}{4}\})\setminus D^{\kappa^*}$. Let $r_{s},r_{m},r_{b}$ be real numbers with $0<r-3d_r<r_{s}<r_{m}<r_{b}<r+3d_r$. Given
$b_{L},b_{R}$, and $\de$ with $\de\in(0,r^2_0/4)$, $(b_{L}-\de,b_{L}+2\de)\subset(0,b)\setminus \Phi^{\kappa^*}$, and $(b_{R}-\de,b_{R}+2\de)\subset(b,\kappa^*)\setminus\Phi^{\kappa^*}$, there exist $\mu>0$ and $\ep_1>0$ ($\mu$ is dependent on $r,d_r,r_s$,and $r_b$; $\ep_1$ is dependent on $r,d_r,r_s,r_b,b_{R},b_{L},\de$,) such that for any $\ep\in(0,\ep_1)$, there is a positive integer $N^*(\ep)$ such that for any positive integer $N\geq N^*$, positive integer $k$, and $P=(p_1,...,p_k)\in \mathcal{P}(k,N)$, there exists a locally Lipschitz continuous function $W:l^2\to l^2$ satisfying the following conditions
\begin{itemize}
\item [(I)] $ D_*f(u)W(u)\geq0$, $\|W(u)\|_{I_i}\leq 2$, $1\leq i\leq k$, for any $u\in l^2$; $W(u)=0$, if $u\in l^2\setminus \mathcal{B}_{r_{b}}^{P}(v)$;
\item [(II)] $D_*f_i(u)W(u)\geq\mu$, if $r_{s}\leq\|u-v(\cdot-p_i)\|_{I_i}\leq r_{m}$, $u\in\mathcal{B}_{r_{m},b_{R}+\de}^{P}(v)$;
\item [(III)] $ D_*f_i(u)W(u)\geq0$, for any $u\in(f^{b_{R}+\de}_i\setminus f^{b_{R}}_i)\cup(f^{b_{L}+\de}_i\setminus f^{b_{L}}_i)$, any $i$, $1\leq i\leq k$;
\item [(IV)] $\lan u,W(u)\ran_{M_i}\geq0$, if $u\in l^2\setminus\mathcal{M}_{4\ep}$;
\item [(V)]  if $\mathcal{C}\cap\mathcal{B}_{r_{s}}^{P}(v)=\emptyset$, then there is $\mu_k>0$ so that
               $ D_*f(u)W(u)\geq\mu_k$, for any $u\in\mathcal{B}_{r_{s}}^{P}(v)$.
\end{itemize}
\end{proposition}

\begin{proof}
First, we introduce several constants:
$$\bar{r}_{s}=\frac{1}{2}(r_{s}+r-3d_r),\ \bar{r}_{b}=\frac{1}{2}(r_{b}+r+3d_r).$$
For the given $r$, it follows from Lemma \ref{Ch2diffposi} and Remark \ref{Ch2k40} that there are $\mu_r>0$ and $d_r\in(0,r/3)$ such that for any $u\in(\mathcal{A}_{r-3d_r,r+3d_r}(\mathcal{C}^b)\cup\mathcal{A'}_{r-3d_r,r+3d_r}(\mathcal{C}^b))\cap f^b$,\ $\|D_*f(u)\|_{*}\geq\mu_r$.
Set
\beq\label{Ch2f41}
\nu:=\inf\{\|D_*f(u)\|_{*}:\ u\in(f^{b_{R}+2\de}\setminus f^{b_{R}-\de})\cup(f^{b_{L}+2\de}\setminus f^{b_{L}-\de})\},
\eeq
by Lemma \ref{Ch2diffposi1}, one has that $\nu>0$.

Set
\begin{equation}\label{Ch2f1}
\ep_1:=\bigg(\min\bigg\{\frac{r_{s}-r+3d_r}{12},\ \frac{r+3d_r-r_{b}}{12},\ \frac{\mu_r}{16},\ \frac{\nu}{16},\ \frac{\de^{1/2}}{6},\ r,\ \|v\|_*\bigg\}\bigg)^2.
\end{equation}

Choose
\beq\label{Ch2f27}
\ep\in(0,\ep_1).
\eeq

For the given $v$, $r$, and $\ep$, by \eqref{Ch2g2} and \eqref{Ch2f3}, set $N^*:=2N_{\ep}$. In the following discussions, take $N>N^*$, and $(p_1,...,p_k)\in \mathcal{P}(k,N)$. In the discussions of \eqref{Ch2g2} and \eqref{Ch2f3}, one has that if $N>N^*$, then the choice of $N$ does not affect the conclusions in \eqref{Ch2g2} and \eqref{Ch2f3}.

Since $r_b<r+3d_r<2r<r_0/4$, for the given $r$ and $\ep$, it follows from Lemma \ref{Ch2diff2} that for any $u\in\mathcal{B}^P_{r_b}(v)$,\ there is a vector field $W_u$ satisfying the inequalities in Lemma \ref{Ch2diff2}.

Now, fix $k\in\NN$. For the given $r$ and $\ep$, let $N>N_{\ep}$ and $(p_1,...,p_k)\in \mathcal{P}(k,N)$.

Since $r_{b}<r+3d_r<2r<r_0/4$ and \eqref{Ch2f27}, it follows from Lemma \ref{Ch2diff2} that there is a vector field $W_u$ satisfying the properties of Lemma \ref{Ch2diff2} for any $u\in\mathcal{B}_{r_{b}}^{P}(v)$ with this given $\ep$.

{\bf Case (1)} We consider $u\in\mathcal{B}_{r_{b},b_{R}+\frac{3\de}{2}}^{P}(v)\setminus\mathcal{B}_{r_{s}}^{P}(v)$.
Set
$$\mathcal{I}_1(u):=\{i:\ 1\leq i\leq k,\ \|u-v(\cdot-p_i)\|_{I_i}\geq r_{s}\}.$$
So, $\mathcal{I}_1(u)\neq\emptyset$.

Set
\begin{equation}\label{Ch2f2}
\z_1:=\min\bigg\{\frac{r_{s}-\bar{r}_{s}}{2},\ \frac{\bar{r}_{b}-r_{b}}{2}\bigg\}.
\end{equation}
Fix any $i\in\mathcal{I}_1(u)$,
$$\mbox{either}\ \mbox{(i)}\ \|u\|_{I_i\cap A_u}\geq\z_1\ \mbox{or}\ \mbox{(ii)}\ \|u\|_{I_i\cap A_u}<\z_1.$$
It follows from Lemma \ref{Ch2diff2}, \eqref{Ch2f1}, \eqref{Ch2f27}, and \eqref{Ch2f2} that $\ep^{1/2}_1\leq\z_1/3,$ and
\begin{equation*}
\begin{split}
& D_*f(u) W_u\geq\frac{1}{2}(\|u\|^2_{A_{u,i-1}}+\|u\|^2_{A_{u,i}}-2\ep)-\sum_{\{l:\ \|u\|^2_{A_{u,l}}<\ep\}}h_l(u)\frac{\ep}{2}\\
\geq&\frac{1}{2}(\|u\|^2_{I_i\cap A_u}-2\ep)-\sum_{\{l:\ \|u\|^2_{A_{u,l}}<\ep\}}h_l(u)\frac{\ep}{2}\geq\frac{\z^2_1}{2}-2\ep\geq\frac{\z^2_1}{4};
\end{split}
\end{equation*}
similarly, one has
\beqq
 D_*f_i(u)W_u\geq\frac{\z^2_1}{2}-2\ep\geq\frac{\z^2_1}{4}.
\eeqq
So, for Case (i), one has
\beq\label{Ch2f32}
\min\{ D_*f(u)W_u,\  D_*f_i(u)W_u\}\geq\frac{\z^2_1}{4}.
\eeq

Now, for Case (i), set
\beq\label{Ch2f34}
W_{u,i}:=0.
\eeq

Now, we consider Case (ii).

First, we show that $\bar{\chi}_{u,i}u\in\mathcal{A}_{r-3d_r,r+3d_r}(v(\cdot-p_i))\cap f^{\kappa^*}$.

It follows from \eqref{Ch2f3} and $N>N_{\ep}$ that $\|v(\cdot-p_i)\|^2_{|t-p_i|\geq N}\leq\ep/4$. This, together with \eqref{Ch2f10}, yields that
\beqq
\begin{split}
&\|u-v(\cdot-p_i)\|^2_{I_i}=\|u-v(\cdot-p_i)\|^2_{I_i\setminus A_u}+\|u-v(\cdot-p_i)\|^2_{I_i\cap A_u}\\
=&\|u-v(\cdot-p_i)\|^2_{I_i\setminus B_u}+\|u-v(\cdot-p_i)\|^2_{(B_u\setminus A_u)\cap I_i}+\|u-v(\cdot-p_i)\|^2_{I_i\cap A_u}\\
\leq&\|u-v(\cdot-p_i)\|^2_{I_i\setminus B_u}+\|u\|^2_{(B_u\setminus A_u)\cap I_i}+\|v(\cdot-p_i)\|^2_{(B_u\setminus A_u)\cap I_i}
+2\|v(\cdot-p_i)\|_{(B_u\setminus A_u)\cap I_i}\|u\|_{(B_u\setminus A_u)\cap I_i}\\
+&\|u\|^2_{I_i\cap A_u}+2\|u\|_{I_i\cap A_u}\|v(\cdot-p_i)\|_{I_i\cap A_u}+
\|v(\cdot-p_i)\|^2_{I_i\cap A_u}\\
\leq&\|u-v(\cdot-p_i)\|^2_{I_i\setminus B_u}+\frac{\ep}{2}+\frac{\ep}{4}
+2\frac{\ep^{1/2}}{2}\frac{\ep^{1/2}}{\sqrt{2}}+\z_1^2+2\z_1\ep^{1/2}/2+\ep/4\\
\leq&\|u-v(\cdot-p_i)\|^2_{I_i\setminus B_u}+4\ep_1+2\ep_1^{1/2}\z_1+\z^2_1.
\end{split}
\eeqq

Since $\|u-v(\cdot-p_i)\|^2_{I_i}\geq r_{s}^2$, one has
\beq\label{Ch2f22}
\|u-v(\cdot-p_i)\|^2_{I_i\setminus B_u}\geq r_{s}^2-(\z^2_1+2\ep_1^{1/2}\z_1+4\ep_1).
\eeq
By \eqref{Ch2f1} and \eqref{Ch2f2}, one has
\beq \label{Ch2f21}
\z^2_1+2\ep^{1/2}_1\z_1+4\ep_1\leq\z^2_1+\ep_1+\z^2_1+4\ep_1=2\z^2_1+5\ep_1\leq r_{s}^2-(r-3d_r)^2.
\eeq
It follows from \eqref{Ch2f22} and \eqref{Ch2f21} that
\beqq
\|u-v(\cdot-p_i)\|^2_{I_i\setminus B_u}\geq(r-3d_r)^2.
\eeqq
So,
\beqq
\|\bar{\chi}_{u,i}u-v(\cdot-p_i)\|_*^2=\|u-v(\cdot-p_i)\|^2_{I_i\setminus B_u}+\|\bar{\chi}_{u,i}u-v(\cdot-p_i)\|^2_{I_i\cap B_u}+\|v(\cdot-p_i)\|^2_{\ZZ\setminus I_i}
\geq (r-3d_r)^2.
\eeqq

On the other hand, by \eqref{Ch2f24}, \eqref{Ch2f1}, and \eqref{Ch2f2}, one has
\beqq
\begin{split}
&\|\bar{\chi}_{u,i}u-v(\cdot-p_i)\|_*^2=\|\bar{\chi}_{u,i}u-v(\cdot-p_i)\|^2_{I_i}+\|v(\cdot-p_i)\|^2_{\ZZ\setminus I_i}\\
\leq&\|\bar{\chi}_{u,i}u-v(\cdot-p_i)\|^2_{I_i}+\ep\leq(\|\bar{\chi}_{u,i}u-u\|_{I_i}+\|u-v(\cdot-p_i)\|_{I_i})^2+\ep\\
\leq&(\|(1-\bar{\chi}_{u,i})u\|_{I_i}+r_{b})^2+\ep\leq(\|u\|_{I_i\cap A_u}+\|(1-\bar{\chi}_{u,i})u\|_{F_{u,i}}+r_{b})^2+\ep\\
\leq&(\z_1+\ep^{1/2}+r_{b})^2+\ep\leq(\z_1+2\ep^{1/2}+r_{b})^2\leq (r+3d_r)^2,
\end{split}
\eeqq
where $\|(1-\bar{\chi}_{u,i})u\|^2_{F_{u,i}}\leq 2\|u\|^2_{F_{u,i}}$ is used, this can be proved by the similar calculation in \eqref{Ch2f51}.
Hence, $\bar{\chi}_{u,i}u\in\mathcal{A}_{r-3d_r,r+3d_r}(v(\cdot-p_i))$.

Now, we show that $\bar{\chi}_{u,i}u\in f^{\kappa^*}$.

Since $d_r\in(0,r/3)$, one has $r_{b}<r+3d_r<r+r=2r$. For any $u\in\mathcal{B}_{r_{b},b_{R}+\frac{3\de}{2}}^{P}(v)$, it follows from \eqref{Ch2f3} and $\ep<r^2$ that
\beqq
\|u\|_{I_i\cap B_u}\leq\|u-v(\cdot-p_i)\|_{I_i}+\|v(\cdot-p_i)\|_{I_i\cap B_u}<r_{b}+r<2r+r=3r.
\eeqq
This, together with the assumption that $r\in(0,r_0/8)$, implies that $\|u\|_{I_i\cap A_u}<r_0/2$. So, by \eqref{Ch2f4}, one has that
$\sum_{t\in I_i\cap A_u}V(t,u(t))\leq\|u\|^2_{I_i\cap A_u}/8$. Hence,
\beq\label{Ch2f25}
\frac{1}{2}\|u\|^2_{I_i\cap A_u}-\sum_{t\in I_i\cap A_u}V(t,u(t))\geq0.
\eeq
By applying similar discussions in \eqref{Ch2f58}, one can show that
\beq\label{Ch2g5}
\|\bar{\chi}_{u,i}u\|^2_{F_{u,i}}\leq 2\|u\|^2_{F_{u,i}}.
\eeq
This, together with \eqref{Ch2f4}, \eqref{Ch2f10}, and \eqref{Ch2f1}, yields that
\beq\label{Ch2f26}
\sum_{t\in F_{u,i}}V(t,\bar{\chi}_{u,i}u)\leq\frac{\|\bar{\chi}_{u,i}u\|^2_{F_{u,i}}}{8}\leq \frac{\|u\|^2_{F_{u,i}}}{4}\leq \frac{\ep}{4}.
\eeq
Hence, by the definitions of $f_i(u)$ in \eqref{Ch2f23} and $\bar{\chi}_{u,i}$ in \eqref{Ch2f10}, \eqref{Ch2f24}, \eqref{Ch2f25}, \eqref{Ch2g5}, and \eqref{Ch2f26}, one has
\beqq
\begin{split}
&f(\bar{\chi}_{u,i}u)=f_i(\bar{\chi}_{u,i}u)\leq f_i(u)-\bigg(\frac{1}{2}\|u\|^2_{I_i\cap A_u}-\sum_{t\in I_i\cap A_u}V(t,u(t))\bigg)+
\frac{1}{2}\|u\|^2_{F_{u,i}}\\
+&\frac{1}{2}\|\bar{\chi}_{u,i}u\|^2_{F_{u,i}}+\bigg|\sum_{t\in F_{u,i}}V(t,u(t))\bigg|+\bigg|\sum_{t\in F_{u,i}}V(t,\bar{\chi}_{u,i}(t)u(t))\bigg|\\
\leq& f_i(u)+\frac{\ep}{4}+\frac{\ep}{2}+\frac{\ep}{16}+\frac{\ep}{4}\\
\leq &f_i(u)+\frac{\de}{2}\leq b_{R}+3\de/2+\de/2=b_{R}+2\de<\kappa^*,
\end{split}
\eeqq
the last two inequalities are derived from the fact that $u\in\mathcal{B}_{r_{b},b_{R}+\frac{3\de}{2}}^{P}(v)$, \eqref{Ch2f1}, and \eqref{Ch2f27}.

Hence, $\bar{\chi}_{u,i}u\in\mathcal{A}_{r-3d_r,r+3d_r}(v(\cdot-p_i))\cap f^{\kappa^*}$, which implies that $\|D_*f(\bar{\chi}_{u,i}u)\|_*\geq\mu_r$.
Consequently, there is $Q_{u,i}\in l^2$ with $\|Q_{u,i}\|_*\leq1$ such that
\beq\label{Ch2f7}
 D_*f_i(\bar{\chi}_{u,i}u)Q_{u,i}= D_*f(\bar{\chi}_{u,i}u)Q_{u,i}\geq\frac{\mu_r}{2},
\eeq
where \eqref{Ch2deriva} and \eqref{Ch2diff3} are used. By direct calculation, one has
\beqq
\begin{split}
& D_*f_i(\bar{\chi}_{u,i}u)Q_{u,i}-D_*f_i(u)(\bar{\chi}_{u,i}Q_{u,i})\\
=&\bigg(\sum_{t\in I_i}\big(\lan\De(\bar{\chi}_{u,i}(t-1)u(t-1)), \De Q_{u,i}(t-1)\ran+\lan\bar{\chi}_{u,i}(t)u(t), L(t)Q_{u,i}(t)\ran\big)\\
-&\sum_{t\in I_i}\lan V'_x(t,\bar{\chi}_{u,i}(t)u(t)), Q_{u,i}(t)\ran\bigg)-\bigg(\sum_{t\in I_i}\big(\lan\De u(t-1),\De(\bar{\chi}_{u,i}(t-1)Q_{u,i}(t-1))\ran\\
+&\lan u(t),L(t) \bar{\chi}_{u,i}(t)Q_{u,i}(t)\ran\big)-\sum_{t\in I_i}\lan V'_x(t,u(t)),\bar{\chi}_{u,i}(t)Q_{u,i}(t)\ran\bigg)\\
=&\sum_{t\in F_{u,i}}\lan(\De\bar{\chi}_{u,i}(t-1))u(t), \De Q_{u,i}(t-1)\ran-\sum_{t\in F_{u,i}}\lan\De u(t-1),(\De\bar{\chi}_{u,i}(t-1))Q_{u,i}(t))\ran\\
-&\sum_{t\in F_{u,i}}\lan V'_x(t,\bar{\chi}_{u,i}(t)u(t)), Q_{u,i}(t)\ran+\sum_{t\in F_{u,i}}\lan V'_x(t,u(t)),\bar{\chi}_{u,i}(t)Q_{u,i}(t)\ran
\end{split}\
\eeqq
where
\beqq
\begin{split}
&\sum_{t\in I_i}\lan\De(\bar{\chi}_{u,i}(t-1)u(t-1)), \De Q_{u,i}(t-1)\ran\\
=&\sum_{t\in I_i}\lan(\De\bar{\chi}_{u,i}(t-1))u(t), \De Q_{u,i}(t-1)\ran+\sum_{t\in I_i}\lan\bar{\chi}_{u,i}(t-1)(\De u(t-1)), \De Q_{u,i}(t-1)\ran,
\end{split}
\eeqq
\beqq
\begin{split}
&\sum_{t\in I_i}\lan\De u(t-1),\De(\bar{\chi}_{u,i}(t-1)Q_{u,i}(t-1))\ran\\
=&\sum_{t\in I_i}\lan\De u(t-1),(\De\bar{\chi}_{u,i}(t-1))Q_{u,i}(t))\ran+\sum_{t\in I_i}\lan\De u(t-1),\bar{\chi}_{u,i}(t-1)(\De Q_{u,i}(t-1))\ran.
\end{split}
\eeqq
So,
\beq\label{Ch2g8}
\begin{split}
&| D_*f_i(\bar{\chi}_{u,i}u)Q_{u,i}-D_*f_i(u)(\bar{\chi}_{u,i}Q_{u,i})|\\
\leq&\frac{1}{N_0}\sum_{t\in F_{u,i}}|\lan u(t),\De Q_{u,i}(t-1)\ran|+\frac{1}{N_0}\sum_{t\in F_{u,i}}|\lan \De u(t-1),Q_{u,i}(t)\ran|\\
+&\frac{1}{8}\|\bar{\chi}_{u,i}u\|_{F_{u,i}}\|Q_{u,i}\|_{F_{u,i}}+\frac{1}{8}\|u\|_{F_{u,i}}\|\bar{\chi}_{u,i}Q_{u,i}\|_{F_{u,i}}\\
\leq&\frac{2}{N_0L_2}\|u\|_{F_{u,i}}\|Q_{u,i}\|_{F_{u,i}}+\frac{1}{8}\|\bar{\chi}_{u,i}u\|_{F_{u,i}}\|Q_{u,i}\|_{F_{u,i}}+
\frac{1}{8}\|u\|_{F_{u,i}}\|\bar{\chi}_{u,i}Q_{u,i}\|_{F_{u,i}}
\leq\ep^{1/2}.
\end{split}
\eeq

By \eqref{Ch2f1} and \eqref{Ch2f27}, one has
\beq\label{Ch2f8}
| D_*f_i(\bar{\chi}_{u,i}u)Q_{u,i}- D_*f_i(u)(\bar{\chi}_{u,i}Q_{u,i})|\leq\frac{\mu_r}{4}.
\eeq
Since the support of $\bar{\chi}_{u,i}$ is contained in $I_i$, one has
\beq\label{Ch2f9}
|D_*f(\bar{\chi}_{u,i}u)Q_{u,i}- D_*f(u)(\bar{\chi}_{u,i}Q_{u,i})|=|D_*f_i(\bar{\chi}_{u,i}u)Q_{u,i}- D_*f_i(u)(\bar{\chi}_{u,i}Q_{u,i})|\leq\frac{\mu_r}{4}.
\eeq
Combining \eqref{Ch2f7} and \eqref{Ch2f9}, one has if $i\in\mathcal{I}_1(u)$ and $\|u\|_{I_i\cap A_u}<\z_1$, then
\beq\label{Ch2f33}
\min\{ D_*f(u)(\bar{\chi}_{u,i}Q_{u,i}),\ D_*f_i(u)(\bar{\chi}_{u,i}Q_{u,i})\}\geq\frac{\mu_r}{4}.
\eeq

In Case (ii), set
\beq\label{Ch2f37}
W_{u,i}:=\frac{1}{2}\bar{\chi}_{u,i}Q_{u,i}.
\eeq
By \eqref{Ch2f28}, \eqref{Ch2f29}, \eqref{Ch2f1}, and \eqref{Ch2f33}, one has in Case (ii),
\beq\label{Ch2f31}
\min\{ D_*f_i(u)(W_{u,i}+W_u),\ D_*f(u)(W_{u,i}+W_u)\}\geq\frac{\mu_r}{8}-\frac{\ep}{2}\geq\frac{\mu_r}{16}.
\eeq
Set
$$\mu:=\min\bigg\{\frac{\mu_r}{32},\ \frac{\z_1^2}{8}\bigg\},$$
and
\beqq
S_{u,1}:=\left\{
  \begin{array}{ll}
    W_u+\sum_{i\in\mathcal{I}_1(u)}W_{u,i}, & \hbox{if}\ u\in\mathcal{B}_{r_{b},b_{R}+\frac{3\de}{2}}^{P}(v)\setminus\mathcal{B}_{r_{s}}^{P}(v), \\
    0, & \hbox{otherwise.}
  \end{array}
\right.
\eeqq
Thus, for any $u\in\mathcal{B}_{r_{b},b_{R}+\frac{3\de}{2}}^{P}(v)\setminus\mathcal{B}_{r_{s}}^{P}(v)$, by \eqref{Ch2f32}, \eqref{Ch2f34}, and \eqref{Ch2f31}, one has
\beqq
\begin{split}
 D_*f(u)S_{u,1}\geq2\mu;\ D_*f_i(u)S_{u,1}\geq2\mu,\ \forall i\in\mathcal{I}_1(u).\
\end{split}
\eeqq
It follows from \eqref{Ch2f38}, \eqref{Ch2f24}, \eqref{Ch2f35}, \eqref{Ch2f36}, \eqref{Ch2f34}, and \eqref{Ch2f37} that
\beqq
\lan u,S_{u,1}\ran_{M_l}=\lan u,W_u\ran_{M_l}\geq\frac{1}{k+1}\|u\|^2_{M_l},\ 0\leq l\leq k.
\eeqq

{\bf Case (2)}  Now, we consider the case that $u\in\mathcal{B}_{r_{b}}^{P}(v)\cap(\cup^k_{i=1}(f_i)^{b_{R}+\de}_{b_{R}})$.

Set
\beqq
\mathcal{I}^{+}_2(u):=\{i:\ 1\leq i\leq k,\ u\in(f_i)^{b_{R}+\de}_{b_{R}}\},\ \z_2:=\frac{\de^{1/2}}{2},
\eeqq
and choose any $i\in\mathcal{I}^{+}_2(u)$.

Now, we consider two different situations
\beqq
\mbox{(iii)}\ \|u\|_{I_i\cap A_u}\geq\z_2;\ \mbox{(iv)}\ \|u\|_{I_i\cap A_u}<\z_2.
\eeqq
For Case (iii), it follows from \eqref{Ch2f1}, \eqref{Ch2f27}, and Lemma \ref{Ch2diff2} that
\beqq
 D_*f(u)W_u\geq\frac{\z^2_2}{2}-2\ep\geq\frac{1}{4}\z_2^2,\ D_*f_i(u)W_u\geq\frac{\z^2_2}{2}-2\ep\geq\frac{1}{4}\z_2^2.
\eeqq
For $u\in\mathcal{B}_{r_{b}}^{P}(v)\cap(\cup^k_{i=1}(f_i)^{b_{R}+\de}_{b_{R}})$ and $i\in\mathcal{I}^{+}_2(u)$, if $\|u\|_{I_i\cap A_u}\geq\z_2$,
set
\beqq
\bar{W}_{u,i}:=0.
\eeqq

Now, we study Case (iv).

First, we show that $\bar{\chi}_{u,i}u\in(f_i)^{b_{R}+2\de}_{b_{R}-\de}.$

By similar discussions in \eqref{Ch2g5}, \eqref{Ch2f10}, and the definition of $\bar{\chi}_{u,i}$, one has
\beq\label{Ch2f39}
\begin{split}
&\|u\|^2_{I_i}-\|\bar{\chi}_{u,i}u\|^2_{I_i}=\|u\|^2_{I_i\cap A_u}+\|u\|^2_{F_{u,i}}+\|u\|^2_{I_i\setminus B_u}-\|\bar{\chi}_{u,i}u\|^2_{F_{u,i}}-\|u\|^2_{I_i\setminus B_u}\\
\leq&\|u\|^2_{I_i\cap A_u}+\frac{\ep}{2}\leq\z^2_2+\frac{\ep}{2}.
\end{split}
\eeq

By \eqref{Ch2f4}, \eqref{Ch2f1}, \eqref{Ch2f27}, \eqref{Ch2g5}, and $\de\in(0,r^2_0/4)$, one has
\beq\label{Ch2f40}
\begin{split}
&\sum_{t\in I_i}\big(V(t,u(t))-V(t,\bar{\chi}_{u,i}(t)u(t))\big)\\
=&\sum_{t\in I_i\cap A_u}V(t,u(t))+\sum_{t\in F_{u,i}}\big(V(t,u(t))-V(t,\bar{\chi}_{u,i}(t)u(t))\big)\\
\leq&\frac{1}{8}\|u\|^2_{I_i\cap A_u}+\frac{1}{8}\|u\|^2_{F_{u,i}}+\frac{1}{8}\|\bar{\chi}_{u,i}u\|^2_{F_{u,i}}\\
<&\frac{\z_2^2}{8}+\frac{3}{16}\ep.
\end{split}
\eeq

So, it follows from \eqref{Ch2f39}, \eqref{Ch2f40}, and \eqref{Ch2f1} that
\beqq
\begin{split}
&|f_i(u)-f_i(\bar{\chi}_{u,i}u)|=\bigg|\frac{1}{2}(\|u\|^2_{I_i}-\|\bar{\chi}_{u,i}u\|^2_{I_i})-\sum_{t\in I_i}(V(t,u(t))-V(t,\bar{\chi}_{u,i}(t)u(t)))\bigg|\\
\leq&\frac{\z^2_2}{2}+\frac{\ep}{4}+\frac{\z_2^2}{8}+\frac{3}{16}\ep\leq\z_2^2+\frac{1}{2}\ep<\de.
\end{split}
\eeqq

This, together with the assumption that $u\in(f_i)^{b_{R}+\de}_{b_{R}}$, yields that $\bar{\chi}_{u,i}u\in(f_i)^{b_{R}+2\de}_{b_{R}-\de}$.
From the definition of $\bar{\chi}_{u,i}$ in \eqref{Ch2f24}, it follows that $f(\bar{\chi}_{u,i}u)=f_i(\bar{\chi}_{u,i}u)$. So, $\bar{\chi}_{u,i}u\in(f)^{b_{R}+2\de}_{b_{R}-\de}$. By \eqref{Ch2f41}, one has $\|D_*f(\bar{\chi}_{u,i}u)\|_*=\|D_*f_i(\bar{\chi}_{u,i}u)\|_*\geq\nu$.

So, there exists $Q_{u,i}\in l^2$ with $\|Q_{u,i}\|_*\leq1$ such that
\beqq
D_*f_i(\bar{\chi}_{u,i}u)Q_{u,i}= D_*f(\bar{\chi}_{u,i}u)Q_{u,i}\geq\frac{\nu}{2}.
\eeqq
By applying similar discussions in \eqref{Ch2g8} and \eqref{Ch2f1}, one has
\beqq
|D_*f(\bar{\chi}_{u,i}u)Q_{u,i}-D_*f(u)(\bar{\chi}_{u,i}Q_{u,i})|=|D_*f_i(\bar{\chi}_{u,i}u)Q_{u,i}- D_*f_i(u)(\bar{\chi}_{u,i}Q_{u,i})|\leq4\ep^{1/2}\leq\frac{\nu}{4}.
\eeqq
Hence,
\beqq
\min\{D_*f(u)(\bar{\chi}_{u,i}Q_{u,i}),\ D_*f_i(u)(\bar{\chi}_{u,i}Q_{u,i})\}\geq\frac{\nu}{4}.
\eeqq

For $u\in\mathcal{B}_{r_{b}}^{P}(v)\cap(\cup^k_{i=1}(f_i)^{b_{R}+\de}_{b_{R}})$ and $i\in\mathcal{I}^{+}_2(u)$, if $\|u\|_{I_i\cap A_u}<\z_2$, set
\beqq
\bar{W}_{u,i}:=\frac{1}{2}\bar{\chi}_{u,i}Q_{u,i}.
\eeqq

Set
$$\nu^+:=\min\bigg\{\frac{\nu}{16},\ \frac{\z_2^2}{4}\bigg\},$$
and
\beqq
S_{u,2}:=\left\{
  \begin{array}{ll}
    W_u+\sum_{i\in\mathcal{I}^+_2(u)}\bar{W}_{u,i}, & \hbox{if}\ u\in\mathcal{B}_{r_{b}}^{P}(v)\cap(\cup^k_{i=1}(f_i)^{b_{R}+\de}_{b_{R}}), \\
    0, & \hbox{otherwise.}
  \end{array}
\right.
\eeqq
Thus, for any $u\in\mathcal{B}_{r_{b}}^{P}(v)\cap(\cup^k_{i=1}(f_i)^{b_{R}+\de}_{b_{R}})$,
\beqq
\begin{split}
D_*f(u)S_{u,2}\geq\nu^+;\ D_*f_i(u)S_{u,2}\geq\nu^+,\ \forall i\in\mathcal{I}^+_2(u).
\end{split}
\eeqq
By using similar method in Case (1), one has
\beqq
\lan u,S_{u,2}\ran_{M_l}=\lan u,W_u\ran_{M_l}\geq\frac{1}{k+1}\|u\|^2_{M_l},\ 0\leq l\leq k.
\eeqq

{\bf Case (3)} Consider $u\in\mathcal{B}_{r_{b}}^{P}(v)\cap(\cup^k_{i=1}(f_i)^{b_{L}+\de}_{b_{L}})$. The situation can be investigated by similar approach in Case (2).

Set
\beqq
\mathcal{I}^{+}_3(u):=\{i:\ 1\leq i\leq k,\ u\in(f_i)^{b_{L}+\de}_{b_{L}}\},\ \z_3:=\frac{\de^{1/2}}{2},
\eeqq
and choose any $i\in\mathcal{I}^{+}_3(u)$.

Now, we consider two different situations
\beqq
\mbox{(v)}\ \|u\|_{I_i\cap A_u}\geq\z_3;\ \mbox{(vi)}\ \|u\|_{I_i\cap A_u}<\z_3.
\eeqq
It follows from \eqref{Ch2f1}, \eqref{Ch2f27}, and Lemma \ref{Ch2diff2} that
\beqq
D_*f(u)W_u\geq\frac{\z^2_3}{2}-2\ep\geq\frac{1}{4}\z_3^2,\ D_*f_i(u)W_u\geq\frac{\z^2_3}{2}-2\ep\geq\frac{1}{4}\z_3^2.
\eeqq
For $u\in\mathcal{B}_{r_{b}}^{P}(v)\cap(\cup^k_{i=1}(f_i)^{b_{L}+\de}_{b_{L}})$ and $i\in\mathcal{I}^{+}_3(u)$, if $\|u\|_{I_i\cap A_u}\geq\z_3$,
set
\beqq
\bar{W}_{u,i}:=0.
\eeqq

Now, we study Case (vi).

First, we show that $\bar{\chi}_{u,i}u\in(f_i)^{b_{L}+2\de}_{b_{L}-\de}.$

By \eqref{Ch2f10}, \eqref{Ch2g5}, and the definition of $\bar{\chi}_{u,i}$, one has
\beq\label{Ch2g6}
\begin{split}
&\|u\|^2_{I_i}-\|\bar{\chi}_{u,i}u\|^2_{I_i}=\|u\|^2_{I_i\cap A_u}+\|u\|^2_{F_{u,i}}+\|u\|^2_{I_i\setminus B_u}-\|\bar{\chi}_{u,i}u\|^2_{F_{u,i}}-\|u\|^2_{I_i\setminus B_u}\\
\leq&\|u\|^2_{I_i\cap A_u}+\frac{\ep}{2}\leq\z^2_3+\frac{\ep}{2}.
\end{split}
\eeq

By \eqref{Ch2f4}, \eqref{Ch2f1}, \eqref{Ch2f27}, \eqref{Ch2g5}, and $\de\in(0,r^2_0/4)$, one has
\beq\label{Ch2g7}
\begin{split}
&\sum_{t\in I_i}\big(V(t,u(t))-V(t,\bar{\chi}_{u,i}(t)u(t))\big)\\
=&\sum_{t\in I_i\cap A_u}V(t,u(t))+\sum_{t\in F_{u,i}}\big(V(t,u(t))-V(t,\bar{\chi}_{u,i}(t)u(t))\big)\\
\leq&\frac{1}{8}\|u\|^2_{I_i\cap A_u}+\frac{1}{8}\|u\|^2_{F_{u,i}}+\frac{1}{8}\|\bar{\chi}_{u,i}u\|^2_{F_{u,i}}\\
<&\frac{\z_3^2}{8}+\frac{3}{16}\ep.
\end{split}
\eeq

So, it follows from \eqref{Ch2f1}, \eqref{Ch2g6}, and \eqref{Ch2g7} that
\beqq
\begin{split}
&|f_i(u)-f_i(\bar{\chi}_{u,i}u)|=\bigg|\frac{1}{2}(\|u\|^2_{I_i}-\|\bar{\chi}_{u,i}u\|^2_{I_i})-\sum_{t\in I_i}(V(t,u(t))-V(t,\bar{\chi}_{u,i}(t)u(t)))\bigg|\\
\leq&\frac{\z^2_3}{2}+\frac{\ep}{4}+\frac{\z_2^2}{8}+\frac{3}{16}\ep\leq\z_3^2+\frac{\ep}{2}<\de.
\end{split}
\eeqq

This, together with the assumption that $u\in(f_i)^{b_{L}+\de}_{b_{L}}$, yields that $\bar{\chi}_{u,i}u\in(f_i)^{b_{L}+2\de}_{b_{L}-\de}$.
From the definition of $\bar{\chi}_{u,i}$ in \eqref{Ch2f24}, it follows that $f(\bar{\chi}_{u,i}u)=f_i(\bar{\chi}_{u,i}u)$. So, $\bar{\chi}_{u,i}u\in(f)^{b_{L}+2\de}_{b_{L}-\de}$. By \eqref{Ch2f41}, one has $\|D_*f(\bar{\chi}_{u,i}u)\|_*=\|D_*f_i(\bar{\chi}_{u,i}u)\|_*\geq\nu$.

So, there exists $Q_{u,i}\in l^2$ with $\|Q_{u,i}\|_*\leq1$ such that
\beqq
 D_*f_i(\bar{\chi}_{u,i}u)Q_{u,i}=D_*f(\bar{\chi}_{u,i}u)Q_{u,i}\geq\frac{\nu}{2}.
\eeqq
By applying similar discussions in Case (1) and \eqref{Ch2f1}, one has
\beqq
|D_*f(\bar{\chi}_{u,i}u)Q_{u,i}-D_*f(u)(\bar{\chi}_{u,i}Q_{u,i})|=|D_*f_i(\bar{\chi}_{u,i}u)Q_{u,i}- D_*f_i(u)(\bar{\chi}_{u,i}Q_{u,i})|\leq4\ep^{1/2}\leq\frac{\nu}{4}.
\eeqq
Hence,
\beqq
\min\{D_*f(u)(\bar{\chi}_{u,i}Q_{u,i}),\ D_*f_i(u)(\bar{\chi}_{u,i}Q_{u,i})\}\geq\frac{\nu}{4}.
\eeqq

For $u\in\mathcal{B}_{r_{b}}^{P}(v)\cap(\cup^k_{i=1}(f_i)^{b_{L}+\de}_{b_{L}})$ and $i\in\mathcal{I}^{+}_3(u)$, if $\|u\|_{I_i\cap A_u}<\z_3$, set
\beqq
\bar{W}_{u,i}:=\frac{1}{2}\bar{\chi}_{u,i}Q_{u,i}.
\eeqq

Set
$$\nu^-:=\min\bigg\{\frac{\nu}{16},\ \frac{\z_3^2}{4}\bigg\},$$
and
\beqq
S_{u,3}:=\left\{
  \begin{array}{ll}
    W_u+\sum_{i\in\mathcal{I}^+_3(u)}\bar{W}_{u,i}, & \hbox{if}\ u\in\mathcal{B}_{r_{b}}^{P}(v)\cap(\cup^k_{i=1}(f_i)^{b_{L}+\de}_{b_{L}}), \\
    0, & \hbox{otherwise.}
  \end{array}
\right.
\eeqq
Thus, for any $u\in\mathcal{B}_{r_{b}}^{P}(v)\cap(\cup^k_{i=1}(f_i)^{b_{L}+\de}_{b_{L}})$,
\beqq
\begin{split}
D_*f(u)S_{u,3}\geq\nu^-;\ D_*f_i(u)S_{u,3}\geq\nu^-,\ \forall i\in\mathcal{I}^+_3(u).
\end{split}
\eeqq
By using similar method in Case (1), one has
\beqq
\lan u,S_{u,3}\ran_{M_l}=\lan u,W_u\ran_{M_l}\geq\frac{1}{k+1}\|u\|^2_{M_l},\ 0\leq l\leq k.
\eeqq

{\bf Case (4)} Now, we study the situation $u\in\mathcal{B}_{r_{s}}^{P}(v)$.

There are two different situations: (vii) $\max_{0\leq l\leq k}\|u\|_{M_l}\geq 4\ep,$  (viii) $\max_{0\leq l\leq k}\|u\|_{M_l}<4\ep$.

For Case (vii), take $j$, $0\leq j\leq k$, such that $\|u\|^2_{M_j}\geq4\ep$. It follows from Lemma \ref{Ch2diff2} and \eqref{Ch2f28} that
\beqq
D_*f(u)W_u\geq\frac{1}{2}(\|u\|^2_{A_{u,j}}-\ep)-\frac{\ep}{2}\geq\ep.
\eeqq
In Case (vii), set
\beqq
S_{u,4}:=W_u.
\eeqq

Next, it is to show that if $\mathcal{C}\cap\mathcal{B}^P_{r_s}(v)=\emptyset$, then for any $u\in\mathcal{B}^P_{r_s}(v)$ and $u$ is in Case (viii), there is $\mu'_k>0$ such that
$\|D_*f(u)\|_{*}\geq\mu'_k$.\ By contradiction, suppose that there is $\{u_m\}^{\infty}_{m=1}\subset\mathcal{B}^P_{r_s}(v)$ such that $\|D_*f(u_m)\|_{*}\to0$.\
It is easy to obtain that $\mathcal{B}^P_{r_s}(v)$ is a bounded set. By the conclusion of Step 1 in the proof of (g3) in Lemma \ref{Ch2pathg}, one has that $f(\mathcal{B}^P_{r_s}(v))$ is bounded, which implies that
$\{u_m\}^{\infty}_{m=1}$ is a PS sequence. Since $u_m\in\mathcal{B}^P_{r_s}(v)$, one has that $\|u_m-v(\cdot-p_1)\|_{I_1}<r_s$,\ $\|u_m-v(\cdot-p_k)\|_{I_k}<r_s$. This, together with $r_s<2r<r_0/4<D_0/16$, yields that the sequence $\{u_m\}^{\infty}_{m=1}$ satisfies the assumptions of Lemma \ref{Ch2b3}. It follows from Lemma \ref{Ch2b3} that there is a convergent subsequence. Since $\{u_m\}^{\infty}_{m=1}\subset\mathcal{B}^P_{r_s}(v)$, it is evident that the limit points of the PS sequence is also belong to the set $\mathcal{B}^P_{r_s}(v)$. It follows from the assumption $\mathcal{C}\cap\mathcal{B}^P_{r_s}(v)=\emptyset$ that the limit point of this PS sequence is the zero element in $l^2$. Hence, $\|v(-p_i)\|_{I_i}<r_s$, $1\leq i\leq k$, yielding that,
$\|v\|_{\{I_i-p_i\}}<r_s$. By \eqref{Ch2f3},\ $r\in(0,\min\{\frac{r_0}{8},\frac{\sqrt{3}\|v\|_{*}}{4}\})\setminus D^{\kappa^*}$,\ \eqref{Ch2f1} and \eqref{Ch2f27}, one has that $\|v\|_{\{I_i-p_i\}}>r_s$. This is a contradiction.

Hence, if $\mathcal{C}\cap\mathcal{B}^P_{r_s}(v)=\emptyset$, then for any $u\in\mathcal{B}^P_{r_s}(v)$ and $u$ is in Case (viii), there is $Q_u\in l^2$\ with $\|Q_u\|_{*}\leq1$ such that
\beqq
 D_*f(u)Q_u\geq\frac{\bar{\mu}'_k}{2}.
\eeqq
For Case (viii), set
\beqq
S_{u,4}:=Q_u.
\eeqq
If $u\not\in\mathcal{B}^P_{r_s}(v)$, set $S_{u,4}:=0$ and
\beqq
\mu_k:=\min\bigg\{\frac{\ep}{2},\ \frac{\bar{\mu}_k}{4}\bigg\}.
\eeqq
Hence, for any $u\in\mathcal{B}^P_{r_s}(v)$,
\beqq
 D_*f(u)S_{u,4}\geq2\mu_k,
\eeqq
and for any $u$ in Case (4), one has that
\beqq
\lan u,S_{u,4}\ran_{M_l}=\lan u,W_u\ran_{M_l}\geq\frac{1}{1+k}\|u\|^2_{M_l},\ 0\leq l\leq k.
\eeqq

For $u\in l^2$, set
\beqq
S_u:=\sum^4_{i=1}S_{u,i}.
\eeqq

By applying the method of constructing pseudo-gradient vector, we could give the vector field satisfying the requirements of the proposition, please refer to the proof of Theorem 5.2.2 in \cite{Xuan}. We give a sketch of the construction.

In the above construction, the constant $\mu$ could be chosen such that it is independent on
the choice of $u$ in the cases of (II). In the Cases (III) and (IV), there are small
positive constants such that the inequalities in (III) and (IV) are bigger than these positive constants.
Hence, for any $u$, we could define a constant vector on a sufficiently small neighborhood of $u$ such that
the requirements in (I)--(V) are satisfied. For any $u\in l^2$, denote $w_u$ as this constant vector and this neighborhood as $N_{u}$. The family
$\{N_u\}_{u\in l^2}$ consists an open covering of $l^2$. Since $l^2$ is a metric space, it is paracompact.
So, there exists a locally finite refinement $\{N_{u_{\alpha}}\}_{\al\in I}$, where $I$ is the index set.
Consider the following distance function
\beqq
d_{\al}(x):=\mbox{dist}(x,l^2\setminus N_{u_{\al}}),\ \al\in I.
\eeqq
Set
\beqq
d(x):=\sum_{\al\in I}d_{\al}(x),\ x\in l^2.
\eeqq
We can show that $d(x)$ is well-defined and locally Lipschitz continuous function on $l^2$.
For any $u\in l^2$, since the subcover is locally finite, there is a finite number $m$ such that
$d(x)=\sum^{m}_{\al=1}d_{\al}(x)>0$.
Set
\beqq
\mu_{\al}(x):=\frac{d_{\al}(x)}{d(x)},\ \forall x\in l^2,\ \al\in I.
\eeqq
It is evident that
\beqq
\mu_{\al}(x)=0,\ \forall x\in l^2\setminus N_{u_{\al}},\ \al\in I,
\eeqq
and
\beqq
0\leq\mu_{\al}\leq1,\ \sum_{\al\in I}\mu_{\al}(x)=1,\ \forall x\in l^2.
\eeqq
It is evident that $\mu_{\al}(x)$ is locally Lipschitz continuous.
Denote
\beqq
W(x):=\sum_{\al\in I}\mu_{\al}(x)w_{\al_u},\ \forall x\in l^2.
\eeqq
It is easy to show that the vector field $W$ is locally Lipschitz continuous and satisfies the requirements.

By the choice of $r_0$ in \eqref{Ch2f4}, the definition of $W_u$ in \eqref{Ch2f36}, and the construction of the vector field above, it is evident that
$\|W(u)\|_{I_i}\leq 2$,\ $1\leq i\leq k$. This completes the whole proof.
\end{proof}
\bigskip

\section*{Acknowledgment}
We would like to thank Professor Yuming Shi for her encouragement, comments, and providing many useful references.

This research was partially supported by the NNSF of China (Grants 11071143, 11101246, and 11301304).

\end{document}